\documentclass[11pt]{article}
\usepackage[normalem]{ulem}

\usepackage{amsmath,amsthm,amssymb,enumerate,amsfonts, graphicx,array,url,framed,fancyvrb,color,verbatim,subcaption,setspace, mathrsfs,bbm}
\usepackage{amsmath,amsthm,amssymb}
\usepackage{graphicx}
\usepackage{authblk}
\usepackage{appendix}
\usepackage[font={small,rm}]{caption}

\font\tencmmib=cmmib10 \skewchar\tencmmib '60
\newfam\cmmibfam
\textfont\cmmibfam=\tencmmib

\def\lessim{\ \lower4pt\hbox{$
		\buildrel{\displaystyle <}\over\sim$}\ }
\def\gessim{\ \lower4pt\hbox{$\buildrel{\displaystyle >}
		\over\sim$}\ }

\newcommand{\R}{\mathbb{R}}
\newcommand{\RR}{\mathcal{R}}

\newcommand{\E}{\mathbb{E}}

\newcommand{\p}{\mathbb{P}}

\renewcommand{\epsilon}{\varepsilon}
\newcommand{\<}{\langle}
\renewcommand{\>}{\rangle}

\newcommand{\s}{\sigma}

\newcommand{\e}{\mathbb{E}}
\newcommand{\la}{\<}
\newcommand{\ra}{\>}
\newcommand{\cf}{C\!F}

\DeclareMathOperator{\MMSE}{MMSE}
\DeclareMathOperator{\DMSE}{DMSE}

\newtheorem{theorem}{Theorem}[section]
\newtheorem{lemma}{\bf Lemma}[section]
\newtheorem{prop}{\bf Proposition}[section]
\newtheorem{remark}{\bf Remark}[section]

\newtheorem{definition}{\bf Definition}[section]
\numberwithin{equation}{section}

\begin{document}
	
\title{Phase transition in random tensors with multiple independent spikes}
\author{Wei-Kuo Chen, Madeline Handschy, Gilad Lerman}
\date{}
\maketitle

\begin{abstract}

Consider a spiked random tensor obtained as a mixture of two components: noise in the form of a symmetric Gaussian $p$-tensor for $p\geq 3$ and signal in the form of a symmetric low-rank random tensor. 
The latter is defined as a linear combination of $k$ independent symmetric rank-one random tensors, referred to as spikes, with weights referred to as signal-to-noise ratios (SNRs).
The  entries of the vectors that determine the spikes are  i.i.d. sampled from general probability distributions supported on bounded subsets of $\mathbb{R}$.
This work focuses on the problem of detecting the presence of these spikes, and establishes the phase transition of this detection problem for any fixed $k \geq 1$. In particular, it shows that for a set of relatively low SNRs it is impossible to distinguish between the spiked and non-spiked Gaussian tensors. Furthermore, in the interior of the complement of this set, where at least one of the $k$ SNRs is relatively high, these two tensors are distinguishable by the likelihood ratio test. In addition, when the total number of low-rank components, $k$,  of the $p$-tensor of size $N$ grows in the order $o(N^{(p-2)/4})$ as  $N$  tends to infinity, the problem exhibits an analogous phase transition. This theory for spike detection is also shown to imply that recovery of the spikes by the minimum mean square error exhibits the same phase transition. The main methods used in this work arise from the study of mean field spin glass models, where  the phase transition thresholds are identified as the critical inverse temperatures distinguishing the high and low-temperature regimes of the free energies. 
In particular, our result formulates the first full characterization of the high temperature regime for  vector-valued spin glass models with independent coordinates.

\end{abstract}

\tableofcontents
\section{Introduction} \label{sec:intro}

This work studies the detection and recovery of a low-rank component in a particular random tensor and characterizes
their corresponding phase transitions.
In order to motivate this problem, we first discuss a simpler and widely-studied question: When can principal component analysis (PCA) detect and recover low-rank linear structures in noisy data?
While detection only requires determining the presence or absence of low-rank structure, the task of recovery aims to reveal the concealed low-rank structure. 

One common setting for addressing this question assumes data points $y_1, \dots, y_L \in \R^N$ drawn independently from the multivariate normal distribution $\mathcal{N}(0,I+\beta  uu^T)$, where $I$ is the $N$-dimensional identity matrix, which generates spherically symmetric Gaussian noise, $u$ is a unit column vector in $\mathbb{R}^N$, which generates a rank-one signal, and $\beta>0$ is the signal-to-noise ratio (SNR). Under this model, the observations $y_i$, $i=1, \dots, L$, take the form $y_i = x_i + \epsilon_i$, where $x_i$ is proportional to the signal $u$ with signal-to-noise ratio $\beta$, and $\epsilon_i$ is the Gaussian noise. The question is then whether or not it is possible to apply PCA to detect the presence of the signal $u$ when given the data points $y_1, \dots, y_L$ with different choices of the SNR parameter $\beta$. The earlier result for this problem traces back to the fundamental work of Johnstone \cite{Johnstone}.

Assume that $N/L \rightarrow \gamma < 1$ as $L \rightarrow \infty$.
When $\beta=0,$ the Marchenko-Pastur distribution \cite{Mar_Pas} describes the limiting distribution of the eigenvalues of the sample covariance matrix. The well-known Baik-Ben Arous-P\'ech\'e phase transition \cite{BBP, baik2} states that when $\beta \leq \sqrt{\gamma}$, the eigenvalues of this matrix still follow the Marchenko-Pastur distribution and thus detection of the low-rank sample is impossible by PCA. In contrast, when $\beta > \sqrt{\gamma}$, the largest eigenvalue of this matrix stays away from the typical location of the Marchenko-Pastur distribution and PCA can detect the presence of the signal. This phase transition of spike detection is extended in Paul \cite{paul2007} to spike recovery by PCA. More precisely, \cite{paul2007} shows that when $\beta > \sqrt{\gamma}$, there is a non-trivial asymptotic correlation between the top eigenvector of the sample covariance and $u$ and thus one can approximately recover $u$ by PCA. Moreover, when $\beta \leq \sqrt{\gamma}$ this  asymptotic correlation is zero and PCA cannot recover $u$. Extension of detection and recovery to the case where $\gamma \geq 1$ is also established in \cite{paul2007}.

Another common setting for studying the detection problem using PCA assumes a random matrix of the form $T = W + \beta N^{-1/2} uu^T$, where $W$ is an $N \times N$ Gaussian Wigner matrix\footnote{$W$ is a symmetric matrix with independent $W_{ij}\thicksim \mathcal{N}(0,1/2)$ for $1 \leq i<j \leq N$ and $W_{ii}\thicksim \mathcal{N}(0,1)$ for $1 \leq i \leq N$.} and $u$ is an $N$-dimensional random vector with i.i.d. entries sampled from a bounded distribution on $\R$. The parameter $\beta$ is the SNR. We refer to the rank one component, $uu^T$ as a spike and to $T$ as a spiked random matrix.
The problem is to detect the presence of the spike in $T$, or equivalently, to distinguish between $T$ and $W$.
This detection problem exhibits a phase transition similar to that of the previous setting, see F\'eral-P\'ech\'e \cite{PecheD}, P\'ech\'e \cite{Peche}, and Benaych-Georges-Nadakuditi \cite{benaych2011, benaych2012}. When the SNR is below  a certain critical threshold, the eigenvalue distribution of $T$ follows Wigner's semi-circle law and it is thus impossible to distinguish between $T$ and $W$. Once the value of $\beta$ exceeds  this  critical threshold, the largest eigenvalue jumps away from the typical location of the Wigner semi-circle law and the top eigenvector nontrivially correlates with the signal.  Consequently, in this case, one can detect and approximately recover the signal by PCA. Recent studies of phase transitions in detection and recovery of low-rank signals in random matrices include Lelarge-Miolane \cite{LM+16}, Miolane \cite{M+17,miolane2018}, Montanari-Reichman-Zeitouni \cite{MRZ17}, Montanari-Richard \cite{MRZ17,MR+16}, Onatski-Moreira-Hallim \cite{OMH13}, and Perry-Wein-Bandeira-Moitra \cite{PWBM16}.

The latter setting of low-rank detection in spiked random matrices has a natural higher-order generalization to spiked random tensors.
This generalization considers the spiked symmetric random $p$-tensor $$T_k=W+\frac{1}{N^{(p-1)/2}}\sum_{r=1}^k\beta_ru(r)^{\otimes p}.$$ The first component, $W$, is the symmetric Gaussian  $p$-tensor of size $N^p$, formally defined  in Section~\ref{sec2.1}. The second component is the signal, which is a linear combination of the spikes $u(1)^{\otimes p}$, $\ldots$,  $u(k)^{\otimes p}$. Here, $u(1),\ldots,u(k)$ are $N$-dimensional vectors whose entries are i.i.d. sampled from probability measures $\mu_1,\ldots,\mu_k$ supported on bounded subsets of the real line. We refer to $\hat\beta = (\beta_1, \dots, \beta_k)$ as the vector of SNRs. The detection problem under this setting asks whether identification of the low-rank signal $\sum_{r=1}^k\beta_ru(r)^{\otimes p}$ in the tensor $T_k$ is possible for a given vector $\bar \beta$.
The recovery problem seeks to recover, if possible, the low-rank signal  for given values of $\bar \beta$. Answering these question of whether the spike is detectable or recoverable requires characterization of the phase transitions in $\bar \beta$ of the detection and recovery problems.

We remark that the generalized tensor setting is significantly more challenging than the above setting of detecting and recovering rank-one structure in matrices. The former setting involved the best rank-one approximation by PCA.
However, for tensors, basic relevant notions, such as rank and best low-rank approximation, are not obvious, see Kolda-Bader \cite{kolda2009tensor}. Furthermore, many common algorithms for computing these and related notions are NP-hard, see Hillar-Lim \cite{HillarL13}.
In this work, we study low-rank tensor detection and recovery by common theoretical tests and estimators, which are hard to compute.  We leave the analysis of tractable procedures to future work.
Following Chen \cite{Chen17}, El Alaoui-Krzkala-Jordan \cite{AKJ+17}, Montanari-Reichman-Zeitouni \cite{MRZ17}, Montanari-Richard \cite{MR+14}, and Perry-Wein-Bandeira \cite{PWB17}, we say that spike detection is impossible if the total variation distance between $W$ and $T_k$ vanishes when $N$ tends to infinity. In other words, any statistical test fails to distinguish $W$ and $T_k$ (see Section~\ref{sec2.1}). On the other hand, we say that detection is possible if this distance is one in the limit.
This means that asymptotically one can find a statistical test, in particular, the likelihood ratio test,  that distinguishes between $W$ and $T_k$ (see Section \ref{sec2.1}).
For recovery, we follow \cite{LMLKZ+17} and use the minimum mean square error (MMSE) and its corresponding estimator.

Many recent works, which are reviewed in Section \ref{prev_res}, have studied detection and recovery under the spiked random tensor model. Nevertheless, the optimal phase transition for low-rank detection in spiked random tensors has not yet been established. This paper aims to close this gap. Our main result states that there exist critical thresholds $\beta_{1,c}, \dots, \beta_{k,c}$ and a set of the form $\bar {\mathcal{R}} = (0,\beta_{1,c}] \times \cdots \times (0,\beta_{k,c}]$ such that detection is impossible if $\bar \beta = (\beta_1, \dots, \beta_k)$ lies strictly in the interior of the set $\bar \RR$.  Furthermore, it is possible to detect the spike via the  likelihood ratio test when $\bar \beta \not \in \bar {\mathcal{R}}$. In other words, detection is possible only when at least one of $\beta_1,\ldots,\beta_k$ exceeds its critical threshold;
whereas, if $\beta_1,\ldots,\beta_k$ are all smaller than their critical thresholds, one cannot detect the spike. Our result also allows the total number of spikes to grow with $N$. In particular, if $\mu_1=\cdots=\mu_k$ and $k=o(N^{(p-2)/4})$, then similar statements hold. A byproduct of these developments is a new proof for a recent result on the recovery problem by Lesieur-Miolane-Lelarge-Krzakala-Zdeborov\'a \cite{LMLKZ+17} when assuming the same setting of the present paper. In essence, their result states that $\beta_{1,c},\ldots,\beta_{k,c}$ are the critical thresholds for the MMSE recovery problem.

Our approach is based on methodologies from the study of mean-field spin glass models. Roughly speaking, spin glasses are spin systems that exhibit both quenched disorder and frustration. That is, the interactions between sites are disordered and spin constraints cannot be simultaneously satisfied. These two features are commonly shared by many problems that involve randomized combinatorial optimization, see M\'ezard-Montanari \cite{Montanari_spin_book2012} and Montanari-Sen \cite{Montanar_Sen_STOC2016}.
The book, M\'ezard-Parisi-Virasoro \cite{MPV}, reviews the area of spin glasses from the point of view of physicists, whereas mathematical treatments of the subject appear in Talagrand \cite{Talbook1,Talbook2} and Panchenko \cite{Pan}.

Mean-field spin glasses are related to the detection problem by the following key observation: The total variation distance between $W$ and $T_k$ can be represented as an integral of the distribution function of the so-called free energy of the re-centered pure $p$-spin model with vector-valued spin configurations (see \eqref{fg} and Lemma \ref{dtv_free} below). From this observation, to study the detection problem, we need to understand the full high-temperature regime of this model and investigate a sharp upper bound on the fluctuation of the free energy for all values of the SNR vector $\bar \beta$. Our results reveal that $\bar {\mathcal{R}}$ is indeed the high-temperature regime of the free energy and its fluctuation is up to the order $N^{-(p/2+1)}$ when $\bar \beta$ lies in the interior of $\bar {\mathcal{R}}$ and is of order 1 when $\bar \beta$ lies in the complement of $\bar {\mathcal{R}}$. These allow us to completely characterize the phase transition of the detection problem. We mention that while our study of the high-temperature behavior of the free energy is mainly used to derive results for spike detection, it is also of independent interest in the field of spin glasses. Indeed, our result (see Theorem \ref{thm6}) gives the first full characterization of the high temperature regime for  vector-valued spin glass models with independent coordinates.


\smallskip
\smallskip

\noindent \textbf{Acknowledgement:} The authors thank the anonymous referee for providing many useful suggestions regarding the presentation of the paper.
The research of W.-K. Chen is partly supported by NSF grants DMS-16-42207 and DMS-17-52184, and Hong Kong Research Grants Council GRF-14302515.
He thanks the National Center for Theoretical Sciences and Academia Sinica in Taipei for the hospitality during his visit in June and July 2018, where part of the results and writings were completed.
In addition, he is grateful to Lenka Zdeborov\'{a} for many illuminating discussions. The research of G. Lerman is partially supported by NSF grants DMS-14-18386 and DMS-18-21266.

\section{Main Results}\label{sec2}
This section states the main results of this paper and provides the necessary mathematical background. Additionally, it reviews prior results and describes the structure of the rest of the paper, particularly the structure of the proofs of Theorems \ref{thm1} - \ref{maintheorem2}. Section \ref{sec2.1} defines the necessary terminology, especially, the distinguishability of two random tensors. Section \ref{sec2.2} describes our main results for the detection problem in the case of a single spike. In particular, it introduces an auxiliary function that characterizes the high-temperature regime and allows one to simulate the critical SNR. Using this function, we demonstrate numerical simulations of the critical SNR for the sparse Rademacher prior. Section \ref{sec2.3} states our main results for the detection problem in the case of multiple spikes.
Section \ref{sec2.4} mentions a result for recovery by MMSE that is later obtained from our results for spike detection.
Section \ref{prev_res} surveys recent related results. Finally, Section \ref{structure} describes the organization of the proofs of the main results.

\subsection{Settings and Definitions} \label{sec2.1} Let $p\geq 2$ be an integer. For any integer $N\geq 1$, denote by $\Omega_N$ the set of all real-valued $p$-tensors $Y = (Y_{i_1,\dots,i_p})_{1 \leq i_1, \ldots, i_p \leq N}$ equipped with the Borel $\sigma$-field. The inner product of two $p$-tensors is
\[
\<Y,Y'\> = \sum_{1 \leq i_1,\dots,i_p \leq N} Y_{i_1,\dots,i_p}Y'_{i_1,\dots,i_p}.
\]
Given a vector $u=(u_1,\ldots,u_N) \in \mathbb{R}^N$, we form a rank-one $p$-tensor using the outer product by
\[
(u^{\otimes p})_{i_1,\dots,i_p} = u_{i_1}\cdots u_{i_p}, \quad \forall 1\leq i_1,\ldots,i_p\leq N.
\]
Given $Y \in \Omega_N$ and a permutation $\pi$ of the set $\{1,2,\dots,p\}$, define $Y^{\pi}$ by
\[
Y^{\pi}_{i_1,\dots,i_p} = Y_{\pi(i_1),\dots,\pi(i_p)}.
\]
A $p$-tensor is said to be symmetric if $Y^{\pi}_{i_1,\dots,i_p} = Y_{i_1,\dots,i_p}$ for all corresponding indices and permutations.
Throughout the rest of the paper, we assume that $Y$ is a random $p$-tensor and all entries in $Y$ are i.i.d. standard Gaussian.
The symmetric Gaussian $p$-tensor of size $N^p$ is obtained by the averaging over all permutations in the symmetric group of $N$ letters:
\[
W = \frac{1}{p!} \sum_{\pi } Y^{\pi}.
\] 
In the case $p=2$, $W$ is  the Gaussian Wigner matrix.

{Next, we} define the notion of distinguishability and indistinguishability between two random $p$-tensors in terms of the total variation distance. For any two random $p$-tensors $U,V$, denote by $d_{TV}(U,V)$ the total variation distance between $U$ and $V$, that is,
\[
d_{TV}(U,V) = \sup_{A} |P(U \in A) - P(V \in A)|,
\]
where the supremum is taken over all sets $A$ in the Borel $\sigma$-algebra generated by symmetric $p$-tensors.

\begin{definition}\label{def1}
Let $U_N,V_N$ be two sequences of random $p$-tensors. We say that they are distinguishable if
$$\lim_{N\rightarrow\infty}d_{TV}(U_N,V_N)=1$$
and  are indistinguishable if
$$
\lim_{N\rightarrow\infty}d_{TV}(U_N,V_N)=0.
$$
\end{definition}

Distinguishability of $U_N$ and $V_N$ means that there exists a sequence of measurable subsets $A_N$ of $\Omega_N$ such that $\lim_{N\rightarrow\infty}\p(U_N\in A_N)=1$ and $\lim_{N\rightarrow\infty}\p(V_N\in A_N)=0.$ From this, if we consider a statistical test $S_N:\Omega_N\rightarrow \{0,1\}$ defined by $S_N(w)=0$ for $w\in A_N$ and $S_N(w)=1$ for $w\notin A_N,$ then as $N$ approaches infinity, the sum of type I and type II errors approaches zero:
\begin{align}
\label{eq:limiting_types}
\lim_{N\rightarrow\infty}\bigl(\p(S_N(U_N)=1)+\p(S_N(V_N)=0)\bigr)=\lim_{N\rightarrow\infty}\bigl(\p(U_N\notin A_N)+\p(V_N\in A_N)\bigr)=0.
\end{align}
This means that one can statistically distinguish $U_N$ and $V_N$ by the test $S_N.$  Furthermore, if $U_N$ and $V_N$ have nonvanishing densities $f_{U_N}$ and $f_{V_N}$, the well-known formula
\begin{align*}
d_{TV}(U_N,V_N)= \int_{f_{U_N} \geq f_{V_N}} (f_{U_N} - f_{V_N}) dw
\end{align*}
implies that
\begin{align*}
d_{TV}(U_N,V_N)=\p\bigl(U_N\in A_N\bigr)-\p\bigl(V_N\in A_N\bigr)
\end{align*}
for $$A_N:=\Bigl\{w\in \Omega_N \ \Big| \ \frac{f_{U_N}(w)}{f_{V_N}(w)}\geq 1\Bigr\}.$$ Therefore,  one can naturally use the likelihood ratio test to distinguish $U_N$ and $V_N$. In contrast, when $U_N$ and $V_N$ are indistinguishable, any statistical test is powerless as in this case the total error approaches one as $N$ tends to infinity.

\begin{remark}
	\rm The setting of spiked matrices (see, e.g., \cite{AKJ+17,WEM19})
	considers a weaker notion of distinguishability that requires the  limiting total error, which appears on the left hand side of \eqref{eq:limiting_types},
	to be less than 1. In our spiked tensor model with $p>2$, the limiting total error converges to either zero or one and this weaker notion of distinguishability coincides with ours.
\end{remark}

\subsection{Main Results for Detection of a Single Spike \label{sec2.2}}

Let $\Lambda$ be a bounded subset of $\mathbb{R}$ and $\mu$ be a probability measure on the Borel $\sigma$-field of $\Lambda$. Assume that $u_1,\ldots,u_N$ are i.i.d. samplings from $\mu$ and are independent of $W$. Denote $u=(u_1,\ldots,u_N).$ We refer to the random variable $u$ as the prior. Consider the spiked random $p$-tensor $T$ defined by
\begin{align}\label{spike}
T=W+\frac{\beta}{N^{(p-1)/2}} u^{\otimes p}.
\end{align}
We say that detection of the spike $u^{\otimes p}$ in $T$ is possible if $W$ and $T$ are distinguishable and detection is impossible if they are indistinguishable in the sense of Definition \ref{def1}.
Note that if $\int a\mu(da)\neq 0$, one can immediately detect the spike by noting that $Y_{i_1,\ldots,i_p}$ are i.i.d. standard Gaussian and using the strong law of large number. Indeed,
\[
\frac{1}{N^{(p+1)/2}} \sum_{i_1,\dots,i_p =1}^N W_{i_1, \dots, i_p} =\frac{1}{N^{(p+1)/2}}\sum_{i_1,\ldots,i_p=1}^NY_{i_1,\ldots,i_p}\rightarrow 0,
\]
while
$$
\frac{1}{N^{(p+1)/2}}\sum_{i_1,\ldots,i_p=1}^NT_{i_1,\ldots,i_p}=\frac{1}{N^{(p+1)/2}}\sum_{i_1,\ldots,i_p=1}^NY_{i_1,\ldots,i_p}+\beta\Bigl(\frac{\sum_{i=1}^Nu_i}{N}\Bigr)^p\rightarrow\beta\Bigl(\int a\mu(da)\Bigr)^p.
$$
We can thus restrict our discussion to the case when $\mu$ is centered, that is, when $\int_{\R} a \mu(da) = 0$. Our first result on spike detection is formulated as follows.

\begin{theorem}\label{thm1}
Assume that $\mu$ is centered.
For any $p\geq 3,$ there exists a constant $\beta_c>0$ such that
\begin{enumerate}
	\item [$(i)$] if $0<\beta<\beta_c$, then detection is impossible;
	\item [$(ii)$] if $\beta > \beta_c$, then detection is possible.
\end{enumerate}	
\end{theorem}

In other words, $\beta_c$ is the critical threshold that describes the phase transition of the detection problem.  As we explained in Section \ref{sec2.1}, when detection is possible, one can use the likelihood ratio test, which uses the ratio of densities $f_{T}(w)/f_W(w)$, to distinguish between $W$ and $T$. In Lemma~\ref{dtv_free} below, we relate this ratio to the free energy of the pure $p$-spin mean field spin glass model.

The precise value of $\beta_c$ can be determined as follows. Let $$
\xi(s)=s^p$$ and
\begin{align}
\label{q1}
v_{*}=\int a^2\mu(da).
\end{align}
For $a\in \mathbb{R}$ and $t>0$, consider the geometric Brownian motion $$Z(a,t) = \exp\Bigl(a B_t-\frac{a^2t}{2}\Bigr),$$
where $B_t$ is a standard Brownian motion.
For $b \geq 0$, define an auxiliary function $\Gamma_b(v)$ on $[0, \infty)$ by
\begin{align}\label{add:eq1}
\Gamma_b(v)&=\int_0^{v}\xi''(s)(\gamma_b(s)-s)ds,
\end{align}
where for $s\geq 0,$
\begin{align}
\gamma_{b}(s):= \mathbb{E}\left[\frac{\left(\int  a Z(a,b^2 \xi'(s))\mu(da)\right)^2}{\int  Z(a,b^2 \xi'(s))\mu(da)}\right]. \label{auxfun}
\end{align}
The critical value $\beta_c$ in Theorem \ref{thm1} can be calculated as follows:
\begin{theorem}\label{thm2.1}
If $p\geq 3$ and $\mu$ is centered, then $\beta_c$ is the largest $b$ such that $\sup_{v\in (0,v_*]}\Gamma_b(v)=0$.
\end{theorem}


As an example of Theorem \ref{thm2.1}, we demonstrate numerical simulations for estimating the critical threshold $\beta_c$ for the sparse Rademacher prior, in which the entries $u_1,\ldots,u_N$ in $u$ are i.i.d. sampled from the probability distribution
$$
\frac{\rho}{2}\delta_{-\frac{1}{\sqrt{\rho }}} + (1-\rho)\delta_0 + \frac{\rho}{2}\delta_{\frac{1}{\sqrt{\rho }}},
$$
with parameter $\rho\in (0,1]$ that controls the sparsity of the prior.
The case $\rho=1$ corresponds to the usual Rademacher prior.
If $\rho<1,$ the sparse Rademacher prior can be regarded as first uniformly sampling  approximately $\rho N$ of the coordinates and then for these coordinates, sampling Bernoulli $\pm 1/\sqrt{\rho}$ random variables with equal probability. The remaining approximately $(1-\rho)N$ coordinates are set to zero. From this construction, the second moment of $\|u\|/\sqrt{N}$ is of order 1. To simulate $\beta_c$ according to the value established in Theorem~\ref{thm2.1}, we numerically evaluate $\Gamma_{b}(v)$ for test values of $v$ with increments $.001$ in the interval between 0 and $v_*=1$. For this purpose, we have used the numerical integrator of Mathematica. The critical value $\beta_c$ is the largest value $b$ such that $\Gamma_{b}(v) \leq 0$ for all test values of $v$, where discrete positive values of $b$ with increments $0.001$ were tested. Figure \ref{prob_sparse_Rad} summarizes the numerical results for $p=3$, 4, 5, 10 and $\rho = 0.1$, $0.2,$ $\dots$, 1.

The behavior of  $\beta_c$ is influenced by the proportion of zeros and the magnitude of the nonzero jumps.  As can be seen, in each of the four figures there exists a threshold $\rho_*$ (depending on $p$) such that $\rho\mapsto\beta_c$ is increasing on $[0,\rho_*]$ and decreasing on $[\rho_*,1]$. Heuristically, in the interval $[0,\rho_*)$, the large fraction of the zeros dominates the small proportion of far jumps, whose magnitude $1/\sqrt{\rho}$ is large. On the other hand, in the interval $(\rho_*,1]$, the  far jumps overpower the small fraction of zeros and their magnitude has relatively low variation with $\rho$.  In each subfigure of Figure \ref{prob_sparse_Rad}, we indicate by a solid curve the following upper bound for $\beta_c$, which was pointed out in \cite{PWB17},
$$
H(\rho):=\sqrt{2\bigl(-\rho \log \rho - (1-\rho) \log(1-\rho)+\rho \log 2\bigr)}.
$$
We note that as $p$ increases the estimated values of $\beta_c$ are closer to the ones of the upper bound $H(\rho)$. For $p=3$, $4$, $5$, we see that if $\rho$ is sufficiently small, then $H(\rho)$ is still a good approximation for $\beta_c$.

\begin{figure}[!htb]
	\centering
	\begin{subfigure}[t]{0.45\textwidth}
		\centering
		\includegraphics[width=\textwidth]{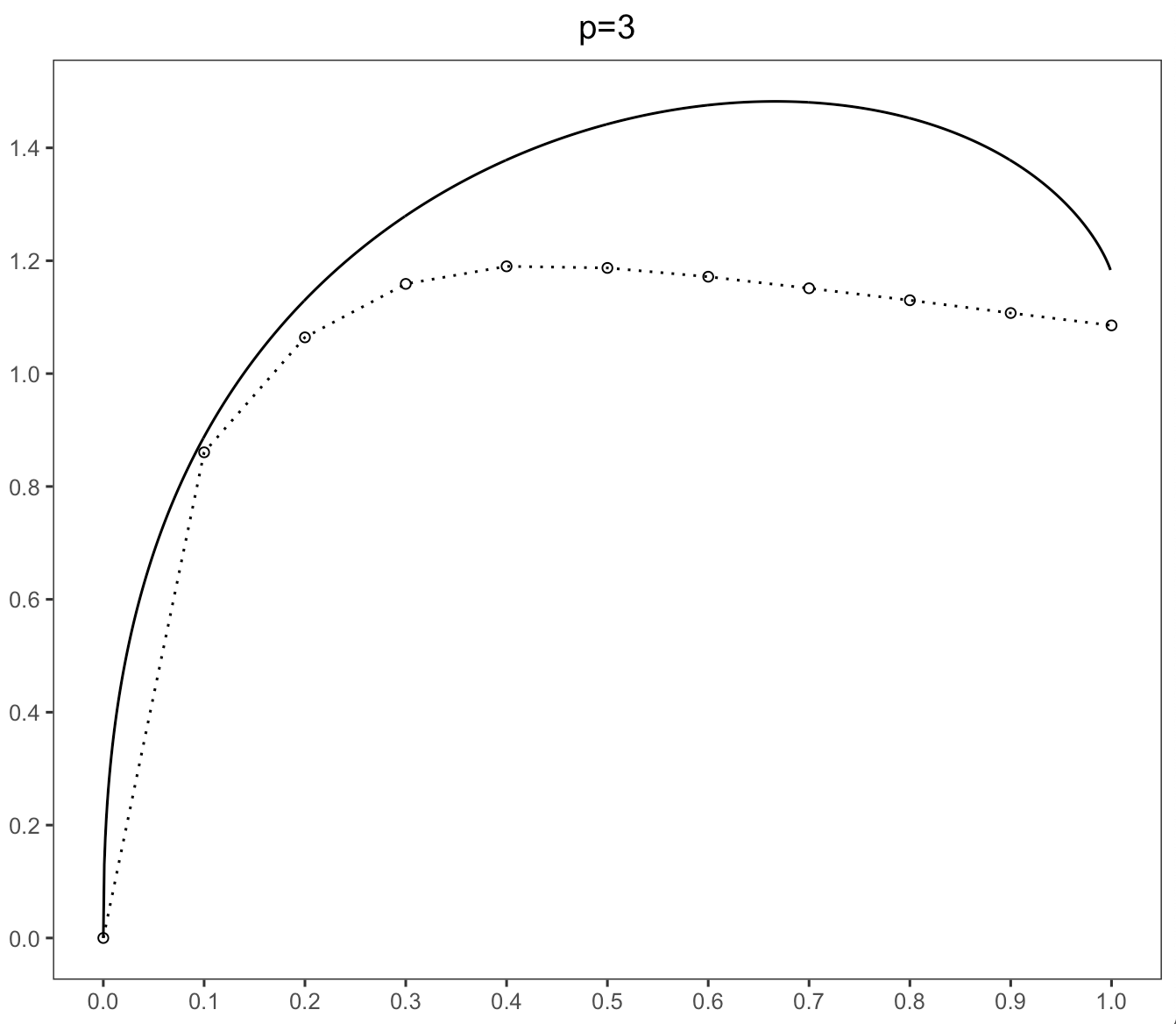}
	\end{subfigure}
	\begin{subfigure}[t]{0.45\textwidth}
		\centering
		\includegraphics[width=\textwidth]{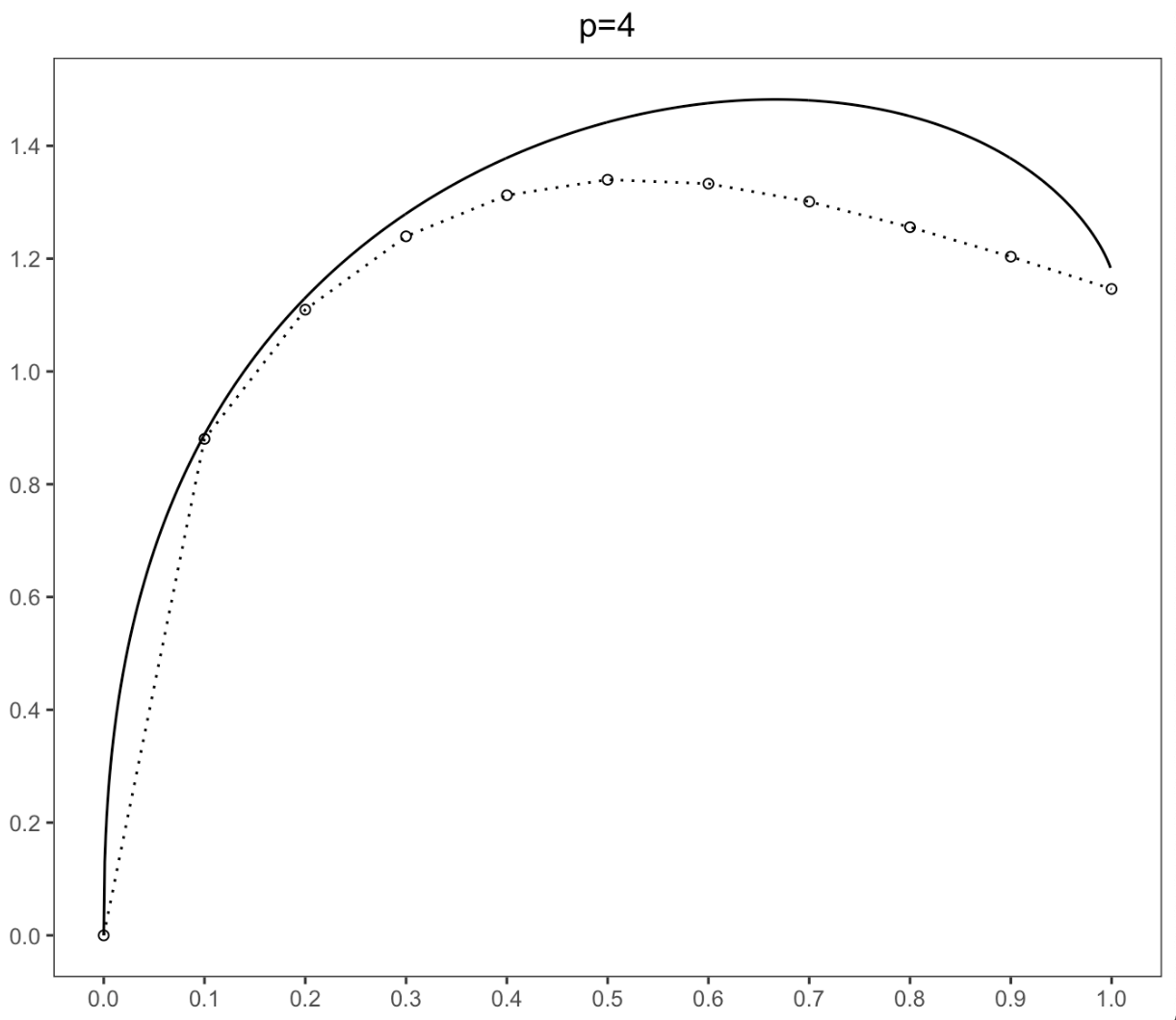}
	\end{subfigure}
	\vskip\baselineskip
	\begin{subfigure}[b]{0.45\textwidth}
		\centering
		\includegraphics[width=\textwidth]{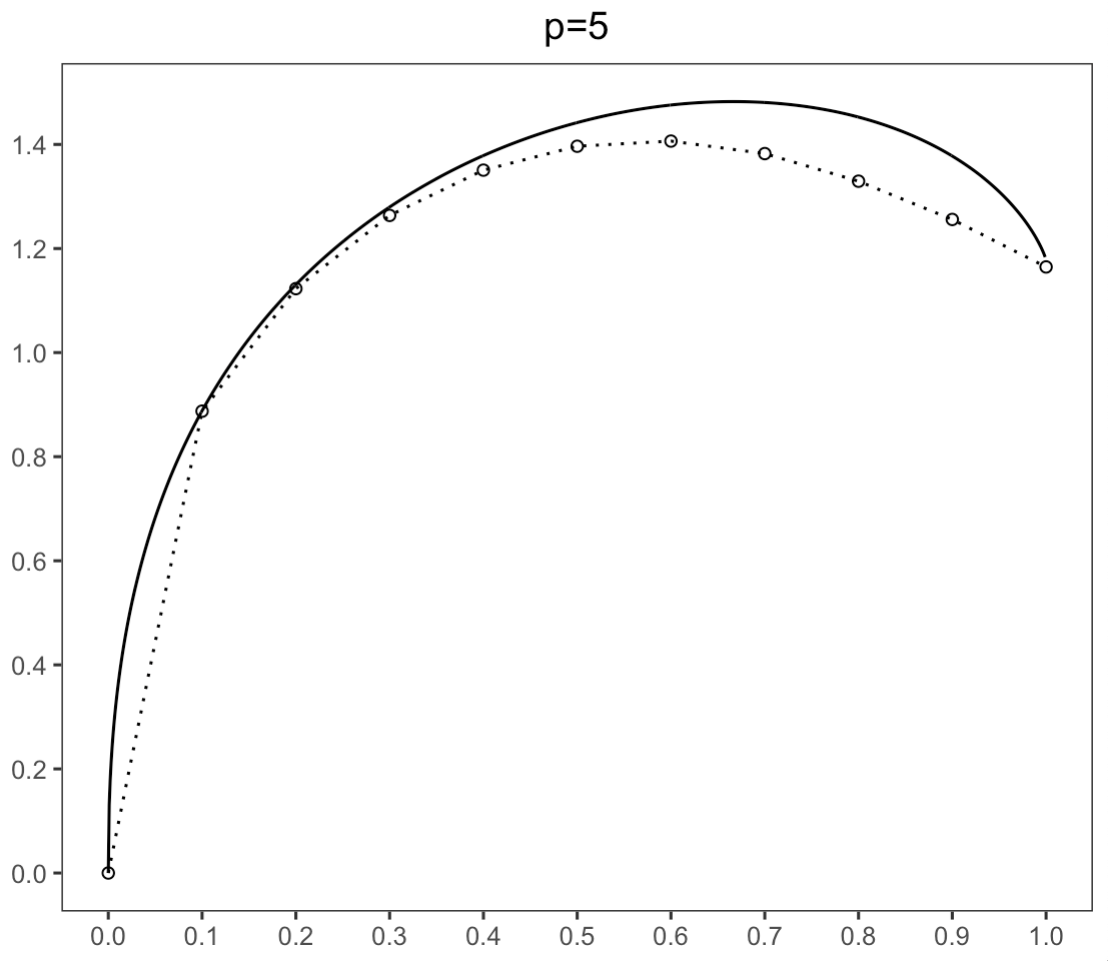}
	\end{subfigure}
	\quad
	\begin{subfigure}[b]{0.45\textwidth}
		\centering
		\includegraphics[width=\textwidth]{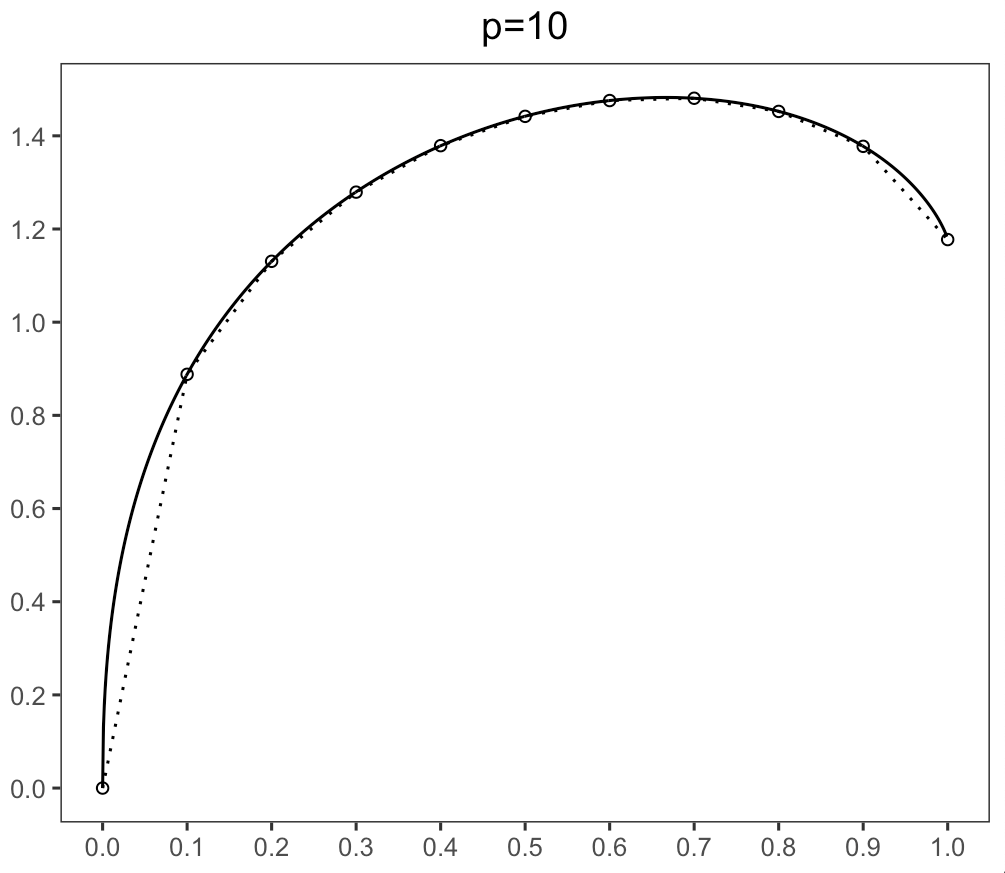}
	\end{subfigure}
	\caption{Numerical simulations for the critical value $\beta_c$ with sparse Rademacher prior and various values of $p$. The top left plot is for $p=3$, the top right for $p=4$, the bottom left for $p=5$ and the bottom right for $p=10$.  The open circles are the simulated critical values $\beta_c$. The dashed curve interpolates between these points and the solid curve describes the function $H(\rho)$.}
	\label{prob_sparse_Rad}
\end{figure}

\subsection{Main Results for  Detection of Multiple Spikes }\label{sec2.3}
In this subsection, we study the case of more than one spike. Denote the number of spikes by $k$.
Let $\Lambda_1,\ldots,\Lambda_k$ be bounded subsets of $\mathbb{R}$ and $\mu_1,\ldots,\mu_k$ be centered probability measures on the Borel $\sigma$-fields of $\Lambda_1,\ldots,\Lambda_k$, respectively.
For any $1\leq r\leq k,$ let $u_1(r),\ldots,u_N(r)$ be i.i.d. samplings from $\mu_r$ and set
$
u(r)=(u_1(r),\ldots,u_N(r)).
$
We assume that $u(1),\ldots,u(k)$ are independent of each other and of $W$.
For $\bar \beta=(\beta_1,\ldots,\beta_k)$ with $\beta_1,\ldots,\beta_k>0,$ the spiked tensor $T_k$ is defined by
\begin{align}\label{spikes}
T_k=W+\frac{1}{N^{(p-1)/2}}\sum_{r=1}^k \beta_ru(r)^{\otimes p}.
\end{align}
In a manner similar to the previous subsection, we say that detection is possible if $W$ and $T_k$ are distinguishable and is impossible if they are indistinguishable. For $1\leq r\leq k,$ denote by $\beta_{r,c}$ the critical threshold obtained by plugging $\mu_r$ into Theorem~\ref{thm2.1}. We extend Theorem \ref{thm1} to the case of multiple spikes as follows.

\begin{theorem}\label{maintheorem}
Assume that $\mu_1,\ldots,\mu_k$ are centered. For $p\geq 3$, the following statements hold.
	\begin{itemize}
		\item[$(i)$] If $\bar \beta\in (0,\beta_{1,c})\times\cdots\times (0,\beta_{k,c})$, then detection is impossible;
		\item[$(ii)$] If $\bar \beta \notin (0,\beta_{1,c}]\times\cdots\times (0,\beta_{k,c}]$,  then detection is possible.
	\end{itemize}
\end{theorem}

Theorem \ref{maintheorem} implies that in order to detect the spikes, at least one of the $\beta_r$'s has to exceed its own marginal critical threshold $\beta_{r,c}$. In particular, if all probability measures are the same, that is, $
\mu_1=\cdots=\mu_k,
$
then the above result implies that $W$ and $T_k$ are indistinguishable if $\max_{1\leq r\leq k}\beta_r<\beta_c$ and are distinguishable if $\max_{1\leq r\leq k}\beta_r>\beta_c$, where $\beta_c$ is the common threshold for all components. 

\begin{remark}
	\rm The first statement of Theorem \ref{maintheorem} directly follows from Theorem \ref{thm1} and a triangle inequality for the total variation distance, which is formulated in Lemma \ref{lem2}. The second statement of Theorem \ref{maintheorem} is nontrivial and requires a thorough study of the high temperature regime of a spin glass model (see Section \ref{hd}).
\end{remark}

It is natural to ask whether this critical threshold $\beta_c$ would change if one allows $k$ to grow  with $N.$ We show that this is not the case if the growth of $k=k(N)$ is of certain polynomial order, which is sufficiently slow in comparison to the size of the $p$-tensor, $N^p.$
To state our result, let $\mu$ be the probability measure considered in Section \ref{sec2.2} and let $\beta_c$ be the corresponding critical value provided by Theorem \ref{thm2.1}. Assume that $\mu_r=\mu$ for all $r\geq 1$ and that $(\beta_r)_{r\geq 1}$ is a sequence of SNRs satisfying $\sup_{r\geq 1}\beta_r<\infty.$ Let $T_k$ be the random tensor in \eqref{spikes} with $\mu_r$ and $\beta_r$ for $1\leq r\leq k.$

\begin{theorem}\label{maintheorem2}
Assume that $p\geq 3$ and $k=k(N)$ satisfies
$$
\lim_{N\to\infty}k(N)=\infty\,\,\mbox{and}\,\,\lim_{N\rightarrow\infty}\frac{k(N)}{N^{(p-2)/4}}=0.
$$
We have that
\begin{enumerate}
	\item[$(i)$] Detection is impossible if $\sup_{r\geq 1}\beta_r<\beta_c$.
	\item[$(ii)$] Detection is possible if either of the following two assumptions is satisfied:
	\begin{itemize}
		\item [$(a.1)$] There exists a fixed $k_0\in \mathbb{N}$ such that $\beta_r>\beta_c$ for at least one $r\leq k_0$ and $\sup_{r> k_0}\beta_r<\beta_c$.
		\item [$(a.2)$] $p$ is even and there exists a fixed $k_0\in \mathbb{N}$ such that $\inf_{r\geq k_0}\beta_r>\beta_c.$
	\end{itemize}
\end{enumerate}
\end{theorem}

\begin{remark}\rm 
We expect that detection is possible if at least one of the $\beta_r$'s exceeds $\beta_c$. However, our proof requires more restrictive conditions for technical reasons.
Note that condition (a.1) excludes the case of infinitely many $\beta_r$'s above $\beta_c$. On the other hand, condition (a.2) does not exclude this case, but requires $p$ to be even.
\end{remark}

As the number of independent spikes grows in $N$, it seems reasonable to believe that the critical threshold $\beta_c$ should become smaller since now we have more spikes and it should be relatively easier to detect them in comparison to the case of  a fixed finite number of spikes. However, Theorem~\ref{maintheorem2} presents a counterintuitive result that if the total number of spikes is of smaller order than $N^{(p-2)/4}$, then the critical threshold remains unchanged. 
In particular, if we let $\beta_r = \tilde{\beta}/k$ for all $1 \leq r \leq k$
and use the average of the spikes, $U_k = k^{-1} \sum_{r=1}^k u(r)^{\otimes p}$, then we can write
\[
T_{k(N)} = W + \frac{\tilde{\beta}}{N^{(p-1)/2}} U_{k(N)}.
\]
If $k=k(N)$ satisfies the growth conditions above, Theorem \ref{maintheorem2} says that detection is impossible if $\tilde{\beta} < k\beta_c$ and detection is possible when $\tilde{\beta} > k\beta_c$ and $p$ is even. Interestingly, this growth rate of $\tilde\beta$ required for detection, matches some recent results about algorithmic thresholds for spike recovery. In \cite{MR+14}, Montanari and Richard established recovery via the tensor unfolding up to the threshold $N^{(\lceil p/2\rceil-1)/2}$ and predicted that the optimal threshold should be $N^{(p-2)/4}.$ In addition, they obtained recovery via the tensor power iteration up to the threshold $N^{(p-1)/2}$ and conjectured that the true threshold for the power iteration and the Approximate Message Passing algorithm is $N^{(p-2)/2}$ with no comment on the initialization. In \cite{BGJ+18}, Ben Arous-Gheissari-Jagannath studied results for the Langevin dynamics and gradient descent and they gave recovery guarantees when $\tilde\beta > N^{(p-2)/2}$ for spherical  and spin glass initial data. It was also known that the degree 4 sum-of-squares algorithm  \cite{hopkins2015tensor} and a related spectral algorithm  \cite{hopkins2016fast} (for $p=3$) have sharp recovery threshold $N^{(p-2)/4}$. More recently,  a hierarchy of spectral methods following Kikuchi free energy has been proposed in the spiked tensor model \cite{WEM19}, where it was shown that  as long as the order parameter $\ell$ satisfies $\ell=o(N)$, then strong recovery and detection hold whenever $\tilde\beta \gg \ell^{-(p-2)/4}N^{(p-2)/4}\sqrt{\log N}.$

\subsection{Byproduct: Result for Recovery by MMSE}
\label{sec2.4}



Recall the settings of Section~\ref{sec2.3}. Let $\hat \theta = (\hat\theta_{i_1,\dots, i_p})$ be a $\R^{N^p}$-valued bounded random variable generated by the $\sigma$-field $\sigma(T_k)$. We also allow the random variable $\hat \theta$ to be dependent on other randomness that are independent of the $u_i(r)$'s and $T_k$.
The minimum mean square error (MMSE) is defined by
\begin{align}\label{MMSE}
\MMSE_N(\bar \beta):=\min_{\hat\theta}\frac{1}{N^p}\sum_{1\leq i_1,\ldots,i_p\leq N}\e\Bigl(\sum_{r=1}^k\beta_ru_{i_1}(r)\cdots u_{i_p}(r)-\hat\theta_{i_1,\ldots,i_p}\Bigr)^2,
\end{align}
where the minimum is taken over all such $\hat \theta$. The minimizer to this problem is attained by the minimum mean square estimator,
\begin{align*}
\hat\theta^{\MMSE}&=\Bigl(\sum_{r=1}^k\beta_r\e [u_{i_1}(r)\cdots u_{i_p}(r)|T_k]\Bigr)_{1\leq i_1,\ldots,i_p\leq N}.
\end{align*}
By restricting the minimum in the definition of $\MMSE_N(\bar \beta)$ to the  so-called dummy estimators \cite{LM+16}, i.e., estimators where $\hat\theta$ is independent of $T_1,\ldots,T_k$, one obtains a trivial upper bound:
\begin{align*}
\MMSE_N(\bar \beta)\leq \frac{1}{N^p}\sum_{1\leq i_1,\ldots,i_p\leq N}\Bigl(\e\Bigl(\sum_{r=1}^k\beta_ru_{i_1}(r)\cdots u_{i_p}(r)\Bigr)^2-\Bigl(\e\Bigl[\sum_{r=1}^k\beta_ru_{i_1}(r)\cdots u_{i_p}(r)\Bigr]\Bigr)^2\Bigr).
\end{align*}
Denote $v_{r,*}:=\int a^2\mu_r(da)$ for $1 \leq r \leq k$. Applying the strong law of large numbers yields
\begin{align*}
\limsup_{N\rightarrow\infty}\MMSE_N(\bar \beta)&\leq \DMSE(\bar \beta):=\sum_{r=1}^k\beta_r^2v_{r,*}^p.
\end{align*}
Note that from the Gaussianity of $Y$, it can be directly checked that conditionally on $T_k,$ the distribution of $(u(1),\ldots,u(k))$ can be described by a Gibbs measure $G_N^A$ on $\Lambda_1^N\times\cdots\times \Lambda_k^N$, see \eqref{add:eq-5} below. Denote by $(\sigma(1),\ldots,\sigma(k))$  a sampling from $G_N^A$ and by $\la \cdot\ra^A$ the expectation associated to this measure. 
We show that 

\begin{theorem}\label{thm0}
For $p\geq 3$, the following statements hold.
	\begin{itemize}
		\item[$(i)$] If $\bar \beta\in (0,\beta_{1,c})\times\cdots\times (0,\beta_{k,c})$, then $$
		\lim_{N\rightarrow\infty}\MMSE_N(\bar \beta)=\DMSE(\bar \beta)\quad\mbox{and}\quad
		\lim_{N\to\infty}\sum_{r,r'=1}^k\beta_r\beta_{r'}\e \Bigl\la \Bigl(\frac{1}{N}\sum_{i=1}^Nu_i(r)\sigma_i(r')\Bigr)^p\Bigr\ra^A=0.$$
		\item[$(ii)$] If $\bar \beta\notin (0,\beta_{1,c}]\times\cdots\times (0,\beta_{k,c}]$, then
	$$
	\limsup_{N\rightarrow\infty}\MMSE_N(\bar \beta)<\DMSE(\bar \beta)\quad\mbox{and}\quad
	\liminf_{N\to\infty}\sum_{r,r'=1}^k\beta_r\beta_{r'}\e \Bigl\la \Bigl(\frac{1}{N}\sum_{i=1}^Nu_i(r)\sigma_i(r')\Bigr)^p\Bigr\ra^A>0.$$
	\end{itemize}
	
\end{theorem}

This theorem asserts that if the SNRs of all marginal spikes are less than their critical thresholds, then when estimating for the tensor $\sum_{r=1}^k\beta_ru(r)^{\otimes p}$, the minimum mean square estimator is no better than a random guess. In contrast, if at least one of the SNRs of the marginal spikes is larger than its critical threshold, the minimum mean square estimator performs better than all dummy estimators. In the case that $p$ is even, Theorem \ref{thm0} $(i)$ further implies that the sampling $(\sigma(1),\ldots,\sigma(k))$ does not provide useful information in recovering $(u(1),\ldots,u(k))$ since their inner products are essentially zero, whereas Theorem \ref{thm0} $(ii)$ shows that a weak form of recovery is possible as $(u(1),\ldots,u(k))$ and $(\sigma(1),\ldots,\s(k))$ are asymptotically correlated.

As mentioned before, the spike recovery in the random tensor for general priors via the MMSE was studied earlier by Lesieur-Miolane-Lelarge-Krzakala-Zdeborov\'a \cite{LMLKZ+17}.
They computed the limiting mutual information between $W$ and $T_k$ and used it to establish a result equivalent to Theorem \ref{thm0}. The proof of Theorem \ref{thm0} relies heavily on our main results for the detection problem and presents  a different approach than the one taken in \cite{LMLKZ+17}.




\subsection{Previous Results} \label{prev_res}


Understanding phase transitions of spike detection and recovery problems in spiked random matrices and tensors has received a lot of attention in the past several years. We summarize some recent works here.

{\noindent \bf Matrix Case: $\boldsymbol{p=2}$.} The PCA approach was studied by Baik-Ben Arous-P\'{e}ch\'{e} \cite{BBP}, Baik-Silverstein \cite{baik2}, F\'{e}ral-P\'{e}ch\'{e} \cite{PecheD}, Johnstone \cite{Johnstone}, Paul \cite{paul2007}, and P\'{e}ch\'{e} \cite{Peche}. Barbier-Dia-Macris-Krzakala-Lesieur-Zdeborov\'{a} \cite{BDM+16} studied the MMSE recovery problem in the spiked random matrix in \eqref{spike} (see the setting in Section \ref{sec2.4} with $p=2$ and $k=1$) by deriving a Parisi-type formula for the mutual information between $W$ and $T$. Analogous study for the case of multiple spikes \eqref{spikes} was handled by Lelarge-Miolane \cite{LM+16}, where $u(1),\ldots,u(k)$ are assumed to have finite second moments and are allowed to be correlated. Similar result for the non-symmetric case was pursued by Miolane~\cite{M+17}.

As for the detection problem, under the same setting as \eqref{spike}, Alaoui-Krzakala-Jordan~\cite{AKJ+17} obtained the same critical value $\beta_c$ specified in Equation \eqref{critical} and Proposition \ref{prop1} below. It was deduced that above $\beta_c$, detection is possible and below $\beta_c$, a weak form of detection remains possible in the sense that the limiting total error (the sum of type one and type two errors) of the likelihood ratio test between $W$ and $T$ is strictly less than one. Incidentally, we mention that when the results of \cite{BDM+16,LM+16} apply to the case \eqref{spike}, $\beta_c$ is also the critical threshold for recovery. In \cite{el2018detection}, El Alaoui and Jordan extended the results of \cite{AKJ+17} to the case of spiked rectangular matrices, where the spike is of the form $uv^T$ and it was assumed that the entries of $u \in \R^M$, $v\in \R^N$ are chosen independently at random from possibly different priors and $M/N \rightarrow \alpha$. It was shown that for a set of parameters $(\alpha,\beta)$ the results of \cite{AKJ+17} hold. This set of parameters is sub-optimal for most priors as the spin-glass methods used fail near the boundaries of the optimal parameter space for the model of \cite{el2018detection}.

{\noindent \bf Tensor Case: $\boldsymbol{p\geq 3}.$} Earlier results trace back to the works of  Montanari-Richard \cite{MR+14} and Montanari-Reichman-Zeitouni~\cite{MRZ17}, where the authors considered \eqref{spikes} with $k=1$ and a spherical prior, i.e., $u$ in \eqref{spike} is uniformly and independently sampled from the sphere, $\{x\in \mathbb{R}^N:\sum_{i=1}^Nx_i^2=N\}.$ By adaptation of the second moment method, they showed that there exist $\beta_-$ and $\beta_+$ such that detection is impossible for $\beta$ below $\beta_-$ and is possible for $\beta$ above $\beta_+.$

Lesieur-Miolane-Lelarge-Krzakala-Zdeborov\'{a} \cite{LMLKZ+17} considered \eqref{spikes}  with a general setting in which the vectors $(u_i(1),\ldots,u_i(k))$ for $1\leq i\leq N$  are i.i.d. sampled from a joint distribution
with finite second moments.
For centered priors, they proved that there exists a vector of critical thresholds $(\beta_{1,c}',\dots,\beta_{k,c}')$ such that for any $\bar \beta = (\beta_1, \dots, \beta_k)$ satisfying $\beta_r>\beta_{r,c}'$ for $1\leq r\leq k$, the MMSE estimator obtains a better error than any dummy estimator. Consequently, one can also detect the spike in that case. In addition, when $\bar \beta$ satisfies $\beta_r<\beta_{r,c}'$ for all $1\leq r\leq k,$ the MMSE estimator is statistically irrelevant to recover the spike. They did not provide results for the detection problem in this case. Notably, if $u(1),\ldots,u(k)$ are chosen as in Section~\ref{sec2.3}, our critical thresholds $\beta_{r,c}$ agree with $\beta_{r,c}'$ and as a consequence, their result in this case is the same as Theorem \ref{thm0}.
Barbier-Macris \cite{BM+17} provided a different proof for the results of  \cite{LMLKZ+17} by using stochastic interpolation.
Analogous results to \cite{LMLKZ+17} were developed in non-symmetric settings by  Barbier-Macris-Miolane~\cite{BMM+17}.

Perry-Wein-Bandeira~\cite{PWB17} focused on $k=1$ and three priors: the spherical prior, the Rademacher prior, and the sparse Rademacher prior. In these three settings, it was proved that there exist lower and upper bounds $\beta_{-}'$ and $\beta_{+}'$ such that detection is not possible when $0<\beta \leq \beta_-'$ and is possible when $\beta\geq \beta_+'.$
In particular, their result in the spherical case improved the existing bounds in \cite{MRZ17,MR+14} mentioned above.
For the Rademacher prior, Chen~\cite{Chen17} closed the gap between $\beta_-'$ and $\beta_+'$ by showing that $\beta_c$ in Theorem \ref{thm2.1} is indeed the critical threshold for detection. The present work extends the results of \cite{Chen17,PWB17} to a broader class of priors and also to $k > 1$.

{\noindent \bf Other Related Works.} Since the likelihood ratio test and the MMSE estimator are often intractable to compute, it is natural to ask about the performance of tractable algorithms for detection and recovery for low-rank signals. The works \cite{BDMK+16, BKMMZ+17, BMDK+17, deshpande2016, DM+14,LM+16, M+17,MR+16} studied the performance of the approximate message passing (AMP) algorithm in recovering the spike. See \cite{BM+11, DMM+09, JM+13, RP+16} for the performance of AMP in compressed sensing. See \cite{BGJ+18} for the performance of the Langevin dynamics and the gradient decent in the spiked tensor model. The complexity of energy landscapes in  spiked tensor models was studied in \cite{BMMN+17,RBBC+18}.

\subsection{Our approaches}


As mentioned above, the work \cite{Chen17} considered the Gaussian $p$-tensor model for all $p\geq 3$ with a single spike \eqref{spike} sampled from the Rademacher prior and it obtained the same result as Theorem~\ref{thm1}. In the present paper, we extend \cite{Chen17} to general priors and to multiple spikes.
In view of \cite{Chen17}, the approach was based on a connection between the total variation distance of the pair $(W,T)$ and the free energy (see \eqref{fg}) of the pure $p$-spin spin glass model through an integral representation (see Lemma \ref{dtv_free}). From this, proving the impossibility of detection relies on knowing the decaying rate of the tail probability of the free energy in the high temperature regime. The core ingredient of obtaining this tail probability relied on a delicate study of the fluctuation of the free energy via the Parisi formula, the coupled free energy with overlap constraints, and the two-dimensional Guerra-Talagrand inequality. 

Theorem \ref{thm1} follows essentially from the same treatments as \cite{Chen17} by studying the fluctuation of a one-dimensional spin glass free energy \eqref{fg}. However, while the arguments in \cite{Chen17} were greatly simplified due to the simple structure of the Rademacher prior, there are a number of analytic obstacles in handling our generalization. For example, the main results in \cite{Chen17} critically relied on the strict monotonicity of $\gamma_\beta$ in $\beta$ (recalling \eqref{auxfun}). The proof of this property used the symmetry of the Rademacher prior and it does not carry through in our setting. To prove our main results, we establish an analogous, though more general proposition, in Lemma~\ref{lem3} below, which requires a completely new argument.

Our approach to the high-dimensional generalization, Theorem \ref{maintheorem}, relies on the high temperature behavior of the free energy  associated to the vector-valued pure $p$-spin spin glass model, see Section~\ref{hd}.
In spin glasses, vector-valued models are usually harder as the spin components interact with each other in a highly complicated way. As a result, the analysis of the Parisi formula for the corresponding free energy and its coupled version becomes more involved. Nevertheless, to study the high temperature regime, we can directly handle the free energy  by reducing the high-dimensional Hamiltonian to one-dimensional ones by exploiting the overlap constraints, see Section~\ref{high_temp_k>1}. This helps us avoid controlling the Parisi formula of the vector-valued model and greatly simplifies our argument. Ultimately, this leads to a full characterization of the high temperature regime of the vector-valued pure $p$-spin model (Theorem~\ref{thm6} below) and concludes Theorem \ref{maintheorem}. We expect that this approach is also applicable in characterizing the high temperature regimes for more general spin glass models, for instance, the vector-valued {\it mixed} $p$-spin model without external field.

\begin{remark}\rm 
	The assumption on the boundedness of the support of $\mu$ is  used for technical purposes.  For example, one can note that the estimates in Lemmas \ref{boundderiv}-\ref{jthpower2} use the bound $M$ on the size of the support of $\mu$.
	We believe that our results remain valid when $\mu$  is unbounded, but has exponential tail probability. 
\end{remark}

\begin{remark}\label{rmk3}
	\rm 
	
We derive our main results in Section \ref{sec2} by studying the pure $p$-spin mean-field spin glass model, whose one-dimensional Hamiltonian is defined in \eqref{hamx}. As an alternative approach, one can study instead the planted model, whose Hamiltonian is formulated in \eqref{pham} with $t=1$. It is natural for Bayesian estimation and can save some technical issues. Most notably, the analog of Proposition \ref{new:lem0} for the planted model has a simpler proof.
For details, we refer the readers to check \cite[Proposition 13]{AKJ+17} and  \cite[Proposition 16]{AKJ20}, which are 
analogs of Proposition \ref{new:lem0}, but use the planted model for the one-dimensional spike detection problem in the matrix setting (i.e., $p=2$).  
The planted model also yields an alternative scheme for determining the critical threshold $\beta_c$ by using the replica symmetric Parisi formula \cite[Theorem 1]{LMLKZ+17}.
Nevertheless, the model that we consider in this paper is essential to the field of spin glasses (see, e.g., Bolthausen \cite[Section 6]{Bol18}) and our study of its high-temperature behavior is of independent interest. 
\end{remark}


	

\subsection{Structure of the Rest of the Paper}\label{structure}

The key ingredient of this paper relies on an observation that the total variation distance between $W$ and $T_k$ can be expressed as an integral related to the free energy of the pure $p$-spin models with scalar- and vector-valued spin configurations (Lemma \ref{dtv_free}). Section \ref{sec3} defines these models, characterizes their high-temperature regimes and presents results on the fluctuation of the free energy and concentration of the overlap of the models. Section \ref{sec4} establishes Theorems \ref{thm1}-\ref{maintheorem2}, while Section \ref{sub4.4} presents the proof of Theorem \ref{thm0}. The rest of the sections are devoted to establishing the main results in Section~\ref{sec3}. In Sections \ref{sec:high_temp_k1} and \ref{high_temp_k>1}, we prove the asserted structures of the high-temperature regimes. These proofs are the most crucial components in this paper. Sections \ref{sec6} and \ref{sec7}  establish  the high-temperature behavior of the overlap and the free energy when $k=1.$

\section{Pure $p$-spin Models}\label{sec3}

In this section, we introduce the pure $p$-spin mean field spin glass models with scalar-valued and vector-valued spin configurations and formulate some crucial results regarding their high-temperature behavior. Their proofs are deferred to later sections.

\subsection{Scalar-valued Model}\label{vec_spin_model}

Recall the random tensor $Y$ from Section \ref{sec2.1} and the probability space $(\Lambda,\mu)$ from Section \ref{sec2.2}. For any $\sigma\in \Lambda^N$, the Hamiltonian of the pure $p$-spin model is defined as
\begin{align}
\label{hamx}
X_N(\s) = \frac{1}{N^{(p-1)/2}}\bigl\la Y,\sigma^{\otimes p}\bigr\ra=\frac{1}{N^{(p-1)/2}} \sum_{1 \leq i_1,\dots,i_p \leq N} Y_{i_1,\dots,i_p} \sigma_{i_1}\cdots \sigma_{i_p},
\end{align}
where the $Y_{i_1,\dots,i_p}$'s are i.i.d. standard Gaussian random variables. Note that by the symmetry of $W$, we also have the identity $X_N(s)=N^{-(p-1)/2}\la W,\sigma^{\otimes p}\bigr\ra.$ For any two spin configurations $\s^1$ and $\s^2$, the covariance of $X_N$ can be computed as
\[
\mathbb{E}\bigl(X_N(\sigma^1)X_N(\sigma^2)\bigr) = N \bigl(R(\sigma^1,\sigma^2)\bigr)^p,
\]
where $R(\s^1,\s^2)$ is the overlap between $\s^1$ and $\s^2$ defined by
\begin{align*}
R(\sigma^1,\sigma^2)=\frac{1}{N}\sum_{i=1}^N\sigma_i^1\sigma_i^2.
\end{align*}
Define the re-centered Hamiltonian $H_{N,\beta}(\s)$ by
\begin{align*}
H_{N,\beta}(\s)=\beta X_N(\s)-\frac{\beta^2N}{2}R(\s,\s)^p.
\end{align*}
Note that $\e e^{H_{N,\beta}(\s)}=1.$
Define the free energy and Gibbs measure respectively by
\begin{align}\label{fg}
F_N(\beta)=\frac{1}{N}\log \int e^{H_{N,\beta}(\s)}\mu^{\otimes N}(d\s)
\end{align}
and
\begin{align*}
G_{N,\beta}(d\s)&=\frac{e^{H_{N,\beta}(\s)}\mu^{\otimes N}(d\s)}{Z_{N,\beta}},
\end{align*}
where $Z_{N,\beta}$ is the normalizing constant so that $G_{N,\beta}$ is a probability measure on $\Lambda^N.$ Denote by $\la \cdot\ra_\beta$ the Gibbs expectation with respect to the i.i.d. samplings $\s,\s^1,\s^2,\ldots$ from the Gibbs measure $G_{N,\beta}.$ 

A few properties of $F_N$ are in position. First of all, an application of the Gaussian concentration of measures implies that $F_N(\beta)$ is concentrated around $\e F_N(\beta)$. Here, from the Jensen's inequality, 
\begin{align*}
\e F_N(\beta)&\leq \frac{1}{N}\log \int \e e^{H_{N,\beta}(\s)}\mu^{\otimes N}(d\s)=\frac{1}{N}\log 1=0.
\end{align*}
In addition, $\e F_N(\beta)$ is a nonincreasing function since by using Gaussian integration by parts, 
\begin{align*}
\frac{d}{d\beta}\e F_N(\beta)&=\frac{1}{N}\e \la X_{N}(\sigma)\ra_\beta-\beta\e \la R(\sigma,\sigma)^p\ra_\beta\\
&=-\beta\e \la R(\sigma^1,\sigma^2)^p\ra_\beta=-\beta \sum_{1\leq i_1,\ldots,i_p\leq N}\e \bigl(\la \sigma_{i_1}\cdots \sigma_{i_p}\ra_\beta^2\bigr)\leq 0.
\end{align*}
Second, it can be shown (see Proposition \ref{new:parisiformula:k=1} below) that for all $\beta$, $\lim_{N\to\infty}\e F_{N}(\beta)$ exists. Denote this limit by $F(\beta)$. From above, $F(\beta)$ is nonpositive and nonincreasing. Define the high-temperature regime as
$$\mathcal{R}=\{\beta>0:F(\beta)=0\}.$$ 
The low-temperature regime is define as $\mathcal{R}^c$. Set the critical threshold $\beta_c$ by
\begin{align}\label{critical}
\beta_c=\sup\mathcal{R}.
\end{align}
In spin glasses, the parameter $\beta$ is understood as the (inverse) temperature parameter, while in the detection problem of \eqref{spikes}, it is interpreted as the signal strength or SNR. These equivalent meanings of $\beta$ are justified below in Lemma \ref{dtv_free} via an integral representation for the total variation distance between $W$ and $T$.

The following proposition shows that the high-temperature regime $\cal{R}$ is an interval and its right-end boundary is $\beta_c$.
It also gives a characterization of $\cal{R}$ in terms of the constant $v_*$  and the auxiliary function $\Gamma_b(v)$ defined in \eqref{q1} and \eqref{add:eq1}, respectively.

\begin{prop}\label{prop1}
	For $p \geq 2$, $\mathcal{R} = (0, \beta_c].$  For  $\beta>0,$ $\beta\in \mathcal{R}$ if and only if $\sup_{v\in (0,v_*]}\Gamma_\beta(v)\leq 0.$
\end{prop}

Next, we show that in the interior of the high-temperature regime, the overlap between two i.i.d. samples $\s^1$ and $\s^2$ is concentrated around zero.

\begin{theorem}\label{momentcontrol}
    For $p \geq 2$, $m\in \mathbb{N}$, and $0<\beta <\beta_c$,  there exists a constant $K>0$, depending only on $p$, $m$, and $\beta$, such that
	\begin{align}\label{momentcontrol:eq1}
	\E \bigl\la |R(\s ^1 ,\s ^2 )|^{2m}\bigr\ra_{s\beta } \leq \frac{K}{N^{m}}, \,\, \forall s\in [0,1],N\geq 1.
	\end{align}
\end{theorem}

Furthermore, we control the fluctuation of the free energy as follows.

\begin{prop}\label{fluctuations} For  $p\geq 2$ and  $0<\beta <\beta_c$, there exists a constant $K$, depending only on $p$ and $\beta$, such that
	\[
	\p \left(|F_N(\beta)| \geq l \right) \leq \frac{K}{l^2 N^{p/2+1}}, \, \,\forall l > 0, N \geq 1.
	\]
\end{prop}

In the case that $\mu$ is a uniform probability measure on $\{-1,1\}$, the behavior of the overlap and the fluctuation of the free energy at high-temperature is well-understood. The case $p=2$ corresponds to the famous Sherrington-Kirkpatrick (SK) model. In this case, Aizenman-Lebowitz-Ruelle~\cite{ALR} proved that $NF_N(\beta)$ converges to a Gaussian random variable when $\beta <\beta_c=1$ and Talagrand \cite[Chapters 11 and 13]{Talbook2} obtained the moment control of Theorem \ref{momentcontrol}. For $p\geq 3,$  Bardina-M\'{a}rquez-Carreras-Rovira-Tindel \cite{Tindel} established \eqref{momentcontrol:eq1} for $\beta\ll\beta_c$. For even $p\geq 4,$ Bovier-Kurkova-L\"{o}we~\cite{BKL} showed that $N^{p/4+1/2}F_N(\beta)$ has a Gaussian fluctuation up to some temperature strictly less than $\beta_c .$ More recently, Chen \cite{Chen17} obtained the same statements as Theorem \ref{momentcontrol} and Proposition \ref{fluctuations} for this choice of $(\Lambda,\mu).$ Our main contribution here is to establish concentration of the overlap and the fluctuation of the free energy up to the critical temperature for any spin configurations sampled from a probability measure on a bounded subset of the real line.



\subsection{High Temperature Regime of the Vector-valued Model}\label{hd}

Next we consider the pure $p$-spin model with $k$-dimensional vector-valued spin configurations, where $k\geq 2.$ Recall the probability spaces $(\Lambda_1,\mu_1),\ldots,(\Lambda_k,\mu_k)$ from Section \ref{sec2.3}. Set the product space and measure by
\begin{align*}
\bar {\Lambda}&=\Lambda_1\times\cdots\times\Lambda_k, \\
\bar {\mu}&=\mu_1\otimes\cdots\otimes \mu_k.
\end{align*}
For $\sigma(r)\in \Lambda_r^{N}$, $1\leq r\leq k,$ denote
\begin{align*}
\bar {\s}_i &= (\sigma_i(1),\dots,\sigma_i(k))^T \in \bar {\Lambda},  \ 1 \leq i \leq N,\\
\bar \s& = (\bar \s_1,\dots, \bar \s_N) \in {\bar {\Lambda}}^N.
\end{align*} 
In other words, the spin configuration $\bar \s$ is a $k\times N$ matrix: the rows are $\s(1)\in \Lambda_1^N,\ldots,\s(k)\in \Lambda_k^N$ and the columns are $\bar {\s}_1,\ldots,\bar {\s}_N\in \bar {\Lambda}$. Given $\bar \beta=(\beta_1, \ldots, \beta_k)$ with $\beta_1,\ldots,\beta_k>0$, the re-centered pure $p$-spin Hamiltonian with vector-valued spin configurations is defined as
\begin{align*}
H_{N,\bar \beta}(\bar  \s) &= \sum_{r=1}^k \beta_r X_N(\s(r)) - \sum_{r,r'=1}^k \frac{\beta_r\beta_{r'}}{2}NR(\s(r),\s(r'))^p,\,\,\bar \s\in \Lambda^N.
\end{align*}
 Similar to the scalar-valued model, the free energy and the Gibbs measure are defined as
\begin{align}\label{add:fg}
F_{N}(\bar {\beta})  = \frac{1}{N} \log \int e^{H_{N,\bar \beta}(\bar  \s)}\bar \mu^{\otimes N} (d\bar \s) 
\end{align}
and
\begin{align*}
G_{N,\bar {\beta}}(d\bar \s)=\frac{e^{ H_{N,\bar \beta}(\bar \s)} {\bar \mu}^{\otimes N}(d\bar \s)}{Z_{N,\bar \beta}},
\end{align*}
where $Z_{N,\bar \beta}$ is the normalizing constant. Define
$$
F(\bar \beta)=\limsup_{N\rightarrow \infty}F_N(\bar \beta).
$$

There is a technical subtlety here that is not present in the previous subsection.
In the case of even $p,$  Panchenko \cite{Panchenko2015} proved that if one drops the overlap term in $H_{N,\bar \beta}$, then the limiting free energy with overlap constraint exists. Consequently, one can show that $F(\bar \beta)=\lim_{N\rightarrow\infty}F_N(\bar \beta)$ (see the proof of Proposition~\ref{new:parisiformula:k=1} below). When $p$ is odd, this limit is preserved if $k=1$, as explained in the previous subsection, but whether it is still true for $k\geq 2$ remains an open question.

An application of Jensen's inequality  ensures that $F(\bar \beta)\leq 0$. The high-temperature regime is defined as
\[
\bar {\mathcal{R}} = \bigl\{\bar  \beta=(\beta_1,\ldots,\beta_k) \mid \beta_r > 0 \text{ for all } 1 \leq r \leq k \text{ and } F(\bar \beta) = 0\bigr\}.
\]
Again, while $\bar \beta$ is understood as the vector of SNRs in the detection problem, we read the entries of this vector as the temperature parameters in the setting of spin glass models. Let $\beta_{r,c}$ be the critical temperature obtained from Section \ref{vec_spin_model} by taking $(\Lambda,\mu)=(\Lambda_r,\mu_r).$
The following theorem states that the high-temperature regime of the vector-valued $p$-spin model is equal to the product of the high-temperature regimes of the marginal systems.

\begin{theorem}\label{thm6}
	For $p \geq 3$, $\bar {\mathcal{R}}=(0,\beta_{1,c}]\times\cdots \times(0,\beta_{k,c}].$
\end{theorem}

Theorem \ref{thm6} highlights an interesting phenomenon: Although the Hamiltonian $H_{N,\bar \beta}$ involves interactions coming from the overlaps $R(\s(r),\s(r'))$ for all $r\neq r'$, in the high-temperature regime the marginal spin configurations $\s(1),\ldots,\s(k)$ under $H_{N,\bar \beta}$ essentially interact with each other independently. Consequently, they behave like $k$ independent one-dimensional systems associated to $H_{N,\beta_1},\ldots,H_{N,\beta_k}.$ As a result, the high-temperature regime of $H_{N,\bar \beta}$ is simply the product of  the high-temperature regimes of the marginal systems.

\section{Establishing Spike Detection}\label{sec4}

This section proves the main theorems of this paper. Section \ref{total_variation} first expresses the total variation distance that appears in the detection problem in terms of the free energy of the pure $p$-spin model. Using this expression and results described in Section \ref{sec3}, Sections \ref{main_th_iii}-\ref{sec4.3} conclude the proofs of Theorems \ref{thm1}-\ref{maintheorem2}.


\subsection{Total Variation Distance}\label{total_variation}

It is well-known that one can relate the total variation distance between two continuous random variables to the ratio of their probability densities. See for instance {\rm\cite[Lemma 1]{Chen17}}.

\begin{lemma} \label{dtv}
	If $U$ and $V$ are two $N$-dimensional random vectors with densities $f_U$ and $f_V$, respectively, and $f_U(x)$, $f_V(x) \neq 0$ a.e., then
	\[
	d_{TV}(U,V) = \int_0^1 \mathbb{P}\left(\frac{f_U(V)}{f_V(V)} < x \right)dx = \int_0^1 \mathbb{P}\left(\frac{f_U(U)}{f_V(U)} > \frac{1}{x} \right) dx.
	\]
\end{lemma}

Recall  $T$ and $T_k$ from \eqref{spike} and \eqref{spikes}. Note that $W$ is a symmetric Gaussian $p$-tensor and the spikes are independent of $W$. From these, one can compute the density functions for $W$, $T$, and $T_k$ explicitly and then apply Lemma \ref{dtv} to the pairs $(W,T)$ and $(W,T_k)$ to get

\begin{lemma}\label{dtv_free}
	For any $\beta \in (0,\infty)$ and $\bar  \beta \in (0,\infty)^k$,
	\begin{align}
	\begin{split}\label{dtv_free:Eq1}
	d_{TV}(W,T) = \int_0^1 \mathbb{P}\bigl(F_N(\beta) < N^{-1}\log x \bigr)dx,
	\end{split}\\
	\begin{split}\label{dtv_free:Eq2}
	d_{TV}(W,T_k) = \int_0^1 \mathbb{P}\bigl(F_N(\bar \beta) < N^{-1}\log x\bigr)dx.
	\end{split}
	\end{align}
\end{lemma}

For a detailed derivation, we refer the reader to \cite[Lemma 2]{Chen17}.

\subsection{Proof of Theorems \ref{thm1} and \ref{thm2.1}}\label{main_th_iii}


\begin{proof}[\bf Proof of Theorem \ref{thm1}] Let $\beta_c$ be the critical temperature defined in \eqref{critical}. Assume that $0<\beta<\beta_c$. From \eqref{dtv_free:Eq1}, using change of variable $y=-\log x$ and writing $\int_0^\infty=\int_0^\varepsilon+\int_\varepsilon^\infty$ imply that
	\begin{align*}
	d_{TV}(W,T)&=\Bigl(\int_0^{\epsilon} + \int_{\epsilon}^\infty\Bigr) \p\bigl(F_N(\beta)<-N^{-1}y\bigr)e^{-y}dy\\
	&\leq \int_0^{\epsilon}e^{-y}dy + \int_{\epsilon}^\infty\p\bigl(|F_N(\beta)|\geq N^{-1}y\bigr)e^{-y}dy\\
	&\leq \epsilon + \int_{\epsilon}^\infty\frac{K}{y^2N^{\frac{p}{2}-1}}e^{-y}dy\\
	&\leq \epsilon+\frac{K}{\epsilon N^{\frac{p}{2}-1}},\,\,\forall \varepsilon>0,
	\end{align*}
	where the second inequality used Proposition \ref{fluctuations}. Letting $\epsilon=N^{-(p-2)/4}$ yields
	\begin{align}\label{add:eq4}
	d_{TV}(W,T)\leq \frac{1+K}{N^{\frac{p-2}{4}}}.
	\end{align}
	This implies that $W$ and $T$ are indistinguishable, so detection is impossible. Next, assume that $\beta>\beta_c.$ Recall that $F_N(\beta)$ converges to $F(\beta)$ almost surely and note that $F(\beta)<0.$ It follows that
	\[
	\lim_{N \rightarrow \infty} \p\bigl(F_N(\beta) -N^{-1}\log x< 0 \bigr) = \p\left(F(\beta)<0 \right)=1
	\]
	and the dominated convergence theorem yields
	\[
	\lim_{N\rightarrow\infty}d_{TV}(W,T) = \int_0^1\p(F(\beta)< 0)dx=1.
	\]
\end{proof}


\begin{proof}[\bf Proof of Theorem \ref{thm2.1}] We have seen from the proof of Theorem \ref{thm1} that the critical temperature $\beta_c$ defined in \eqref{critical} is the critical threshold for detection. In addition, from Proposition \ref{prop1}, we see that $\beta_c$ satisfies $\sup_{v\in (0,v_*]}\Gamma_{\beta_c}(v)\leq 0$ and that any $\beta>0$ satisfying $\sup_{v\in (0,v_*]}\Gamma_{\beta}(v)\leq 0$ must also satisfy $\beta\leq \beta_c$. From this, to complete the proof, it suffices to show that $\sup_{v\in (0,v_*]}\Gamma_{\beta_c}(v)=0.$ If on the contrary $\sup_{v\in (0,v_*]}\Gamma_{\beta_c}(v)<0,$ then there exists some $\beta>\beta_c$ such that $\sup_{v\in (0,v_*]}\Gamma_{\beta}(v)<0$ since $\Gamma_{\beta}(v)$ is a continuous function in $\beta$ and $v$. This contradicts the fact that $\beta\leq \beta_c.$ 

\end{proof}

\subsection{Proof of Theorems \ref{maintheorem} and \ref{maintheorem2}}\label{sec4.3}

The proof of Theorem \ref{maintheorem} relies on the following simple lemma:

\begin{lemma}\label{lem2}
	Assume that $Y_1,Y_2,Y_3$ are random vectors of the same size and $Y_2$ is independent of $Y_1$ and $Y_3$. Then
	\begin{align*}
	d_{TV}(Y_1,Y_1+Y_2+Y_3)&\leq d_{TV}(Y_1,Y_1+Y_2)+d_{TV}(Y_1,Y_1+Y_3).
	\end{align*}
\end{lemma}

\begin{proof}
	The assertion follows immediately by using the triangle inequality,
	\begin{align*}
	d_{TV}(Y_1,Y_1+Y_2+Y_3)&\leq d_{TV}(Y_1,Y_1+Y_2)+d_{TV}(Y_1+Y_2,Y_1+Y_2+Y_3)
	\end{align*}
	and noting that the independence between  $Y_2$ and $Y_1,Y_3$ yields
	\begin{align*}
	d_{TV}(Y_1+Y_2,Y_1+Y_2+Y_3)&=\sup_{A}\Bigl|\e_{Y_2}\Bigl[\p(Y_1\in A-Y_2|Y_2)-\p(Y_1+Y_3\in A-Y_2|Y_2)\Bigr]\Bigr|\\
	&\leq \e_{Y_2}\Bigl[\sup_{A}\Bigl|\p(Y_1\in A-Y_2|Y_2)-\p(Y_1+Y_3\in A-Y_2|Y_2)\Bigr|\Bigr]\\
	&=\e_{Y_2}\Bigl[\sup_{A}\Bigl|\p(Y_1\in A)-\p(Y_1+Y_3\in A)\Bigr|\Bigr]\\
	&= d_{TV}(Y_1,Y_1+Y_3),
	\end{align*}
	where $\e_{Y_2}$ is the expectation with respect to $Y_2$ only.
\end{proof}

\begin{proof}[\bf Proof of Theorem \ref{maintheorem}] Let $\bar \beta = (\beta_1, \dots, \beta_k) \in (0,\beta_{1,c})\times\cdots\times (0,\beta_{r,c}).$ For $1\leq r\leq k,$ set
	$$
	T_{k,r}=W+\frac{\beta_r}{N^{(p-1)/2}}u(r)^{\otimes p}.
	$$
	From Lemma \ref{lem2} and an induction argument,
	\begin{align}\label{add:eq10}
	d_{TV}(W,T_k)&\leq \sum_{r=1}^kd_{TV}(W,T_{k,r}).
	\end{align}
	Since $\beta_r\in (0,\beta_{r,c})$,  \eqref{add:eq4} implies that there exists a constant $K_r>0$ such that for any $N\geq 1,$
	\begin{align}\label{add:eq5}
	d_{TV}(W,T_{k,r})\leq \frac{K_r}{N^{(p-2)/4}}.
	\end{align}
	This together with \eqref{add:eq10} implies that detection is impossible. Next, assume that $\bar \beta\notin (0,\beta_{1,c}]\times\cdots \times (0,\beta_{k,c}].$ Since $$\limsup_{N\rightarrow\infty}F_N(\bar \beta)=F(\bar \beta)<0,\,\,a.s.,$$ the Fatou lemma yields that for any $x>0,$
	\begin{align*}
	\liminf_{N\to\infty}\p(F_N(\bar \beta)<N^{-1}{\log x})&=\liminf_{N\to\infty}\e \bigl[I(F_N(\bar \beta)<N^{-1}\log x)\bigr]\\
	&\geq \e \bigl[\liminf_{N\to\infty}I(F_N(\bar \beta)<N^{-1}\log x)\bigr]=1,
	\end{align*}
	where $I(\cdot)$ is an indicator function. Using this, \eqref{dtv_free:Eq2}, and the Fatou lemma again, we arrive at
	\begin{align*}
	\liminf_{N\rightarrow\infty}d_{TV}(W,T_k)& =  \liminf_{N\rightarrow\infty}\int_0^1\p(F_N(\bar \beta)<N^{-1}{\log x})dx\\
	&\geq \int_0^1\liminf_{N\rightarrow\infty}\p(F_N(\bar \beta)<N^{-1}{\log x})dx=\int_0^1 1dx=1.
	\end{align*}
	Thus, detection is possible.
\end{proof}

\begin{proof}[\bf Proof of Theorem \ref{maintheorem2}]
	Assume that $\sup_{r\geq 1}\beta_r<\beta_c.$ From \eqref{add:eq10} and \eqref{add:eq5}, 
	\begin{align*}
	d_{TV}(W,T_k)&\leq \frac{Kk}{N^{(p-2)/4}},\,\,\forall N\geq 1,
	\end{align*}
	where $K$ is a universal constant independent of $N.$
	From the assumption on $k$, the right-hand side vanishes as $N$ tends to infinity and this establishes the assertion $(i)$. 
	
	Next, we establish $(ii)$ assuming $(a.1)$. Suppose that $N$ satisfies $k=k(N)>k_0.$ Denote $$
	\Delta_k=\frac{1}{N^{(p-1)/2}}\sum_{k_0< r\leq k}\beta_ru(r)^{\otimes p}.
	$$
	Since $(u(r):1\leq r\leq k_0)$ is independent of $(u(r):r>k_0)$, 
	\begin{align*}
	d_{TV}(T_k,T_{k_0})&\leq d_{TV}(W+\Delta_k,W)\leq \sum_{k_0<r\leq k}d_{TV}\bigl(W,T_{k,r}\bigr),
	\end{align*}
	where the second inequality used Lemma \ref{lem2} and $T_{k,r}$ is defined in the proof of Theorem \ref{maintheorem}. 
	Hence, from the triangular inequality and \eqref{add:eq5}, there exists a positive constant $K$ such that
	\begin{align*}
	d_{TV}(W,T_{k_0})&\leq d_{TV}(W,T_k)+d_{TV}(T_k,T_{k_0})\leq d_{TV}(W,T_k)+\frac{Kk}{N^{(p-1)/4}},\,\,\forall N\geq 1.
	\end{align*}
	Here, since $\beta_r>\beta_c$ for at least one $r\leq k_0,$ it means that $(\beta_1,\ldots,\beta_{k_0})\notin (0,\beta_{c}]\times\cdots(0,\beta_c]$ and from Theorem \ref{maintheorem}, $d_{TV}(W,T_{k_0})\to 1$. This implies that $d_{TV}(W,T_k)\to 1$ and the assertion $(ii)$ follows under $(a.1)$.
	
	To establish Theorem \ref{maintheorem2} assuming $(a.2)$, note that since $p$ is even, dropping the overlap terms in $H_{N,\bar \beta}(\bar \sigma)$ yields
    $F_N(\bar \beta)\leq \sum_{r=1}^k F_{N,r}(\beta_r).$ In addition, note that $\e F_{N,r}(\cdot)$ is the same function for every $r$ and it can be checked, by using  Gaussian integration by parts, that its derivative is uniformly bounded as long as $\beta$ stays in a bounded interval. Hence, for every $r\geq 1,$ $\e F_{N,r}(\cdot)$ is a sequence of equicontinuous functions. As a result, the assumption that $\inf_{r\geq k_0}\beta_r>\beta_c$ and $\sup_{r\geq k_0}\beta_r<\infty$ implies that there exist some $\delta>0$ and $N_0\geq 1$ such that $\e F_{N,r}(\beta_r)\leq -\delta$ for all $r\geq k_0$ and $N\geq N_0.$ On the other hand, the Gaussian concentration inequality implies that there exists a universal constant $K>0$ such that
    \begin{align*}
    \p\bigl(\Omega_{N,r,t}\bigr)&\leq Ke^{-t^2N/K},\,\,\forall N\geq 1,r\geq 1,t>0,
    \end{align*}
    where $\Omega_{N,r,t}:=\bigl\{\bigl|F_{N,r}(\beta_r)-\e F_{N,r}(\beta_r)\bigr|\geq t\bigr\}$.
    Fix $0<t\leq \delta/2$. From these, the probability of the event $\cap_{1\leq r\leq k}\Omega_{N,r,t}^c$ is at least $1-kKe^{-t^2N/K}$. Furthermore, as long as $N\geq N_0$ satisfies $k=k(N)>k_0$, on this event,
    \begin{align*}
    \sum_{r=1}^k F_{N,r}(\beta_r)&=\sum_{r=1}^k \bigl(F_{N,r}(\beta_r)-\e F_{N,r}(\beta_r)\bigr)+\sum_{r=1}^k \e F_{N,r}(\beta_r)\\
    &\leq \frac{\delta k}{2}-\sum_{k_0<r\leq k}\delta=\frac{k_0\delta}{2}-\frac{1}{2}\sum_{k_0< r\leq k}\delta,
    \end{align*}
    where the first inequality used $\e F_{N,r}(\beta_r)\leq 0.$
    Hence, from the assumption on $k$ and the Borel-Cantelli lemma, $\limsup_{N\to\infty}F_N(\bar \beta)=\-\infty$ a.s. and this implies that for all $x\in (0,1)$,
    \begin{align*}
    \lim_{N\to\infty}\p\bigl(F_N(\bar \beta)\leq N^{-1}\log x\bigr)=1.
    \end{align*}
    From \eqref{dtv_free:Eq2}, the assertion $(ii)$ follows.
\end{proof}


\section{Establishing Spike Recovery}\label{sub4.4}

We present the proof of Theorem \ref{thm0} in this section. Recall that we handled the detection problem by means of the free energies of the spin glass models defined in Section \ref{sec3}. Our treatment for Theorem~\ref{thm0} will also rely on an auxiliary spin glass model, which arises naturally from the conditional distribution of $u(1),\ldots,u(r)$ given $T_k.$ This allows us to establish the so-called Nishimori identity and connect the MMSE to the free energy associated to this auxiliary spin system.


\subsection{Nishimori Identity} \label{Nishimori}
Recall the probability spaces $(\Lambda_r,\mu_r)$, the product probability space $(\bar \Lambda,\bar \mu)$, and the Hamiltonians $H_{N,\bar \beta}(\bar \s)$ from Section \ref{hd}. Fix a SNR vector $\bar \beta$. For any $t\geq 0,$ define the random tensor $T_k(t)$ by
\begin{align*}
T_k(t)=W+\sqrt{\frac{t}{N^{p-1}}}\sum_{r=1}^k\beta_ru(r)^{\otimes p}.
\end{align*}
For $\bar \s\in \bar \Lambda^N,$ define the auxiliary Hamiltonian
\begin{align}\label{pham}
&H_{N,t,\bar{\beta}}^A(\bar \sigma)=\frac{\sqrt{t}}{N^{(p-1)/2}}\sum_{r=1}^k\beta_r\bigl\la T_k(t),\sigma(r)^{\otimes p}\bigr\ra-\frac{t}{2}\sum_{r,r'=1}^k\beta_r\beta_{r'}NR(\s(r),\s(r'))^p.
\end{align}
For $t\geq 0$, define the auxiliary free energy and Gibbs measure by
\begin{align*}
F_{N}^A(t)&=\frac{1}{N}\log\int e^{H_{N,t}^A(\bar \s)}\bar \mu^{\otimes N}(d\bar \s)
\,\,\mbox{and}\,\,
G_{N,t}^A(d\bar \s)=\frac{e^{H_{N,t}^A(\bar \s)}\bar \mu^{\otimes N}(d\bar \s)}{\int e^{H_{N,t}^A(\bar \s')}\bar \mu^{\otimes N}(d\bar \s')}.
\end{align*}
Denote by $\bar {\sigma}^{1},\bar {\sigma}^2,\ldots$ the i.i.d. samplings from $G_{N,t}^A$ and by $\la \cdot\ra_t^A$ the Gibbs expectation with respect to $G_{N,t}^A.$  

A key observation here is that the distribution of $(u(1),\ldots,u(k))$ conditionally on $T_k(t)$ is described by the Gibbs measure,
\begin{align}\label{add:eq-5}
\p\bigl((u(1),\ldots,u(k))\in \cdot|T_k(t)\bigr)&=G_{N,t}^A(\cdot).
\end{align}
To see this equation, one uses the fact that $W$ is Gaussian and is independent of $u(r)$'s and then express the joint density of $u(1),\ldots,u(k),T_k(t)$ in terms of the Gaussian density, see, e.g., \cite{Chen17}. As a consequence, \eqref{add:eq-5} implies the so-called Nishimori identity, namely,
\begin{align}\label{Nish}
\la f(\bar {\sigma}^1,\ldots,\bar {\sigma}^n,u(1),\ldots,u(k))\ra_t^A= \la f(\bar {\sigma}^1,\ldots,\bar {\sigma}^{n},\bar {\sigma}^{n+1})\ra_t^A
\end{align}
for any bounded measurable function $f$. One may find more general settings, for instance, in \cite{LMLKZ+17}.

Consider the following auxiliary minimum mean square error
\begin{align*}
\MMSE_N^A(\bar \beta, t):=\min_{\hat\theta}\frac{1}{N^p}\sum_{i_1,\ldots,i_p=1}^N\e \Bigl(\sum_{r=1}^k\beta_ru_{i_1}(r)\cdots u_{i_p}(r)-\hat\theta_{i_1,\ldots,i_p}\Bigr)^2,
\end{align*}
where the minimum is taken over all $\mathbb{R}^{N^p}$-valued bounded random variables $\hat\theta=(\hat\theta_{i_1,\ldots,i_p})$ that are generated by the $\sigma$-field $\sigma(T_k(t))$ and are allowed to depend on other randomness independent of both $u_i(r)$'s and $T_k$. The following lemma summarizes some key properties of $\e F_{N}^A(t)$ and relates the auxiliary minimum mean square error to the derivative of the free energy. These were originally discovered in \cite{GWSV,WV}. For completeness, we present their proofs here.

\begin{lemma}\label{add:lem4}
	The following statements hold:
	\begin{itemize}
		\item [$(i)$] $\e F_N^A(t)$ is a nondecreasing, nonnegative, and convex function of $t$.
		\item [$(ii)$] $\frac{d}{dt}\e F_N^A( t)=\frac{1}{2}\sum_{r,r'=1}^k\beta_r\beta_{r'}\e \bigl\la R(\s(r),u(r'))^p\bigr\ra_{t}^A.$
		\item[$(iii)$] $\MMSE_N^A(\bar \beta, t)=\sum_{r,r'=1}^k\beta_r\beta_{r'}\e R(u(r),u(r'))^p-2\frac{d}{dt}\e F_N^A( t).$
	\end{itemize}	
\end{lemma}

\begin{proof}
	Using Gaussian integration by parts implies
	\begin{align*}
	\frac{d}{dt}\e F_N^A(t)
	&=\sum_{r,r'=1}^k\beta_r\beta_{r'}\Bigl(-\frac{1}{2}\e \bigl\la R(\s^1(r),\s^2(r'))^p\bigr\ra_{t}^A+\e \bigl\la R(\s(r),u(r'))^p\bigr\ra_{t}^A\Bigr).
	\end{align*}
	From \eqref{Nish}, $(ii)$ follows. To establish $(iii)$, note that the minimizer of $\MMSE_N^A$ is attained by the estimator $$
	\hat\theta_{i_1,\ldots,i_p}^A=\sum_{r=1}^k\beta_r \e\bigl[u_{i_1}(r)\cdots u_{i_p}(r)|T_k(t)\bigr]=\sum_{r=1}^k\beta_r\bigl\la \s_{i_1}\cdots \s_{i_p}\bigr\ra_{t}^A,
	$$
	where the second equality used \eqref{Nish}.
	Plugging this estimator into $\MMSE_N^A(\bar \beta, t)$ and applying \eqref{Nish} yield $(iii).$ 
	
	Finally, we prove $(i).$
	Note that setting $\hat\theta_{i_1,\ldots,i_p}\equiv 0$ gives the upper bound
\[
\MMSE^A_N(\bar \beta, t) \leq \sum_{r,r'=1}^k\beta_r\beta_{r'}\e R(u(r),u(r'))^p.
\]
Combining this with $(iii)$ shows that $\frac{d}{dt}\e F_N^A(t)$  is nonnegative, so $\E F_N^A(t)$ is non-decreasing in $t$. In addition, since $F_N^A(0)=0$, we conclude that $\e F_N^A(t)$ is nonnegative. To establish the convexity of $\e F_N^A$ in $t$, from $(iii)$ it suffices to show that $\MMSE_N^A(\bar \beta, t)$ is nonincreasing in $t$. For any $0\leq t< t'$, write
	\begin{align*}
	\frac{1}{\sqrt{t}}T_k(t)&=\frac{1}{\sqrt{N^{p-1}}}\sum_{r=1}^k\beta_ru^{\otimes p}(r)+\frac{1}{\sqrt{t}}W\stackrel{d}{=}\frac{1}{\sqrt{t'}}T_k(t')+\sqrt{\frac{1}{t}-\frac{1}{t'}}W',
	\end{align*}
	where $W'$ is an independent copy of $W$ and is also independent of $u(1),\ldots,u(k).$ Write
	\begin{align*}
	\e[u_{i_1}\cdots u_{i_p}|T_k(t')]&=\e[u_{i_1}\cdots u_{i_p}|T_k(t'),W']=\e[u_{i_1}\cdots u_{i_p}|T_k(t),T_k(t')].
	\end{align*}
	It follows that
	\begin{align*}
	\MMSE_N^A(\bar \beta, t')
	&=\frac{1}{N^p}\sum_{i_1,\ldots,i_p=1}^N\e\Bigl(\sum_{r=1}^k\beta_r\bigl(u_{i_1}\cdots u_{i_p}-\e[u_{i_1}\cdots u_{i_p}|T_k(t),T_k(t')]\bigr)\Bigr)^2 \\
	&\leq\frac{1}{N^p}\sum_{i_1,\ldots,i_p=1}^N\e\Bigl(\sum_{r=1}^k\beta_r\bigl(u_{i_1}\cdots u_{i_p}-\e[u_{i_1}\cdots u_{i_p}|T_k(t)]\bigr)\Bigr)^2=\MMSE_N^A(\bar \beta, t).
	\end{align*}
	This establishes $(i)$ and completes our proof.
\end{proof}

\subsection{Proof of Theorem \ref{thm0}}
    	We prove Theorem \ref{thm0} $(i)$ first.  Assume that $\bar \beta\in (0,\beta_{1,c})\times\cdots\times (0,\beta_{k,c}).$ From Theorem~\ref{maintheorem}, $d_{TV}(W,T_k)\rightarrow 0$. Note that
	\begin{align*}
	d_{TV}(W,T_k) = \int_0^1 \mathbb{P}\bigl(F_N(\bar \beta) < N^{-1}\log x \bigr)dx=\int_0^1 \mathbb{P}\bigl(F_N^A(1) > -N^{-1}\log x \bigr)dx.
	\end{align*}
	Here the first equality is from Lemma \ref{dtv_free}, while the second equality follows from a similar argument as that for Lemma \ref{dtv_free} by using the second equality in Lemma \ref{dtv}.	
 By Fatou's lemma and the above display, $$\liminf_{N\rightarrow\infty}\mathbb{P}\bigl(F_N^A(1) > -N^{-1}\log x \bigr)\rightarrow 0,\,\,x\in (0,1)$$
	and consequently, $\limsup_{N\rightarrow\infty}\mathbb{P}(B_N(\varepsilon))=1$ for all $\varepsilon\in (0,1)$, where $B_N(\varepsilon):=\{F_N^A(1) \leq \varepsilon\}.$ Therefore, from H\"{o}lder's inequality,
	\begin{align*}
	\e F_N^A(1)&=\e [F_N^A(1);B_N(\varepsilon)]+\e [F_N^A(1);B_N(\varepsilon)^c]\\
	&\leq \varepsilon+\bigl(\e F_N^A(1)^2\bigr)^{1/2}\bigl(\p(B_N(\varepsilon)^c)\bigr)^{1/2}.
	\end{align*}
	Note that since $\mu_1,\ldots,\mu_k$ are defined on bounded sets, one can verify that the second moment of the random variable $F_N^A(1)$ is bounded in $N.$ As a result,
	$$\limsup_{N\to\infty}\e F_N^A(1)\leq 0.$$
	From Lemma \ref{add:lem4} $(i)$, we then conclude that
	$$\lim_{N\rightarrow\infty} \E F_N^A(t)=0,\,\,\forall t\in [0,1].$$
	Now the convexity of $\E F_N^A$ implies that $\lim_{N\to\infty}\frac{d}{dt}\e F_N^A(t)=0$ for $t\in [0,1]$. From Lemma~\ref{add:lem4} $(ii)$ and $(iii)$ and the strong law of large numbers, Theorem \ref{thm0} $(i)$ follows.
	
	Next, we assume that $\bar \beta\not\in (0,\beta_{1,c}]\times \cdots\times(0,\beta_{k,c}].$  For $s\in [0,1]$, define an interpolating free energy by
	\begin{align*}
	F_{N}^I(s)&=\frac{1}{N}\log \int \exp\Bigl(H_{N,\bar \beta}(\bar \s)+s\sum_{r,r'=1}^k\beta_r\beta_{r'}NR(\sigma(r),u(r'))^p\Bigr)\bar \mu^{\otimes N}(d\bar \s).
	\end{align*}
	Note that when $t=1$,
	\begin{align*}
	H_{N,t}^A(\bar \s)&=H_{N,\bar \beta}(\bar \s)+\sum_{r,r'=1}^k\beta_r\beta_{r'}NR(\sigma(r),u(r'))^p.
	\end{align*}
	This implies that $F_N^I(1)=F_N^A(1)$ and $F_N^I(0)=F_N(\bar \beta).$ In addition, from Lemma \ref{add:lem4} $(ii)$, and the convexity of $F_N^I(s)$,
	\begin{align}\label{add:eq14}
	\frac{d}{dt}\e F_N^A(1)=\frac{1}{2}\frac{d}{ds}\e F_N^I(1)\geq \frac{1}{2}\frac{d}{ds}\e F_N^I(s),\,\,\forall s\in[0,1].
	\end{align}
	Note that since $\e F_N^I$ is a family of  equicontinuous and convex functions, one can pass to a subsequence $(N_n)_{n\geq 1}$ via a diagonalization procedure to show that $\e F_N^I$ is pointwise convergent along this subsequence. Furthermore, we can ensure that along this subsequence, $$\lim_{n\rightarrow\infty}\frac{d}{dt}\e F_{N_n}^I(1)=\liminf_{N\to \infty}\frac{d}{dt}\e F_{N}^A(1).$$ Denote $F^I=\lim_{n\to \infty}\e F_{N_n}^I.$
	Note that on the one hand, $\e F_N^A(1)\geq 0$ by Lemma~\ref{add:lem4} $(i)$ and on the other hand, $F(\bar \beta)=\limsup_{N\to\infty}\e F_N(\bar \beta)<0$ by Theorem~\ref{thm6}. Using the identities $F_N^I(1)=F_N^A(1)$ and  $F_N^I(0)=F_N(\bar \beta)$ yields that
    $
	F^I(0)< 0\leq F^I(1).
	$
	Consequently,	there exists some $s_0\in (0,1)$ such that $F^I$ is differentiable at this point and 
	\begin{align*}
	\lim_{n\rightarrow\infty}\frac{d}{ds}\e F_{N_n}^I(s_0)&=\frac{d}{ds}F^I(s_0)>0.
	\end{align*}
	This and \eqref{add:eq14} together yield
	\begin{align*}
	\lim_{n\rightarrow\infty}\frac{d}{dt}\e F_{N_n}^A(1)\geq  \frac{1}{2}\frac{d}{ds}F^I(s_0)>0.
	\end{align*}
	Finally, from this inequality, Lemma \ref{add:lem4} $(ii)$ and $(iii)$, and the strong law of large numbers, the assertion of Theorem \ref{thm0} $(ii)$ follows.


\section{Structure of the Regime $\RR$}\label{sec:high_temp_k1}

In this section, we establish the proof of Proposition \ref{prop1}. It is based on a subtle control of the Parisi formula for the free energy. While a similar argument has appeared in \cite{Chen17} for the case that there is only one spike and it is sampled from the Rademacher prior, our argument here works for more general priors. 


\subsection{The Parisi Formula}\label{sec5.1}

Recall the probability space $(\Lambda,\mu)$ from Section \ref{vec_spin_model}. Denote
\begin{align}\label{V}
\mathcal{V}=\{u^2:u\in \Lambda\}.
\end{align}
Fix $v\in \mathcal{V}$ and let $\mathcal{M}_v$ be the space of probability measures on $[0,v]$. Recall that $\xi(s)=s^p.$ For $\alpha\in \mathcal{M}_v$ and $\lambda\in \mathbb{R}$, define the Parisi functional by
\begin{align*}
\mathcal{P}_{\beta,v}(\alpha,\lambda)&=\Phi_{\beta,v,\alpha}(0,0,\lambda)-\lambda v-\frac{\beta^2}{2}\int_0^v\alpha(s)\xi''(s)sds,
\end{align*}
where $\Phi_{\beta,v,\alpha}(0,0,\lambda)$ is defined as the weak solution of the following PDE on $[0,v] \times \R \times \R$ (see \cite{JT2}):
\begin{align*}
\partial_s\Phi_{\beta,v,\alpha}&=-\frac{\beta^2\xi''}{2}\bigl(\partial_{xx}\Phi_{\beta,v,\alpha}+\alpha \bigl(\partial_x\Phi_{\beta,v,\alpha}\bigr)^2\bigr)
\end{align*}
with the boundary condition
\begin{align*}
\Phi_{\beta,v,\alpha}(v,x,\lambda)&=\log \int e^{xa+\lambda a^2}\mu(da).
\end{align*}
The Parisi formula states that
\begin{align*}
\lim_{N\rightarrow\infty}\frac{1}{N}\log \int e^{\beta X_N(\s)}\mu^{\otimes N}(d\s)=\sup_{v\in \mathcal{V}}\inf_{(\alpha,\lambda)\in \mathcal{M}_v\times\mathbb{R}}\mathcal{P}_{\beta,v}(\alpha,\lambda).
\end{align*}
This formula was initially established by Talagrand \cite{Tal03} for the mixture of even $p$-spin Hamiltonians and $\Lambda=\{-1,1\}.$ Later it was generalized to arbitrary mixtures of pure $p$-spin Hamiltonians including odd $p$ and any probability space $(\Lambda,\mu)$ with bounded $\Lambda\subset \mathbb{R}$ by Panchenko \cite{Panchenko05,Panchenko2015}. 
The following proposition shows that the limiting free energy $F(\beta)$ can also be expressed as a Parisi-type formula.

\begin{prop}[Parisi formula] \label{new:parisiformula:k=1} For any $\beta>0,$
	\begin{align*}
	F(\beta)&:=\lim_{N\rightarrow\infty}F_N(\beta)=\sup_{v\in \mathcal{V}}\inf_{\alpha,\lambda}\mathcal{Q}_{\beta,v}(\alpha,\lambda),
	\end{align*}
	where for $(\alpha,\lambda)\in \mathcal{M}_v\times \mathbb{R},$
	$$
	\mathcal{Q}_{\beta,v}(\alpha,\lambda):=\mathcal{P}_{\beta,v}(\alpha,\lambda)-\frac{\beta^2v^p}{2}.
	$$
\end{prop}

\begin{proof} For any measurable $A\subset \mathcal{V}$, define the free energy restricted to $A$ by
	\begin{align*}
	F_N(\beta,A)&=\frac{1}{N}\log \int_{R(\s,\s)\in A}e^{H_{N,\beta}(\s)}\mu^{\otimes N}(d\s).
	\end{align*}
	For any $\eta>0$ and $v\in \mathcal{V}$, set $A_\eta(v)=(v-\eta,v+\eta).$
	Note that it is already known from \cite{Panchenko2015} that for any $v\in \mathcal{V}$,
	\begin{align*}
	\lim_{\eta\downarrow 0}\lim_{N\rightarrow\infty}F_N\bigl(\beta,A_\eta(v)\bigr)&=\inf_{\mathcal{M}_v\times \mathbb{R}}\mathcal{Q}_{\beta,v}(\alpha,\lambda).
	\end{align*}
	From this, for any $\delta>0,$ there exist $\eta(v)$ and $N(v)$ such that for any $N\geq N(v)$
	\begin{align}\label{new:prop1:proof:eq1}
	\Bigl|F_N\bigl(\beta,A_\eta(v)\bigr)- \inf_{\mathcal{M}_v\times \mathbb{R}}\mathcal{Q}_{\beta,v}(\alpha,\lambda)\Bigr|\leq\delta.
	\end{align}
	Note that $\mathcal{V}$ is bounded and that for any $\eta>0,$ $(A_\eta(v):v \in \mathcal{V})$ forms an open covering for the closure of $\mathcal{V}$. From these, we can pass to a finite covering,  $A_\eta(v_j)$ for $1\leq j\leq n$, such that \eqref{new:prop1:proof:eq1} is valid. From this,
	\begin{align*}
	F_N(\beta,A_\eta(v_j))
	\leq F_N(\beta)
	\leq \frac{1}{N}\log \sum_{j=1}^n \exp NF_N(\beta,A_\eta(v_j))
	\leq \frac{\log n}{N}+\max_{1\leq j\leq n}F_N(\beta,A_\eta(v_j))
	\end{align*}
	and hence, as long as $N$ is large enough,
	\begin{align*}
	\inf_{\mathcal{M}_v\times \mathbb{R}}\mathcal{Q}_{\beta,v}(\alpha,\lambda)-\delta\leq F_N(\beta)&\leq \max_{1\leq j\leq n}\inf_{\mathcal{M}_{v_j}\times \mathbb{R}}\mathcal{Q}_{\beta,v_j}(\alpha,\lambda)+2\delta.
	\end{align*}
	Thus,
	\begin{align*}
	\max_{1\leq j\leq n}\inf_{\mathcal{M}_{v_j}\times \mathbb{R}}\mathcal{Q}_{\beta,v_j}(\alpha,\lambda)-\delta&\leq \liminf_{N\rightarrow\infty}F_N(\beta)\\
	&
	\leq \limsup_{N\rightarrow\infty}F_N(\beta)\leq \max_{1\leq j\leq n}\inf_{\mathcal{M}_{v_j}\times \mathbb{R}}\mathcal{Q}_{\beta,v_j}(\alpha,\lambda)+2\delta.
	\end{align*}
	This completes our proof by letting $\delta\downarrow 0$ and noting that $\inf_{\mathcal{M}_v\times \mathbb{R}}\mathcal{Q}_{\beta,v}(\alpha\lambda)$ is continuous in $v$.
\end{proof}

\subsection{Two Technical Lemmas}\label{sec5.2}
Recall $\Gamma_b$ from \eqref{add:eq1} and $\gamma_b$ from \eqref{auxfun}. The following technical inequality establishes the strict monotonicity of $\gamma_b$ in the temperature parameter $b$. This will be of great importance for the rest of this section as well as in Section \ref{sec6}.  

\begin{lemma}\label{lem3}
	If $0<\beta<\beta'$, then $\gamma_{\beta}(s)<\gamma_{\beta'}(s)$ for all $s>0.$
\end{lemma}

\begin{proof}
	Note that $\gamma_\beta(0)=0.$
	Let $B_t$ be a standard Brownian motion. Define
	\begin{align*}
	g_j(t,x)&=\int a^je^{ax-\frac{a^2t}{2}}\mu(da),\,\,\forall j=0,1,2,3.
	\end{align*}
	Set $X_t=g_1(t,B_t)^2$ and $Y_t=g_0(t,B_t)^{-1}.$ Note that $\gamma_\beta(s)=\e X_tY_t$ if we let $t=\beta^2\xi'(s).$ From It\^{o}'s formula,
	\begin{align*}
	dX_t&=2g_1 \partial_tg_1 dt+2g_1 \partial_xg_1 dB_t+\Bigl(g_1 \partial_{xx}g_1 +\bigl(\partial_xg_1 \bigr)^2\Bigr)dt\\
	&=-g_1 g_3 dt+2g_1 g_2 dB_t+\Bigl(g_1 g_3 +\bigl(g_2 \bigr)^2\Bigr)dt=\bigl(g_2 \bigr)^2dt+2g_1 g_2 dB_t
	\end{align*}
	and
	\begin{align*}
	dY_t&=-\frac{\partial_tg_0 }{g_0 ^2}dt-\frac{\partial_xg_0 }{g_0 ^2}dB_t-\frac{1}{2}\Bigl(\frac{\partial_{xx}g_0 }{g_0 ^2}-\frac{2\bigl(\partial_xg_0 \bigr)^2}{g_0 ^3}\Bigr)dt\\
	&=\frac{g_2 }{2g_0 ^2}dt-\frac{g_1 }{g_0 ^2}dB_t-\frac{1}{2}\Bigl(\frac{g_2 }{g_0 ^2}-\frac{2\bigl(g_1 \bigr)^2}{g_0 ^3}\Bigr)dt=\frac{\bigl(g_1 \bigr)^2}{g_0 ^3}dt-\frac{g_1 }{g_0 ^2}dB_t.
	\end{align*}
	Now from the product rule,
	\begin{align*}
	d(X_tY_t)&=X_tdY_t+Y_tdX_t+d\bigl\< X_t,Y_t\bigr\>\\
	&=g_1^2\Bigl(\frac{g_1^2}{g_0^3}dt-\frac{g_1}{g_0^2}dB_t\Bigr)+g_0^{-1}\Bigl(\bigl(g_2\bigr)^2dt+2g_1g_2dB_t\Bigr)-\frac{2g_1^2g_2}{g_0^2}dt\\
	&=\Bigl(\frac{g_1^4}{g_0^3}+\frac{g_2^2}{g_0}-\frac{2g_1^2g_2}{g_0^2}\Bigr)dt+\Bigl(-\frac{g_1^3}{g_0^2}+\frac{2g_1g_2}{g_0}\Bigr)dB_t\\
	&=g_0\Bigl(\frac{g_1^2}{g_0^2}-\frac{g_2}{g_0}\Bigr)^2dt+\Bigl(-\frac{g_1^3}{g_0^2}+\frac{2g_1g_2}{g_0}\Bigr)dB_t,
	\end{align*}
	where $\la\cdot,\cdot\ra$ is the quadratic variation and the third equality follows from the identity
	\begin{align*}
	 \frac{g_1^4}{g_0^3}+\frac{g_2^2}{g_0}-\frac{2g_1^2g_2}{g_0^2}&=g_0\bigl(\frac{g_1^4}{g_0^4}+\frac{g_2^2}{g_0^2}-2\frac{g_1^2}{g_0^2}\frac{g_2}{g_0}\bigr)=g_0\Bigl(\frac{g_1^2}{g_0^2}-\frac{g_2}{g_0}\Bigr)^2.
	\end{align*}
	From this, we conclude that $X_tY_t$ is a submartingale and thus $\e X_tY_t\leq \e X_{t'}Y_{t'}$ for any $0\leq t< t'.$
	
If equality holds for some $0\leq t<t',$ then
	\begin{align*}
	\int_{t}^{t'}\e\Bigl[g_0(s,B_s)\Bigl(\frac{g_1(s,B_s)^2}{g_0(s,B_s)^2}-\frac{g_2(s,B_s)}{g_0(s,B_s)}\Bigr)^2\Bigr]ds=\e X_{t'}Y_{t'}-\e X_tY_t=0.
	\end{align*}
	This implies that
	\begin{align*}
	\left(\frac{\int_\Lambda ae^{aB_s-\frac{a^2s}{2}}\mu(da)}{\int_\Lambda e^{aB_s-\frac{a^2s}{2}}\mu(da)}\right)^2=\frac{g_1(s,B_s)^2}{g_0(s,B_s)^2}=\frac{g_2(s,B_s)}{g_0(s,B_s)}=\frac{\int_\Lambda a^2e^{aB_s-\frac{a^2s}{2}}\mu(da)}{\int_\Lambda e^{aB_s-\frac{a^2s}{2}}\mu(da)}
	\end{align*}
	for all $t\leq s\leq t'.$ From this, the necessary condition for obtaining equality in Jensen's inequality implies that
	\begin{align*}
	a'=\frac{\int_\Lambda ae^{aB_s-\frac{a^2s}{2}}\mu(da)}{\int_\Lambda e^{aB_s-\frac{a^2s}{2}}\mu(da)},\,\,\forall a'\in\Lambda.
	\end{align*}
	The above equation implies that $\Lambda$ consists of a single element, which contradicts the assumption that $\mu$ is centered and $\Lambda$ contains more than one element. Therefore, $EX_tY_t<EX_{t'}Y_{t'}$ for any $0\leq t< t'.$ Finally, for $s>0$ and $0\leq \beta<\beta'$,  plugging $t=\beta^2\xi'(s)$ and $t'={\beta'}^2\xi'(s)$ into this inequality yields
	$
	\gamma_\beta(s)<\gamma_{\beta'}(s).
	$
\end{proof}

Recall the constant $v_*$ from \eqref{q1}. Set the parameter $$\lambda_*=-\frac{\beta^2\xi'(v_*)}{2}=-\frac{\beta^2 pv_*^{p-1}}{2}.$$ 
Recall the Parisi formula from Proposition \ref{new:parisiformula:k=1}. For any $v\in [0,v_*],$ define $\alpha_v\in \mathcal{M}_v$ by $\alpha_v(s)=1$ for $s\in [0,v].$ The next lemma studies some variational properties of the functional $\mathcal{Q}_{\beta,v}$ defined in Proposition \ref{new:parisiformula:k=1}.

\begin{lemma} \label{new:lem3}
	The following two statements hold:
	\begin{itemize}
		\item[$(i)$] If $v\neq v_*$, then $\inf_{\lambda}\mathcal{Q}_{\beta,v}(\alpha_v,\lambda)<0.$
		\item[$(ii)$] If $v= v_*$, then $\inf_{\lambda}\mathcal{Q}_{\beta,v}(\alpha_v,\lambda)=0$ and $\lambda_*$ is a minimizer.
	\end{itemize}
\end{lemma}

\begin{proof}
	Note that
	\begin{align*}
	\mathcal{P}_{\beta,v}(\alpha_v,\lambda)&=\log \e \int\exp \Bigl(\beta z\sqrt{\xi'(v)}a+\lambda a^2\Bigr)\mu(da)-\lambda v-\frac{\beta^2}{2}\int_0^v\xi''(s)sds\\
	&=\log \int \exp \Bigl(\frac{\beta^2pv^{p-1}}{2}a^2+\lambda a^2\Bigr)\mu(da)-\lambda v-\frac{\beta^2(p-1)}{2}v^p,
	\end{align*}
	where $z$ is standard Gaussian. Consequently,
	\begin{align*}
	\inf_{\lambda}\mathcal{Q}_{\beta,v}(\alpha_v,\lambda)
	&=\inf_{\lambda}\Bigl(\log \int \exp\bigl(\lambda a^2\bigr)\mu(da)-\lambda v\Bigr),
	\end{align*}
	where the right-hand side is obtained through a change of variable $\lambda \mapsto \lambda-\beta^2  \xi'(v)/2$. Define
	\begin{align*}
	F(v,\lambda)&=\log \int \exp\bigl(\lambda a^2\bigr)\mu(da)-\lambda v.
	\end{align*}
	Note that H\"older's inequality implies that $F(v,\cdot)$ is convex. If $v=v_*,$ then $\partial_\lambda F(v,0)=0$ and thus $\lambda = 0$ is a minimizer of $F(v,\cdot).$ Recalling the substitution $\lambda \mapsto \lambda - \beta^2\xi'(v)/2$, this means that property $(ii)$ holds. To show $(i)$, note that $F(v,0)=0.$ If $v\neq v_*$, then $\partial_\lambda F(v,0)=v_*-v\neq 0$. This means that $0$ is not a minimizer of $F(v,\cdot)$ and therefore $\inf_\lambda F(v,\lambda)<0.$ This implies property $(i)$ and completes our proof.
\end{proof}

\subsection{Proof of Proposition \ref{prop1}}\label{sec5.3}

First, we prove that for $\beta>0$, $\beta\in \mathcal{R}$ if and only if $\sup_{v\in (0,v_*]}\Gamma_\beta(v)\leq 0.$ Let $\beta \in \RR$. From Proposition \ref{new:parisiformula:k=1}, 
\[
0=F(\beta)=\sup_v\inf_{\alpha,\lambda}\mathcal{Q}_{\beta,v}(\alpha,\lambda).
\]
From Lemma \ref{new:lem3} $(i)$, we see that for any $v\neq v_*,$
\begin{align*}
\inf_{\alpha,\lambda}\mathcal{Q}_{\beta,v}(\alpha,\lambda)\leq \inf_{\lambda}\mathcal{Q}_{\beta,v}(\alpha_v,\lambda)<0,
\end{align*}
which implies that
\[
\sup_v\inf_{\alpha,\lambda}\mathcal{Q}_{\beta,v}(\alpha,\lambda) = \inf_{\alpha,\lambda}\mathcal{Q}_{\beta,v_*}(\alpha,\lambda) = 0.
\]
From this and Lemma \ref{new:lem3} $(ii)$, we conclude that $(\alpha_{v_*},\lambda_*)$ is an optimizer of $\mathcal{Q}_{\beta,v_*}$. Now we use this conclusion to show that $\beta$ must satisfy $\sup_{v\in (0,v_*]}\Gamma_\beta(v)\leq 0$ as follows. Note that
\begin{align*}
\mathcal{Q}_{\beta,v_*}(\alpha,\lambda_*)&=\Phi_{\beta,v_*,\alpha}(0,0,\lambda_*)-\frac{\beta^2}{2}\int_0^{v_*}\alpha(s)\xi''(s)ds+\frac{\beta^2(p-1)v_*^p}{2}.
\end{align*}
Since the boundary condition $\Phi_{\beta,v_*,\alpha}(v_*,x,\lambda)$ is convex in $(x,\lambda),$ an argument identical to that in \cite{AC14} yields that $(\alpha,\lambda)\in \mathcal{M}_{v_*}\times\mathbb{R}\mapsto \mathcal{Q}_{\beta,v_*}(\alpha,\lambda)$ is a convex functional. For any $(\alpha,\lambda)\in \mathcal{M}_{v_*}\times\mathbb{R}$ and $\theta\in [0,1]$, set
\begin{align*}
\alpha_\theta&:=(1-\theta)\alpha_{v_*}+\theta\alpha\,\,\mbox{and}\,\,\lambda_\theta:=(1-\theta)\lambda_*+\theta\lambda.
\end{align*}
The directional derivative of $\mathcal{Q}_{\beta,v_*}$ at $(\alpha_{v_*},\lambda_*)$ can be computed as (see, e.g., \cite[Theorem 2]{C14} and the derivation of \eqref{add:eq--1} below),
\begin{align*}
\frac{d}{d\theta}\mathcal{Q}_{\beta,v_*}\bigl(\alpha_{\theta},\lambda_{\theta}\bigr)\Big|_{\theta=0}&=\frac{\beta^2}{2}\int_0^{v_*}\xi''(s)\bigl(\alpha(s)-\alpha_{v_*}(s)\bigr)\bigl(\gamma_\beta(s)-s\bigr)ds+\Bigl(\int a^2\mu(da)-v_*\Bigr)(\lambda-\lambda_*)\\
&=\frac{\beta^2}{2}\int_0^{v_*}\xi''(s)\bigl(\alpha(s)-\alpha_{v_*}(s)\bigr)\bigl(\gamma_\beta(s)-s\bigr)ds,
\end{align*}
where the derivative is from the right-hand side of $0$. As a result, the optimality of $(\alpha_{v_*},\lambda_*)$ implies that the last line of the above display is nonnegative.
Write
\begin{align*}
&\int_0^{v_*}\xi''(s)\bigl(\alpha(s)-\alpha_{v_*}(s)\bigr)\bigl(\gamma_\beta(s)-s\bigr)ds\\
&=\int_0^{v_*}\int_0^s\xi''(s)\bigl(\gamma_\beta(s)-s\bigr)\alpha(da)ds-\int_0^{v_*}\xi''(s)\bigl(\gamma_\beta(s)-s\bigr)ds\\
&=\int_0^{v_*}\Bigl(\int_{a}^{v_*}\xi''(s)\bigl(\gamma_\beta(s)-s\bigr)ds \Bigr)\alpha(da)-\int_0^{v_*}\xi''(s)\bigl(\gamma_\beta(s)-s\bigr)ds.
\end{align*}	
From this, the optimality of $(\alpha_{v_*},\lambda_*)$ is equivalent to
\begin{align*}
\int_{v}^{v_*}\xi''(s)\bigl(\gamma_\beta(s)-s\bigr)ds\geq \int_0^{v_*}\xi''(s)\bigl(\gamma_\beta(s)-s\bigr)ds,\,\,\forall v\in [0,v_*],
\end{align*}
and hence, this is also equivalent to $\Gamma_\beta(v)\leq 0$ for all $v\in[0,v_*]$. Conversely, if $\Gamma_\beta(v)\leq 0$ for all $v\in[0,v^*]$, then this inequality implies that the above directional derivative of $\mathcal{Q}_{\beta,v_*}$ is nonnegative. This means that $(\alpha_{v_*},\lambda_*)$ is an optimizer of the variational problem $\inf_{\alpha,\lambda}\mathcal{Q}_{\beta,v_*}(\alpha,\lambda)$ and $\inf_{\alpha,\lambda}\mathcal{Q}_{\beta,v_*}(\alpha,\lambda)=\mathcal{Q}_{\beta,v_*}(\alpha_{v_*},\lambda_*)=0.$ From these and  Proposition \ref{new:parisiformula:k=1}, we arrive at $0=\inf_{\alpha,\lambda}\mathcal{Q}_{\beta,v_*}(\alpha,\lambda)\leq F(\beta)\leq 0$ and hence, $F(\beta)=0.$ This establishes the statement that for $\beta>0$, $\beta\in \mathcal{R}$ if and only if $\sup_{v\in (0,v_*]}\Gamma_\beta(v)\leq 0.$ 

Finally, the assertion $\mathcal{R}=(0,\beta_c]$ can be established similarly. Clearly, $\mathcal{R}\subseteq (0,\beta_c].$ If $0<\beta\leq \beta_c,$ then Lemma \ref{lem3} and the above proof imply that $\Gamma_\beta(v)\leq \Gamma_{\beta_c}(v)\leq 0$ for all $v\in [0,v_*]$ and thus, $\beta\in \mathcal{R}.$ This completes our proof.

\section{Overlap Concentration with Exponential Tail} \label{sec6}

Recall the probability space $(\Lambda,\mu)$, the Gibbs measure $G_{N,\beta}$, and the Gibbs expectation $\la\cdot\ra_{\beta}$ from Section \ref{vec_spin_model}. The following proposition states that in the high-temperature regime, the overlap of two i.i.d. sampled spin configurations from $G_{N,\beta}$ is concentrated around the origin with overwhelming probability. This result will be essential when we later bound the overlap moments. Let $I(A)$ denote the indicator function of a set $A$.

\begin{prop}\label{new:lem0}
	Assume that $0<\beta<\beta_c$ and that $s_0\in (0,1).$ For any $\varepsilon>0,$ there exists a constant $K>0$, depending only on $\beta$, $s_0,$ and $\varepsilon,$ such that the following property holds for any $N\geq 1$ and $s\in [s_0,1]$: For i.i.d. samplings $\s^1$ and $\s^2$ from $G_{N,s\beta},$
	\begin{align*}
	\e\bigl\la I(|R(\s^1,\s^2)|\geq \varepsilon)\bigr\ra_{s\beta}\leq Ke^{-N/K}.
	\end{align*}
\end{prop}

The rest of this section is devoted to proving Proposition \ref{new:lem0}.

\subsection{The Guerra-Talagrand Bound}\label{sec6.1}

Our main tool is the Guerra-Talagrand bound for the coupled free energy that we formulate as follows. Denote by $M_2(\mathbb{R})$ the space of all real-valued $2\times 2$ matrices equipped with the metric $$\|V-V'\|_{\max}=\max_{1\leq r,r'\leq 2}|V_{rr'}-V_{rr'}'|.$$ For $C,D\in M_2(\mathbb{R})$, denote by $\<C,D\>$ the inner product of $C$ and $D,$ i.e., $\la C,D\ra=\sum_{i,j=1}^2C_{ij}D_{ij};$ when $x,y\in \mathbb{R}^2$, denote by $\la x,y\ra$ the usual scalar product between $x$ and $y.$ For any $\s^1,\s^2\in \Lambda^N$, define the overlap matrix by
\begin{align*}
\mathbf{R}(\s^1,\s^2)=\left[
\begin{array}{ll}
R(\s^1,\s^1)&R(\s^1,\s^2)\\
R(\s^1,\s^2)&R(\s^2,\s^2)
\end{array}
\right].
\end{align*}
For any subset $A\subset M_2(\mathbb{R})$, define the coupled free energy restricted to $A$ by
\begin{align*}
\cf_N(\beta,A)&=\frac{1}{N}\log \int_{{\bf R}(\s^1,\s^2)\in A}e^{H_{N,\beta}(\s^1)+H_{N,\beta}(\s^2)}\mu^{\otimes N}(d\s^1)\mu^{\otimes N}(d\s^2). 
\end{align*}
Recall the space $\mathcal{V}$ from \eqref{V} and recall that $\mathcal{M}_{v}$ is the set of all probability measures on $[0,v]$. Let $v\in \mathcal{V}$ and $v_0\in [0,v]$ be fixed. Set $$V:=\left[
\begin{array}{cc}
v&v_0\\
v_0&v
\end{array}\right].
$$
Let $T$ be a $M_2(\mathbb{R})$-valued function on $[0,v]$ defined by
\begin{align*}
T(s)=\left[
\begin{array}{cc}
1&1\\
1&1
\end{array}\right],\,\,\forall s\in [0,v_0),
\quad \text{and} \quad
T(s)=\left[
\begin{array}{cc}
1&0\\
0&1
\end{array}\right],\,\,\forall s\in [v_0,v].
\end{align*}
For any $\alpha\in \mathcal{M}_v$, consider the weak solution $\Psi_{\beta,V,\alpha }$ to the following PDE for $(s,x,\lambda)\in [0,v)\times \mathbb{R}^2\times M_2(\mathbb{R})$:
\begin{align*}
\partial_s\Psi_{\beta,V,\alpha }&=-\frac{\beta^2\xi''}{2}\bigl(\la \triangledown^2 \Psi_{\beta,V,\alpha },T\ra+\alpha\la T\triangledown \Psi_{\beta,V,\alpha },\triangledown \Psi_{\beta,V,\alpha }\ra\bigr)
\end{align*}
with boundary condition
$
\Psi_{\beta,V,\alpha }(v,x,\lambda)=\log\int e^{\la a,x\ra+\la\lambda a,a\ra}(\mu\otimes \mu)(da).
$
For the existence of $\Psi_{\beta,V,\alpha}$, we refer the readers to \cite{JT2}. For $\alpha\in \mathcal{M}_v$ and $\lambda\in M_2(\mathbb{R}),$ define
\begin{align*}
\mathcal{P}_{\beta,V}(\alpha ,\lambda)&=\Psi_{\beta,V,\alpha }(0,0,\lambda)-\la\lambda,V\ra-\beta^2\Bigl(\int_0^{v}\xi''(s)s\alpha (s)ds+\int_0^{v_0}\xi''(s)s\alpha (s)ds\Bigr).
\end{align*}
Denote $A_\eta(V)=\{V'\in M_2(\mathbb{R}):\|V-V'\|_{\max}<\eta\}.$ The Guerra-Talagrand inequality (see \cite{Talbook2}) states that if $p$ is even, then for any $(\alpha,\lambda)\in \mathcal{M}_v\times M_2(\mathbb{R}),$
\begin{align}\label{GTB}
\lim_{\eta\downarrow 0}\limsup_{N\rightarrow\infty}\frac{1}{N}\e\log \int_{\mathbf{R}(\s^1,\s^2)\in A_\eta(V)}e^{\beta X_{N}(\s^1)+\beta X_{N}(\s^2)}\mu^{\otimes N}(d\s^1)\mu^{\otimes N}(d\s^2)\leq \mathcal{P}_{\beta,V}(\alpha ,\lambda).
\end{align}
When $p$ is odd, the validity of this inequality is an open question. Nevertheless, in the case of the Rademacher prior, i.e., $\mu=\delta_1/2+\delta_{-1}/2$,  the work \cite{Chen17} proved that for odd $p$, \eqref{GTB} remains valid if $\alpha\in \mathcal{M}_{v,v_0}$ for $0\leq v_0\leq v$, where 
$$\mathcal{M}_{v,v_0}:=\{\alpha\in \mathcal{M}_v:\mbox{$\alpha$ is a fixed constant on $[0,v_0)$ and $\alpha(s)=1$ on $[v_0,v].$}\}$$
In view of the proof in \cite{Chen17}, the argument does not rely on the measure $\mu$ in an essential way and it is  applicable to the current general setting so that \eqref{GTB} remains valid for odd $p$ and $(\alpha,\lambda)\in \mathcal{M}_{v,v_0}\times M_2(\mathbb{R})$. In the appendix, we present a sketch of the proof for this inequality. Now substituting the overlap term in $H_{N,\beta}$ via the restriction $A_{\eta}(V)$ in $\cf_N(\beta,A_\eta(V))$ yields

\begin{prop}[Guerra-Talagrand Bound]\label{prop2}
	For any $p\geq 2$, $v\in \mathcal{V}$, $v_0\in [0,v]$,  $\lambda\in M_2(\mathbb{R})$ and $\alpha\in \mathcal{M}_{v,v_0}$,
	\begin{align*}
	\lim_{\eta\downarrow 0}\limsup_{N\rightarrow\infty}\e\cf_N(\beta,A_\eta(V))&\leq \mathcal{Q}_{\beta,V}(\alpha,\lambda):=\mathcal{P}_{\beta,V}(\alpha ,\lambda)-\beta^2v^p.
	\end{align*}
\end{prop}



\subsection{Proof of Proposition \ref{new:lem0}}\label{sec6.2}

Before proving Proposition \ref{new:lem0}, we need a lemma, which will also be used later in Section \ref{high_temp_k>1}.

\begin{lemma}\label{add:lem0}
Assume that $\beta\in \mathcal{R}.$ For any $s_0\in (0,1]$ and $\varepsilon>0,$ there exists a constant $K>0$ independent of $N$ such that 
\begin{align*}
\e\bigl\la I\bigl(|R(\sigma,\sigma)-v_*|\geq \varepsilon\bigr)\bigr\ra_{s\beta}\leq Ke^{-N/K},\,\,\forall s\in [s_0,1],N\geq 1,
\end{align*}
where $\sigma$ is a sampling from $G_{N,s\beta}.$
\end{lemma}

\begin{proof}
From Jensen's inequality and $\e e^{H_{N,s\beta}}(\sigma)=1$,
\begin{align*}
&\limsup_{N\to\infty}\frac{1}{N}\e\log\int_{R(\sigma,\sigma)\notin (v_*-\varepsilon,v_*+\varepsilon)}e^{H_{N,s\beta}(\sigma)}\mu^{\otimes N}(d\sigma)\leq \limsup_{N\to\infty}\frac{1}{N}\log\int_{R(\sigma,\sigma)\notin (v_*-\varepsilon,v_*+\varepsilon)}\mu^{\otimes N}(d\sigma).
\end{align*}
Observe that in the second integral $\sigma=(\sigma_1,\ldots,\sigma_N)$ are i.i.d. random variables with respect to the measure $\mu$, which has a bounded support and variance $v_*$. Using Cram\'{e}r's theorem (see, e.g., \cite[Theorem 2.2.3]{dembo_zeit}), there exists a $\delta>0$ such that 
\begin{align*}
\int_{R(\sigma,\sigma)\notin (v_*-\varepsilon,v_*+\varepsilon)}\mu^{\otimes N}(d\sigma)&\leq e^{-N\delta},\,\,\forall N\geq 1.
\end{align*}
Hence, from the above inequalities and Proposition \ref{prop1},
$$
\limsup_{N\to\infty}\frac{1}{N}\e\log\int_{R(\sigma,\sigma)\notin (v_*-\varepsilon,v_*+\varepsilon)}e^{H_{N,s\beta}(\sigma)}\mu^{\otimes N}(d\sigma)\leq -\delta=F(s\beta)-\delta,\,\,\forall s\in [s_0,1].
$$
Next, by using  Gaussian integration by parts, one can compute the derivatives in the $s$ variable for 
$$\frac{1}{N}\e\log\int_{R(\sigma,\sigma)\notin (v_*-\varepsilon,v_*+\varepsilon)}e^{H_{N,s\beta}(\sigma)}\mu^{\otimes N}(d\sigma)\,\,\mbox{and}\,\,\e F_N(s\beta)$$
to show that the resulting derivatives are uniformly bounded over $s\in [s_0,1].$ As a consequence, these two sequences of functions are \ equicontinuous. From this and the above inequality, there exists some $N_0\geq 1$ such that
$$
\frac{1}{N}\e\log\int_{R(\sigma,\sigma)\notin (v_*-\varepsilon,v_*+\varepsilon)}e^{H_{N,s\beta}(\sigma)}\mu^{\otimes N}(d\sigma)\leq \e F_N(s\beta)-\frac{\delta}{2},\,\,\forall s\in [s_0,1],N\geq N_0.
$$
Finally, from this inequality and the Gaussian concentration inequality for $F_N(s\beta)$ and its restricted free energy, there exists a universal constant $K>0$ such that for any $s\in [s_0,1]$ and $N\geq N_0$, with probability at least $1-Ke^{-N/K}$,
$$
\frac{1}{N}\log\int_{R(\sigma,\sigma)\notin (v_*-\varepsilon,v_*+\varepsilon)}e^{H_{N,s\beta}(\sigma)}\mu^{\otimes N}(d\sigma)\leq  F_N(s\beta)-\frac{\delta}{4},
$$ 
which implies that
$$
\e\bigl\la I\bigl(|R(\sigma,\sigma)-v_*|\geq \varepsilon\bigr)\bigr\ra_{s\beta}\leq e^{-N\delta/4}+Ke^{-N/K}.
$$
This completes our proof.
\end{proof}

Now we establish the proof of Proposition \ref{new:lem0}. Let $0<\beta<\beta_c$, $0<\varepsilon<v_*$, and $s_0\in (0,1)$ be fixed. Let $v=v_*$. Suppose that $v_0\in [\varepsilon,v].$ For $s\in [0,1]$, denote $\beta_s=s\beta.$ Let $\lambda\in M_2(\mathbb{R})$ with $\lambda_{1,1}=\lambda_{2,2}=-\beta_s^2\xi'(v_*)/2$ and $\lambda_{1,2}=\lambda_{2,1}=0.$ Let $\alpha\in \mathcal{M}_{v,v_0}$ satisfy $\alpha\equiv 0$ on $[0,v_0)$ and $\alpha\equiv 1$ on $[v_0,v].$ For $\theta\in [0,1]$, set
\begin{align*}
\alpha_\theta(s)=\left\{
\begin{array}{ll}
\frac{1-\theta}{2},&\mbox{if $s\in [0,v_0)$},\\
1,&\mbox{if $s\in [v_0,v]$}.
\end{array}
\right.
\end{align*}
Using the Cole-Hopf transformation, one can compute that
\begin{align*}
\mathcal{Q}_{\beta_s,V}(\alpha_\theta ,\lambda)&=\frac{2}{1-\theta}\log \e\Bigl[ g_0\bigl(\beta_s^2\xi'(v_0),\beta_s \xi'(v_0)^{1/2}z\bigr)^{1-\theta}\Bigr]+\beta_s^2\xi'(v)v\\
&-\beta_s^2\Bigl((1-\theta)\int_0^{v_0}\xi''(r)rdr+\int_{v_0}^{v}\xi''(r)rdr\Bigr)-\beta_s^2v^p,
\end{align*}
where $$g_0(t,x):=\int e^{ax-a^2t/2}\mu(da)$$
and $z$ is a standard normal random variable. A direct differentiation in $\theta$ yields
\begin{align*}
\partial_\theta\mathcal{Q}_{\beta_s,V}(\alpha_{\theta} ,\lambda)\Big|_{\theta=0}&=-2\e g(t,\sqrt{t}z)\log g(t,\sqrt{t}z)+\beta_s^2\int_0^{v_0}\xi''(s)ds,\,\,t=\beta_s^2\xi'(v_0).
\end{align*}
To handle this equation, we use Gaussian integration by parts to get
\begin{align*}
&\frac{d}{dt}\e g\bigl(t,\sqrt{t}z\bigr)\log g\bigl(t,\sqrt{t}z\bigr)\\
&=\e\partial_tg(t,\sqrt{t}z)\bigl(\log g(t,\sqrt{t}z)+1\bigr)+\frac{1}{2}\Bigl(\e \partial_{xx}g(t,\sqrt{t}z)\bigl(\log g(t,\sqrt{t}z)+1\bigr)+ \e\frac{\partial_xg(t,\sqrt{t}z)^2}{g(t,\sqrt{t}z)}\Bigr)\\
&=-\frac{1}{2}\e\partial_{xx}g(t,\sqrt{t}z)\bigl(\log g(t,\sqrt{t}z)+1\bigr)+\frac{1}{2}\Bigl(\e \partial_{xx}g(t,\sqrt{t}z)\bigl(\log g(t,\sqrt{t}z)+1\bigr)+ \e\frac{\partial_xg(t,\sqrt{t}z)^2}{g(t,\sqrt{t}z)}\Bigr)\\
&=\frac{1}{2}\e\frac{\partial_xg(t,\sqrt{t}z)^2}{g(t,\sqrt{t}z)},
\end{align*}
where the last equality used the observation $\partial_tg=-\partial_{xx}g/2.$ Consequently, 
\begin{align}
\begin{split}\label{add:eq--1}
\partial_\theta\mathcal{Q}_{\beta_s,V}(\alpha_{\theta} ,\lambda)\Big|_{\theta=0}&=2\beta_s^2\int_0^{v_0}\frac{1}{2}\xi''(r)\e\frac{\partial_xg(\beta_s^2\xi'(r),\beta_s\sqrt{\xi'(r)}z)^2}{g(\beta_s^2\xi'(r),\beta_s\sqrt{\xi'(r)}z)}dr-\beta_s^2\int_0^{v_0}\xi''(s)ds
\end{split}\\
\begin{split}\notag
&=\beta_s^2\int_0^{v_0}\xi''(r)(\gamma_{\beta_s}(r)-1)dr\\
&<\beta_s^2\int_0^{v_0}\xi''(r)(\gamma_{\beta_c}(r)-1)dr=\beta_s^2\Gamma_{\beta_c}(v_0)\leq 0
\end{split}
\end{align}
for all $0<s\leq 1,$ where the strict inequality used the monotonicity of the function $\gamma_{b}(r)$ in $b$ by Lemma \ref{lem3} and the last inequality used Proposition \ref{prop1}.
Now since $\mathcal{Q}_{\beta_s,V}(\alpha_{\theta} ,\lambda)$ is a continuous function in $(s,v_0,\theta)$ and $\mathcal{Q}_{\beta_s,V}(\alpha_{\theta} ,\lambda)|_{\theta=0}=0$, there exist $\delta,\delta' > 0$ such that
$$
\sup_{s\in [s_0,1]}\sup_{v_0\in [\varepsilon,v+\delta']}\inf_{\theta\in [0,1]}\mathcal{Q}_{\beta_s,V}(\alpha_{\theta},\lambda)\leq -\delta.
$$
Consequently, from Proposition \ref{prop2}, we see that the coupled free energy exhibits a free energy cost, that is, for any $v_0\in [\varepsilon,v+\delta']$ and $s\in [s_0,1]$,
\begin{align}
\lim_{\eta\downarrow 0}\limsup_{N\rightarrow \infty}\E \cf_{N}(\beta_s, {A}_{\eta}( V)) \leq -\delta. \label{freeencost}
\end{align}
Set
\begin{align*}
A^+&=\bigl\{V\in M_2(\mathbb{R}):V\,\,\mbox{positive semi-definite with }V_{12}=V_{21}\geq \varepsilon, V_{11},V_{22}\in [v-\delta',v+\delta']\bigr\},\\
A^-&=\bigl\{V\in M_2(\mathbb{R}):V\,\,\mbox{positive semi-definite with }V_{12}=V_{21}\leq -\varepsilon, V_{11},V_{22}\in [v-\delta',v+\delta']\bigr\}.
\end{align*}
Using the inequality $$
\log (x_1+\cdots+x_k)\leq \log k+\max_{1\leq \ell\leq k}\log x_\ell,\,\,\forall x_1,\ldots,x_k>0,
$$
we see from the compactness of $A^+$ and \eqref{freeencost} that for any $s\in [s_0,1]$,
\begin{align}
\limsup_{N\rightarrow \infty}\E \cf_{N}(\beta_s, {A}^+) \leq -\frac{\delta}{2}.
\end{align}
To obtain the same inequality for $A^-,$ we note that when $p$ is even, $H_{N,\beta}(\sigma)=H_{N,\beta}(-\sigma)$, which implies that  $\E \cf_{N}(\beta_s, {A}^+)=\E \cf_{N}(\beta_s, {A}^-)$ and thus,
\begin{align*}
\limsup_{N\rightarrow \infty}\E \cf_{N}(\beta_s, {A}^-)\leq -\frac{\delta}{2}.
\end{align*}
When $p$ is an odd number, Jensen's inequality yields that 
\begin{align*}
\limsup_{N\to\infty}\e\cf_N(\beta_s,A^-)&\leq \limsup_{N\to\infty}\frac{1}{N}\log \int_{{\bf R}(\s^1,\s^2)\in A^-} e^{\beta^2NR(\s^1,\s^2)^p}\mu^{\otimes N}(\s^1)\mu^{\otimes N}(\s^2)<-\beta^2\varepsilon^p<0.
\end{align*}
Combining these and Proposition \ref{prop1} together implies that there exists a constant $\delta''>0$ such that for all $s\in [s_0,1],$
\begin{align}
\begin{split}\label{add:eq-11}
\limsup_{N\to\infty}\e\cf_N(\beta_s,A^+\cup A^-)&<-\delta''=2F(\beta_s)-\delta''.
\end{split}
\end{align}

The rest of the proof follows essentially in the same way as that for Lemma \ref{add:lem0}. 
By computing the derivatives of $\e\cf_N(\beta_s,A^+\cup A^-)$ and $\e F_N(\beta_s)$ in $s$ and using Gaussian integration by parts, it can be checked that these derivatives are uniformly bounded over $s\in [s_0,1].$ Hence, $s\mapsto \e\cf_N(\beta_s,A^+\cup A^-)$ and $s\mapsto \e F_N(\beta_s)$ are two families of  equicontinuous functions. From this and \eqref{add:eq-11}, there exists some $N_0\geq 1$ such that as long as $N\geq N_0$ and $s\in [s_0,1]$, 
\begin{align*}
\e\cf_N(\beta_s,A^+\cup A^-)&\leq 2\e F_N(\beta_s)-\frac{\delta''}{2}.
\end{align*}
From the Gaussian concentration inequality for $\cf_N(\beta_s,A^+\cup A^-)$ and $F_N(\beta_s)$, there exists a constant $K>0$ such that for any $N\geq N_0$ and $s\in [s_0,1]$, with probability at least $1-Ke^{-N/K}$,
\begin{align*}
\cf_N(\beta_s,A^+\cup A^-)\leq 2F_N(\beta_s) -\frac{\delta''}{4},
\end{align*}
which leads to 
\begin{align*}
\bigl\la I\bigl({\bf R}(\s^1,\s^2)\in A^+\cup A^-\bigr)\bigr\ra_{\beta_s}=e^{N(\cf_N(\beta_s,A^+\cup A^-)- 2F_N(\beta_s))}\leq e^{-\frac{\delta'' N}{4}}.
\end{align*}
By taking the expectation, there exists a constant $K'>0$ such that for any $N\geq N_0$ and $s\in [s_0,1],$
\begin{align}\label{add:eq0}
\e\bigl\la I\bigl({\bf R}(\s^1,\s^2)\in A^+\cup A^-\bigr)\bigr\ra_{\beta_s}\leq K' e^{-N/K'}.
\end{align}
Finally, from Lemma \ref{add:lem0}, there exists some $K''>0$ such that for any $N\geq 1$, $s\in [s_0,1],$ and $\ell=1,2,$
\begin{align*}
\e\bigl\la I\bigl(R(\s^\ell,\s^\ell)\notin (v-\delta',v+\delta')\bigr)\bigr\ra_{\beta_s}\leq K'' e^{-N/K''}.
\end{align*}
This together with \eqref{add:eq0} implies that for any $N\geq N_0$ and $s\in [s_0,1],$
\begin{align*}
\e\bigl\la I\bigl(|R(\s^1,\s^2)|\geq \varepsilon\bigr)\bigr\ra_{\beta_s}\leq K' e^{-N/K'}+2K'' e^{-N/K''}
\end{align*}
and this completes the proof of Proposition \ref{new:lem0}. 

\section{Overlap Concentration with Moment Control} \label{sec7}

As we have seen in Proposition \ref{new:lem0}, overlaps between i.i.d. samples of $G_{N,\beta}$  are concentrated around the origin with exponential tail control. The aim of this section is to establish the proof of Theorem~\ref{momentcontrol}, namely, the moment control of the overlap.  The proof is based on the so-called cavity method in mean field spin glasses. As an immediate consequence of Theorem~\ref{momentcontrol}, we also present the proof of Proposition \ref{fluctuations}.		

Briefly speaking, the cavity method is an induction argument that compares the systems of sizes $N$ and $N-1$ by parameterizing an interpolating path between the two systems and controlling the derivative in the parameter along this path. This technique is a very well-known tool in the physics literature, see \cite{MPV}. Mathematically, it was implemented in the study of the high-temperature behavior for a number of mean field spin glass models by Talagrand \cite{Talbook1}. For technical reasons, most of the existing results in \cite{Talbook1} are valid only for a sub-region of the high-temperature regime and not up to the critical temperature. In the present paper, by adapting the argument in \cite[Chapter 13]{Talbook2} and \cite{Chen17}, it turns out that from our understanding of the structure of the high-temperature regime $\mathcal{R}$ as well as the Parisi variational formula for the marginal free energy, we can show that the cavity method can indeed be applied throughout the entire high-temperature regime and ultimately it leads to the asserted moment control of the overlaps.

Before turning to the proof, we set some notation.
For any $\ell,\ell'\geq 1$ and $1\leq r,r'\leq k$, denote by $\s ,\s ^1,\s ^2,\ldots$ the spin configurations from $\Lambda^N.$ Set overlaps
\begin{align*}
R_{\ell,\ell'} &=R\bigl(\s ^\ell ,\s ^{\ell'} \bigr),\quad R^-=\frac{1}{N}\sum_{i=1}^{N-1}\sigma_i \sigma_i ,\quad	
R_{\ell,\ell'}^-=\frac{1}{N}\sum_{i=1}^{N-1}\sigma_i^\ell \sigma_i^{\ell'} .
\end{align*}
In Section \ref{Cavity}, we device an interpolating system that connects the model of sizes $N-1$ and $N$. Section \ref{preliminarylemmas} computes and bounds the derivative of the expectations of functions of the replicas sampled from the interpolating Gibbs measure along our interpolation. Additionally, this subsection presents some lemmas that are simple yet necessary to bound various powers of overlaps. These results are used in Section \ref{sec7.3}, where we present the cavity argument to establish an iterative inequality for the moments of the overlaps. Finally, Sections \ref{sec7.4} and \ref{sec7.5} prove Theorem \ref{momentcontrol} and Proposition \ref{fluctuations}, respectively.

\subsection{Constructing an Interpolation Path}\label{Cavity}

For each $S  \subseteq \{1,\dots,p\}$, define $I_S $ as the set of indices $(i_1,\ldots,i_p)\in \{1,\ldots,N\}^p$ such that $i_s = N$ for all $s\in S$ and $i_s<N$ otherwise. For example, if $p=4$ and $S=\{1,3\}$, then $I_S=\{(N,i,N,j):1\leq i,j<N\}.$ Define a Gaussian process indexed by $S $:
\[
X_N^S (\s  ) = \frac{1}{N^{(p-1)/2}} \sum_{(i_1,\dots,i_p) \in I_S } Y_{i_1,\dots,i_p}\sigma_{i_1}  \cdots \sigma_{i_p} .
\]
It is easy to check that
\begin{align*}
\E X_N^S (\s ^1 )X_N^{S }(\s ^2 )=N(R_{1,2}^-)^{p-|S |}\bigl(\sigma_N^1 \sigma_N^2 \bigr)^{|S |}.
\end{align*}
Notice that $S  = \emptyset$ is the only set such that $X_N^{S }(\s  )$ does not involve the last spin $\sigma_N $. For $t \in [0,1]$, define the interpolating Hamiltonian by
\begin{align*}
H_{N,\beta,t}(\s ) &= \beta\Bigl( X_{N}^{\emptyset}(\s  ) +\sqrt{t}\sum_{j=1}^p \sum_{S \subset\{1,\ldots,p\}: |S |=j}X_N^S (\s  )\Bigr)\\
&-\frac{\beta^2}{2} \Bigl((R^-)^p+\frac{t}{N^{j-1}}\sum_{j=1}^p \sum_{S\subset\{1,\ldots,p\} : |S |=j}(R^-)^{p-j}(\sigma_N \sigma_N )^j\Bigr).
\end{align*}
From the binomial formula, one readily checks that when $t=0$, $H_{N,\beta,0}(\s )$ is equal to $H_{N-1}$ at a different temperature: $$\beta \frac{(N-1)^{(p-1)/2}}{N^{(p-1)/2}}.$$ When $t=1$, $H_{N,\beta,1}(\s )$ is simply the original Hamiltonian $H_{N,\beta}(\sigma)$. We define the Gibbs measure associated to $H_{N,\beta,t}$ in the same manner as $G_{N,\beta},$ i.e.,
\begin{align*}
G_{N,\beta,t}(d\s )&=\frac{\exp H_{N,\beta,t}(\s )\mu^{\otimes N}(d\s )}{\int \exp H_{N,\beta,t}(\s ')\mu^{\otimes N}(d\s ')}.
\end{align*}
As before, denote by $(\s ^\ell)_{\ell\geq 1}$ a sequence of i.i.d. samples from $G_{N,\beta,t}$ and by $\la\cdot\ra_{\beta,t}$ the Gibbs average with respect to this sequence. For any bounded measurable function $f$ of the sequence $(\s ^\ell)_{\ell\geq 1},$ set $\nu_{\beta,t}(f)=\e\la f\ra_{\beta,t}$. When $t=1$, we simply write $\nu_\beta(f)=\nu_{\beta,1}(f).$ We also denote the $t$-derivative of $\nu_{\beta,t}(f)$ by $\nu_{\beta,t}'(f).$

\subsection{Some Auxiliary Lemmas} \label{preliminarylemmas}

We gather some lemmas that will be used in the proof of Theorem \ref{momentcontrol} below. As their proofs are fairly standard, we refer the readers to \cite[Chapter 1]{Talbook1} or \cite{Chen17}. First, we compute the derivative of $\nu_{\beta,t}(f).$ 

\begin{lemma}\label{deriv} For any bounded function $f$ of $\sigma^1$, $\dots$, $\sigma^n$, we have
	\begin{align*}
	\nu_{\beta,t}'(f)  = \beta^2\sum_{j=1}^p {p \choose j} \frac{1}{N^{j-1}} &\Big( \sum_{1 \leq \ell < \ell' \leq n} \nu_{\beta,t}(f (R_{\ell,\ell'}^-)^{p-j}(\sigma_N^{\ell} \sigma_N^{\ell'} )^j) \\
	& -n \sum_{\ell \leq n}\nu_{\beta,t}(f(R^-_{\ell,n+1})^{p-j}(\sigma_N^{\ell} \sigma_N^{n+1} )^j)\\
	& +\frac{n(n+1)}{2} \nu_{\beta,t}(f(R^-_{n+1,n+2})^{p-j}(\sigma_N^{n+1} \sigma_N^{n+2} )^j)\Big).
	\end{align*}
\end{lemma}

Note that since $\Lambda$ is bounded, there exists a constant $M>1$ such that $\Lambda\subseteq [-M,M].$ The next lemma controls $\nu_{\beta,t}(f)$ by the terminal value $\nu_\beta(f).$

\begin{lemma}\label{boundderiv}
	For any non-negative and bounded function $f$ of $\sigma^1$, $\dots$, $\sigma^n$, we have
	\[
	\nu_{\beta,t}(f) \leq \exp\left(n^2 2^{p+1}M^{2p} \beta^2 \right) \nu_\beta(f).
	\]
\end{lemma}

In the proof of Theorem \ref{momentcontrol}, it will sometimes be desirable to work with the overlaps $R_{1,2}^-$ instead of the overlaps $R_{1,2}  $ and vice versa. Lemma \ref{easybound} will allow us to replace $(R_{1,2}  )^m$ by $(R_{1,2}^-)^m$ (or vice versa).
On the other hand, Lemma \ref{jthpower2} states that we can also control the moments of $R^-_{1,2}$ by the bounds on $R  _{1,2}.$ 

\begin{lemma}\label{easybound}
	For any $m\geq 1$,
	\[
	|(R_{1,2}  )^{m+1}-(R^-_{1,2})^{m+1}| \leq \frac{M^2 m}{N} \left(|R_{1,2}  |^{m}+|R^-_{1,2}|^{m}\right).
	\]
\end{lemma}



\begin{lemma}\label{jthpower2}Let $m\in \mathbb{N}.$ Assume that there exists some $K\geq 1$ such that
$\nu_\beta((R_{1,2}  )^{2j}) \leq K/N^j$
	for any $0\leq j\leq m.$ Then
	$$
	\nu_\beta((R^-_{1,2})^{2m}) \leq \frac{2^{2m}M^{4m}K}{N^{m}}.
	$$
\end{lemma}

\subsection{Cavity Argument} \label{sec7.3}

The following lemma is the key ingredient of our argument. It is obtained via a purely algebraic cavity computation and does not require any fact about the high-temperature behavior of the overlaps.

\begin{lemma}\label{add:lem1}
	Let $m$ be a nonnegative integer and $\beta>0$. Assume that there exists a constant $K_0 \geq 1$ such that for all $0 \leq j \leq m$ and $N \geq 1$
	\begin{align*}
	\nu_\beta((R_{1,2}  )^{2j}) \leq \frac{K_0}{N^j}.
	\end{align*}
	Then
	\begin{align*}
	\nu_\beta(|R_{1,2}|^{2m +2}) \leq   {K}_1(\beta)\nu_\beta(|R_{1,2}|^{2m +3})+ \frac{ {K}_2(\beta)}{N^{m +1}}
	\end{align*}
	for all $N\geq 1,$
	where $ {K}_1$ and $ {K}_2$ are two nonnegative continuous functions of $\beta$ and they are independent of $N$. In addition, $ {K}_1$ is nondecreasing with $ {K}_1(0)=0$ if and only if $\beta=0.$
\end{lemma}

\begin{proof}
	We divide our proof into four steps.

	{\noindent \textbf{Step 1:}} By symmetry between sites, write
	\begin{align*}
	\nu_\beta((R_{1,2})^{2m  + 2}) = \nu_\beta(\sigma_N^1\sigma_N^2(R_{1,2})^{2m +1}) = \nu_\beta(\sigma_N^1\sigma_N^2(R_{1,2}^-)^{2m +1}) + \mathcal{E}.
	\end{align*}
	Here,
	\begin{align*}
	\mathcal{E}: = \nu_\beta(\sigma_N^1\sigma_N^2((R_{1,2})^{2m +1}-(R_{1,2}^-)^{2m +1}))
	\end{align*}
	can be controlled by Lemmas \ref{easybound} and \ref{jthpower2} as follows:
	\begin{align*}
	|\mathcal{E}| &\leq M^2 \nu\left(|(R_{1,2})^{2m +1}-(R_{1,2}^-)^{2m +1}|\right)\\
	&\leq \frac{2mM^4}{N}\left(\nu_\beta(|R_{1,2}|^{2m })+\nu_\beta(|R_{1,2}^-|^{2m }) \right)\leq \frac{ C_1}{N^{m +1}},
	\end{align*}
	where $C_1 := K_0(2m  M^4 + m2^{2m+1}M^{4m+4}).$
    Thus, we arrive at
	\begin{align}
	\nu_\beta((R_{1,2})^{2m +2}) \leq \nu_\beta(\sigma_N^1\sigma_N^2(R_{1,2}^-)^{2m +1}) + \frac{C_1}{N^{m +1}}. \label{proof1}
	\end{align}
	Next, in order to control the right-hand side, we define $f = \sigma_N^1\sigma_N^2 (R_{1,2}^-)^{2m +1}$. Recall that since $\mu$ is centered, $\nu_0(f) = 0$. This together with an application of the mean value theorem and \eqref{proof1} results in the inequality
	\begin{align}
	\nu_\beta((R_{1,2})^{2m +2}) \leq \nu_\beta(f)+ \frac{C_1}{N^{m +1}}\leq \sup_{0 \leq t \leq 1} |\nu_{\beta,t}'(f)| + \frac{C_1}{N^{m +1}}. \label{proof2}
	\end{align}		
	
	{\noindent\textbf{Step 2:} } We control $|\nu_{\beta,t}'(f)|.$ Applying Lemma \ref{deriv} with $n=2$ and noting that $|\sigma_N^\ell\sigma_N^{\ell'}| \leq M^2$ for any $1 \leq \ell,\ell' \leq n$ yield  the following bound on $|\nu_{\beta,t}'(f)|$:
	\begin{align*}
	\beta^2 \sum_{j=1}^p {p \choose j} \frac{(M^2)^{j+1}}{N^{j-1}}&\left(\nu_{\beta,t}(|R_{1,2}^-|^{2m +1}|R_{1,2} |^{p-j}) +2 \nu_{\beta,t}(|R_{1,2}^-|^{2m +1}|R_{1,3}^-|^{p-j})\right.\\
	&+\left.2\nu_{\beta,t}(|R_{1,2}^-|^{2m +1}|R_{2,3}^-|^{p-j}) +3\nu_{\beta,t}(|R_{1,2}^-|^{2m +1}|R_{3,4}^-|^{p-j} )\right).
	\end{align*}
	For each $1 \leq j \leq p$, set the H\"{o}lder conjugate exponents $$\tau_j^1= \frac{2m +1+p-j}{2m +1},\,\,\tau_j^2 = \frac{\tau_j^1}{1-\tau_j^1}.$$ By H\"older's inequality, for each $1 \leq j \leq p$ and each pair of replica indices $1 \leq \ell,\ell' \leq n$,
	\begin{align*}
	\nu_{\beta,t}(|R_{1,2}^-|^{2m +1}|R_{\ell,\ell'}^-|^{p-j})&\leq \nu_{\beta,t}(|R_{1,2}^-|^{2m +1+p-j})^{1/\tau_j^1}\nu_{\beta,t}(|R_{\ell,\ell'}^-|^{2m +1+p-j})^{1/\tau_j^2}\\
	&=\nu_{\beta,t}(|R_{1,2}^-|^{2m +1+p-j}).
	\end{align*}
	This inequality leads to		
	\[
	|\nu_{\beta,t}'(f)| \leq 8 \beta^2 \sum_{j=1}^p {p\choose j} \frac{(M^2)^{j+1}}{N^{j-1}}\nu_{\beta,t}(|R_{1,2}^-|^{2m +1+p-j}).
	\]
	Consequently, Lemma \ref{boundderiv} allows us to replace $\nu_{\beta,t}$ by $\nu_{\beta}$ on the right-hand side above to obtain that
	
	\begin{align}
	\begin{split}\label{add:eq6}
	|\nu_{\beta,t}'(f)| &\leq 8 e^{2^{p+3}M^{2p}\beta^2}\beta^2 \sum_{j=1}^p {p\choose j} \frac{(M^2)^{j+1}}{N^{j-1}} \nu_{\beta}(|R_{1,2}^-|^{2m +1+p-j}).
	\end{split}
	\end{align}
	{\noindent \bf Step 3:} We break the above sum into two pieces: When $j=1,$
	$\nu_{\beta}(|R_{1,2}^-|^{2m +1+p-j})= \nu_\beta(|R_{1,2}^-|^{2m +p});$
	when $2\leq j\leq p,$ since $M> 1,$
	\begin{align*}
	\nu_{\beta}(|R_{1,2}^-|^{2m +1+p-j})
	& \leq M^{2(1+p-j)}\nu_{\beta}(|R_{1,2}^-|^{2m})\leq \frac{2^{2m}M^{2(1+p-j+2m)}K_0}{N^m}.
	\end{align*}
	The last inequality used the given assumption and Lemma \ref{jthpower2}. Let
	\begin{align*}
	C(\beta) := 8 e^{2^{p+3}M^{2p}\beta^2}\beta^2,\,\,
	C_2 := pM^4, \,\,
	C_3  := 2^{2m+p}M^{2(2m+p+2) }K_0.
	\end{align*}
	From \eqref{add:eq6} and the last two inequalities,
	\[
	|\nu_{\beta,t}'(f)| \leq C(\beta)C_2 \nu_\beta(|R_{1,2}^{-}|^{2m  + p}) + \frac{C(\beta)C_3}{N^{m +1}},
	\]
	and subsequently, plugging this into \eqref{proof2} gives
	\begin{align}
	\nu_\beta((R_{1,2})^{2m +2}) \leq  C(\beta)C_2 \nu_\beta(|R_{1,2}^{-}|^{2m  + p}) + \frac{C(\beta)C_3+C_1}{N^{m +1}}. \label{proof3}
	\end{align}

	{\noindent\textbf{Step 4:}} We may now perform a procedure similar to Step 1 to bring $R_{1,2}^-$ back to $R_{1,2}$. By Lemma~\ref{easybound} and $M>1,$
	\begin{align*}
	\nu_\beta(|R_{1,2}^{- }|^{2m  + p}) & \leq \nu_\beta(|R_{1,2}|^{2m +p})+ \frac{2m  + p}{N}M^2\left(\nu_\beta(|R_{1,2}^{-}|^{2m +p-1})+\nu_\beta(|R_{1,2}|^{2m +p-1}) \right)\\
	& \leq \nu_\beta(|R_{1,2}|^{2m +p}) +\frac{2m  +p}{N}M^{2p}\left(\nu_\beta(|R_{1,2}^-|^{2m })+\nu_\beta(|R_{1,2}|^{2m }) \right).
	\end{align*}
	Consequently, from Lemma \ref{jthpower2}, it follows that
	\begin{align*}
	\nu_\beta(|R_{1,2}^{-}|^{2m  + p}) \leq \nu_\beta(|R_{1,2}|^{2m +p}) + \frac{C_4}{N^{m+1}}.
	\end{align*}
	for 
	$
	C_4 :=(2m+p)M^{2p}K_0(2^{2m}M^{4m}+1).
	$
	Plugging this into \eqref{proof3} and noting that $p\geq 3$ imply
	\begin{align*}
	\nu_\beta((R_{1,2})^{2m +2}) &\leq  C(\beta)C_2 \nu_\beta((R_{1,2})^{2m  + p})+ \frac{C(\beta)(C_2C_4+C_3)+C_1}{N^{m +1}}\\
	&\leq C(\beta)C_2 M^{2(p-3)}\nu_\beta(|R_{1,2}|^{2m  + 3})+ \frac{C(\beta)(C_2C_4+C_3)+C_1}{N^{m +1}}.
	\end{align*}
	Setting
	$
	{K}_1(\beta) = C(\beta)C_2M^{2(p-3)}
	$
	and
	$
	{K}_2(\beta) = C(\beta)(C_2C_4+C_3)+C_1
	$
	completes the proof.
\end{proof}

\subsection{Proof of Theorem \ref{momentcontrol}}\label{sec7.4}

We prove Theorem \ref{momentcontrol} by induction on $m\geq 0.$ Clearly the case $m=0$ is valid. Assume that for some $m\geq 0$, there exists a constant $K\geq 1$ such that \eqref{momentcontrol:eq1} holds for all $N\geq 1$ and $s\in [0,1].$ Our goal is to show that there exists some $K'\geq 1$ such that
\begin{align*}
\nu_{s\beta}(|R_{1,2}|^{2(m +1)}) \leq \frac{K'}{N^{m +1}} 
\end{align*}
for all $N\geq 1$ and $s\in [0,1].$ Let $ {K}_1$ and $ {K}_2$ be the two nonnegative continuous functions from the statement of Lemma \ref{add:lem1} so that for all $N\geq 1$ and $s\in[0,1]$
\begin{align}\label{momentcontrol:proof:eq1}
\nu_{s\beta}(|R_{1,2}|^{2m +2}) \leq   {K}_1(s\beta)\nu_{s\beta}(|R_{1,2}|^{2m +3})+ \frac{ {K}_2(s\beta)}{N^{m +1}}.
\end{align}
Note that $ {K}_1(s\beta)$ is a nondecreasing function in $s$ and $ {K}_1(0)=0.$ Set $$s_0=\sup\Bigl\{s\in [0,1]: {K}_1(s\beta)M^2\leq \frac{1}{2}\Big\}.$$
Now we divide our proof into two cases:

{\noindent \bf Case 1:} ${s \in [0, s_0]}$.  Combining \eqref{momentcontrol:proof:eq1} and the observation that $|R_{1,2}  |\leq M^2$ results in
\begin{align*}
\nu_{s\beta}((R_{1,2})^{2m +2}) &\leq   {K}_1(s_0\beta)M^{2}\nu_{s\beta}(|R_{1,2}|^{2m  + 2})+ \frac{ {K}_2(s\beta)}{N^{m +1}}\\
&\leq  \frac{1}{2}\nu_{s\beta}((R_{1,2})^{2m  + 2})+ \frac{ {K}_2(s\beta)}{N^{m +1}},\,\forall s\in[0,s_0].
\end{align*}
This gives that
\[
\nu_{s\beta}((R_{1,2})^{2m +2}) \leq \frac{2 {K}_2(s\beta)}{N^{m +1}},\,\,\forall s\in[0,s_0].
\]
{\noindent \bf Case 2:} ${s\in (s_0,1]}.$ Choose $\epsilon > 0$ such that
\begin{align}\label{add:eq8}
\epsilon \max_{s\in [s_0,1]}  K_1(s\beta)< \frac{1}{2}.
\end{align}
From Proposition \ref{new:lem0}, there exists a constant $K''$ independent of $N$ and $s\in [s_0,1]$ such that
\[
\nu_{s\beta}(I(|R_{1,2}|>\epsilon))<K'' e^{-N/K''},\,\,\forall N\geq 1.
\]
Note that
\begin{align*}
\nu_{s\beta}(|R_{1,2}|^{2m +3}) & = \nu_{s\beta}(|R_{1,2}|^{2m +3}I(|R_{1,2}|>\epsilon))+\nu_{s\beta}(|R_{1,2}|^{2m +3}I(|R_{1,2}|<\epsilon))\\
& \leq M^{2(2m +3)} \nu_{s\beta}(I(|R_{1,2}|>\epsilon))+ \varepsilon\nu_{s\beta}(|R_{1,2}|^{2m +2}I(|R_{1,2}|<\epsilon))\\
& \leq M^{2(2m +3)}K'' e^{-N/K''} + \epsilon \nu_{s\beta}(|R_{1,2}|^{2m +2}).
\end{align*}
Plugging this into \eqref{momentcontrol:proof:eq1} leads to
\begin{align*}
\nu_{s\beta}(|R_{1,2}|^{2m +2}) & \leq \epsilon K_1(s\beta) \nu_{s\beta}(|R_{1,2}|^{2m +2})+K_1(s\beta)M^{2(2m +3)}K'' e^{-N/K''}+\frac{K_2(s\beta)}{N^{m +1}}.
\end{align*}
Thus, from  \eqref{add:eq8}, we conclude that for all $N\geq 1$ and $s\in [s_0,1]$
\[
\nu_{s\beta}(|R_{1,2}|^{2m +2}) \leq 2K_1(s\beta)M^{2(2m +3)}K'' e^{-N/K''}+\frac{2K_2(s\beta)}{N^{m +1}}.
\]
Finally, from the above two cases, our proof is completed by taking $K'$ as the supremum of this bound for $s\in [0,1].$

\subsection{Proof of Proposition \ref{fluctuations}}\label{sec7.5}

The proof of Proposition \ref{fluctuations} relies on the following bound: for any $p \geq 2, \beta > 0$, and $l > 0$,
\begin{align*}
\p \bigl(|F_N(\beta)| \geq l\bigr) \leq \frac{2\beta^2}{l^2N} \E \bigl\< |R_{1,2} |^p\bigr\>_{\beta} + \frac{2}{l^2}\Bigl(\int_0^\beta t \E \bigl\< |R_{1,2} |^p\bigr\>_{t} dt \Bigr)^2.
\end{align*}
This result is essentially taken from \cite[Lemma 10]{Chen17}. Although there the spin configurations are sampled from the uniform probability measure on the hypercube $\{-1,+1\}^N,$ the same argument applies to the current general setting. Now, from this and the moment control in Theorem \ref{momentcontrol}, there exists a constant $K> 0$ such that
\begin{align*}
\p \bigl(|F_N(\beta)| \geq l \bigr) &\leq \frac{2\beta^2K}{N^{p/2}l^2N} + \frac{2}{l^2}\Bigl(\int_0^\beta \frac{tK}{N^{p/2}}dt\Bigr)^2 = \frac{2\beta^2K}{N^{p/2+1}l^2} + \frac{\beta^4K^2}{2N^pl^2}\leq \frac{2\beta^2K}{N^{p/2+1}l^2}(1+\beta^2K).
\end{align*}
This completes our proof.

\section{Structure of the Regime $\bar {\RR}$} \label{high_temp_k>1}

This section presents the proof of Theorem \ref{thm6}. 
Recall the probability spaces $(\Lambda_r,\mu_r)$, the temperature vector $\bar \beta = (\beta_1,\dots,\beta_r)$, the Hamiltonian $H_{N,\bar \beta}$, the free energy $F_N(\bar \beta)$, the Gibbs measure $G_{N,\bar \beta}$, and the critical temperatures $\beta_{r,c}$ from Section \ref{hd}. For each $1 \leq r \leq k$, let $H_{N,\beta_r}, F_{N,r}(\beta_r)$, and $F_r(\beta_r)$ be the Hamiltonian, free energy, and limiting free energy, respectively, corresponding to the scalar-valued spin glass in Section \ref{vec_spin_model} with temperature $\beta_r$ and probability space $(\Lambda_r, \mu_r)$. Denote $v_{r,*}=\int a^2\mu_r(da)$.

\subsection{Concentration of Total Overlap}
Let $M_k(\mathbb{R})$ be the space of real-valued $k\times k$ matrices equipped with the metric $$\|V-V'\|_{\max}=\max_{1\leq r,r'\leq k}|V_{r,r'}-V_{r,r'}'|.$$
For any $\varepsilon>0$ and $V\in M_k(\mathbb{R})$, let $A_\varepsilon(V)$ be the collection of all $V'\in M_k(\mathbb{R})$ with $\|V-V'\|_{\max}<\varepsilon.$ Denote the total overlap matrix by $$\mathbf{R}(\bar \s)=\mathbf{R}(\s(1),\ldots,\s(k)):=(R(\s(r),\s(r')))_{1\leq r,r'\leq k}.$$  Set $V_*=(V_{*,r,r'})_{1\leq r,r'\leq k}\in M_k(\mathbb{R})$ where $V_{*,r,r}=v_{*,r}$ and $V_{*,r,r'}=0$ for $r\neq r'.$
For any measurable subset $A\subseteq M_k(\mathbb{R}),$ define the restricted free energy $F_N(\bar \beta,A)$ as
\begin{align*}
F_N(\bar \beta,A)=\frac{1}{N}\int_{{\bf R}(\bar \s)\in A}e^{H_{N,\bar \beta}(\bar \s)}\bar \mu^{\otimes N}(d\bar \s).
\end{align*}
For i.i.d. samplings $\bar \s,\bar \s^1,\bar \s^2,\ldots$ from $G_{N,\bar \beta}$, denote by $\la \cdot\ra_{\bar \beta}$ the Gibbs expectation with respect to these random variables. 

The following proposition states that the self-overlap of $\bar \s$ sampled from the Gibbs measure $G_{N,\bar \beta}$ is concentrated around $V_*$ in the high-temperature regime $\bar  \RR$.

\begin{prop}\label{new:lem2}
	Assume that $\bar \beta\in \bar \RR$. Let $\bar \s$ be sampled from $G_{N,{\bar \beta}}.$ For any $\varepsilon>0,$ there exist positive constants $K$ and $\delta$ such that for any $N\geq 1,$
	 with probability at least $1-Ke^{-N/K}$,
	\begin{align}\label{new:lem2:eq2}
	F_N(\bar \beta,A_\varepsilon(V_*)^c)&\leq F_N(\bar \beta)-\delta.
	\end{align}
\end{prop}

\begin{proof} We adapt a similar argument as the one for Lemma \ref{add:lem0}.
   Let $\bar \beta\in \bar {\mathcal{R}}$. For $1 \leq r,r' \leq k$, define
	\[
	A_{\epsilon}(r,r') = \bigl\{V\in M_k(\mathbb{R})\big| |V_{r,r'} - (V_*)_{r,r'}|< \epsilon\bigr\}.
	\]
	Note that
	$
	\cap_{1 \leq r,r' \leq k} A_{\epsilon}(r,r')= A_{\epsilon}(V_*) .
	$
	Using Jensen's inequality and $\e e^{H_{N,\bar \beta}(\bar \sigma)}=1$ yields
	\begin{align*}
	\limsup_{N\rightarrow\infty}\e F_N(\bar \beta,A_\varepsilon(V_*)^c)
	& \leq \limsup_{N\rightarrow\infty}\frac{1}{N}\log \sum_{1 \leq r,r' \leq k} \int_{{\bf R}(\bar \s)\in A_{\epsilon}(r,r')^c}{\bar \mu}^{\otimes N}(d\bar \s).
	\end{align*}		
	Note that the coordinates $\sigma_1(r),\ldots,\sigma_N(r)$ of $\sigma(r)$ are i.i.d. with distribution $\mu_r$, which has bounded support. Also, note that $\sigma(r)$ is independent of $\sigma(r')$ for any $r \neq r'$. In addition, the mean of $\sigma_1(r)\sigma_1(r')$ under $\bar {\mu}$ is equal to $(V_*)_{r,r'}$ for any $1\leq r,r'\leq k.$ By Cram\'{e}r's theorem, there exists a positive constant $\delta$ such that
\[
\sum_{1 \leq r,r' \leq k} \int_{{\bf R}(\bar \s) \in A_{\epsilon}(r,r')^c}\mu^{\otimes N}(d\bar \s) \leq k^2 e^{-N \delta}.
\]
Note that $F(\bar \beta)=0.$ It follows that
	\begin{align*}
	\limsup_{N\rightarrow\infty}\e F_N(\bar \beta,A_\varepsilon(V_*)^c)&\leq F(\bar \beta)-\delta.
	\end{align*}
	Finally, \eqref{new:lem2:eq2} follows by using the Gaussian concentration inequality for $F_N(\bar \beta,A_\varepsilon(V_*)^c)$ and $F_N(\bar \beta)$. We omit the details here as they are the same as those in the proof of Lemma \ref{add:lem0}.
\end{proof}

\subsection{Proof of Theorem \ref{thm6}: $\bar {\RR} \subseteq (0,\beta_{1,c}]\times\cdots\times (0,\beta_{k,c}]$ }

Suppose that $\bar \beta = (\beta_1,\dots \beta_k)\in \bar \RR$. By the definition of $\bar \RR$, $F(\bar \beta) = 0$. Let $\varepsilon>0$. On the one hand, recall that from Proposition \ref{new:lem2}, there exist two positive constants $K,\delta>0$ such that for any $N\geq 1,$ with probability at least $1-Ke^{-N/K}$,
\begin{align}\label{add:eq-12}
F_{N}(\bar \beta,A_\varepsilon(V_*)^c)\leq F_N(\bar {\beta})-\delta.
\end{align}
On the other hand, note that $V_{r,r'}=0$ for $r\neq r'$ and that ${\bf R}(\bar \s)\in A_\varepsilon(V_*)$ implies 
\begin{align*}
|R(\s(r),\s(r'))^p|=|R(\s(r),\s(r'))^p-V_{r,r',*}^p|\leq \varepsilon^p,\,\,\forall 1\leq r\neq r'\leq k.
\end{align*}
From this, we can bound the overlap terms and then release the constraint $A_\varepsilon(V_*)$ to get
\begin{align}\label{add:eq-13}
F_{N}(\bar \beta,A_\varepsilon(V_*))\leq \sum_{r=1}^k F_{N,r}(\beta_r)+\frac{ \varepsilon^p}{2}\sum_{r\neq r'} \beta_r\beta_{r'}.
\end{align}
Note that $$F_N(\bar\beta)=N^{-1}\log \bigl(e^{NF_N(\bar\beta,A_\varepsilon(V_*))}+e^{NF_N(\bar\beta,A_\varepsilon(V_*)^c)}\bigr)$$ and that $$
\log (x+y)\leq \log 2+\max(\log x,\log y),\,\,\forall x,y>0.$$
From \eqref{add:eq-12}, \eqref{add:eq-13}, and these two displays, after taking $N\to\infty,$
\begin{align*}
0=F(\bar \beta)
&\leq \max\Bigl(F(\bar \beta)-\delta,\sum_{r=1}^k F_{r}(\beta_r)+\frac{ \varepsilon^p}{2}\sum_{r\neq r'} \beta_r\beta_{r'}\Bigr)\\
&=\max\Bigl(-\delta,\sum_{r=1}^k F_{r}(\beta_r)+\frac{ \varepsilon^p}{2}\sum_{r\neq r'} \beta_r\beta_{r'}\Bigr).
\end{align*}
Consequently, we obtain
\begin{align*}
0\leq \sum_{r=1}^k F_{r}(\beta_r)+\frac{ \varepsilon^p}{2}\sum_{r\neq r'} \beta_r\beta_{r'}
\end{align*}
and letting $\varepsilon\downarrow 0$ yields $\sum_{r=1}^k F_{r}(\beta_r)\geq 0.$ Since $F_r(\beta_r) \leq 0$ for all $1 \leq r \leq k$, we must have $F_r(\beta_r) = 0$ for all $r$. Hence, $\beta_r \in (0,\beta_{r,c}]$ for all $1 \leq r \leq k$ by Proposition \ref{prop1}.	This establishes that $\bar {\mathcal{R}}\subseteq(0,\beta_{1,c}]\times\cdots(0,\beta_{k,c}].$

\subsection{Proof of Theorem \ref{thm6}: $\bar {\RR} \supseteq (0,\beta_{1,c}]\times\cdots\times (0,\beta_{k,c}]$ }

We divide our discussion into two cases.

{\noindent \bf Case 1:} $\bar \beta\in (0,\beta_{1,c})\times\cdots(0,\beta_{k,c})$, but $\mathbf{\bar \beta \not \in \bar \RR}.$ In this case, $F_r(\beta_r) = 0$ for $1 \leq r \leq k$ and $F(\bar \beta) < 0$. Then there exists a positive constant $\eta$ such that $\E F_N(\bar \beta) < -\eta$ for large enough $N$. Note that for any $\epsilon > 0$ and $N \geq 1$,
\[
F_N(\bar \beta,A_{\epsilon}( V_* )) \leq F_N(\bar \beta).
\]
Thus, from the Gaussian concentration inequality for $F_N(\bar \beta)$, there exists a constant $K>0$ such that for any large enough $N,$ with probability at least $1-Ke^{-N/K}$, 
\[
F_N(\bar \beta,A_{\epsilon}( V_* ))  \leq F_N(\bar \beta)<-\frac{\eta}{2}.
\]
Note that the off-diagonal entries of $V_*$ are all zero. The restriction $A_{\epsilon}( V_* )$ allows us to pull the off-diagonal entries of the total overlap outside of the free energy to get
\[
\frac{1}{N}  \log \int_{{\bf R}(\bar \s)\in A_{\epsilon}( V_* )} \exp \Bigl(\sum_{r=1}^k H_{N,\beta_r}(\s(r)) \Bigr) \mu^{\otimes N}(d \bar \s) < \frac{\varepsilon^p}{2}\sum_{r\neq r'} \beta_r\beta_{r'} - \frac{\eta}{2}.
\]	
Now, if we take $\epsilon>0$ with
${\varepsilon^p}\sum_{r\neq r'} \beta_r\beta_{r'} < \eta/2,
$
then the above inequality reduces to
\begin{align}
\frac{1}{N} \log \int_{{\bf R}(\bar \s )\in A_{\epsilon}( V_* )} \exp \Bigl(\sum_{r=1}^k H_{N,\beta_r}(\s(r)) \Bigr) \bar \mu^{\otimes N}(d \bar \s) < -\frac{\eta}{4}. \label{etaover2}
\end{align}
Denote by $\la \cdot\ra'$ the Gibbs average with respect to the independent samplings
\begin{align*}
\bar \s&=(\s(1),\ldots,\s(k)),\quad\bar \s^1=(\s^1(1),\ldots,\s^1(k)),\quad
\bar \s^2=(\s^2(1),\ldots,\s^2(k))
\end{align*}
from the product measure
$\prod_{r=1}^kG_{N,\beta_r}(d\s(r)).$
The combination of \eqref{etaover2}, the fact that $F_r(\beta_r)=0$,  and the Gaussian concentration inequality for $F_r(\beta_r)$ implies that the self-overlap matrix $\mathbf{R}(\bar \s )$ is concentrated around $V_*$ in the sense that there exists a constant $K'>0$ such that for sufficiently large $N,$
\begin{align}\label{ld:thm1:proof:eq0}
\e \bigl\la I\bigl(\mathbf{R}(\bar \s )\in A_\varepsilon(V_*)\bigr)\bigr\ra'\leq K'e^{-N/K'}.
\end{align}

Next, in order to deduce a contradiction, we recall that Lemma \ref{add:lem0} states
\begin{align}
\begin{split}
\label{ld:thm1:proof:eq1}
&\lim_{N\rightarrow\infty}\e\bigl\la I\bigl(|R(\s(r),\s(r))-v_{r,*}|\leq \delta\bigr)\bigr\ra'=1.
\end{split}
\end{align}
Here we used the fact that for a sampling $\bar \s$ from $\la \cdot\ra'$, the components $\s(1),\ldots,\s(k)$ are independent of each other. For the same reason, it also follows from Proposition \ref{new:lem0} and the assumption $\bar \beta\in (0,\beta_{1,c})\times\cdots\times (0,\beta_{k,c})$ that for any $\delta>0$ and $1\leq r\leq k,$
\begin{align}
\begin{split}\label{ld:thm1:proof:eq2}
&\lim_{N\rightarrow\infty}\e\bigl\la I\bigl(|R(\s^1(r),\s^2(r))|\leq \delta\bigr)\bigr\ra'=1.
\end{split}
\end{align}
For $r\neq r',$ observe that
\begin{align*}
\e \bigl\la R(\s(r),\s(r'))^2 \bigr\ra'&=\frac{1}{N^2}\sum_{i,j=1}^N\e\bigl\la \s_i(r)\s_j(r)\big\ra'\bigl\la\s_i(r')\s_j(r')\bigr\ra',
\end{align*}
where the second equality holds since $\s_i(r),\s_j(r)$ are independent of $\s_i(r'),\s_j(r')$ under $\la \cdot\ra'.$ Consequently, an application of the Cauchy-Schwarz inequality implies that
\begin{align*}
\e \bigl\la R(\s(r),\s(r'))^2 \bigr\ra'&\leq \Bigl(\frac{1}{N^2}\sum_{i,j=1}^N\e\bigl(\bigl\la \s_i(r)\s_j(r)\big\ra'\bigr)^2\Bigr)^{1/2}\Bigl(\frac{1}{N^2}\sum_{i,j=1}^N\e\bigl(\bigl\la \s_i(r')\s_j(r')\big\ra'\bigr)^2\Bigr)^{1/2}\\
&=\Bigl(\e\bigl\la R(\s^1(r),\s^2(r))^2\bigr\ra'\Bigr)^{1/2}\Bigl(\e\bigl\la R(\s^1(r'),\s^2(r'))^2\bigr\ra'\Bigr)^{1/2},
\end{align*}
where the last equality uses the identity
\begin{align*}
\frac{1}{N^2}\sum_{i,j=1}^N\e\bigl(\bigl\la \s_i(r)\s_j(r)\big\ra'\bigr)^2&=\frac{1}{N^2}\sum_{i,j=1}^N\e\bigl\la \s_i^1(r)\s_j^1(r)\s_i^2(r)\s_j^2(r)\big\ra'=\e \bigl\la R(\s^1(r),\s^2(r))^2 \bigr\ra',
\end{align*}
which also holds when $r'$ replaces $r$. From the above inequality and \eqref{ld:thm1:proof:eq2},
\begin{align*}
\lim_{N\rightarrow\infty}\e \bigl\la R(\s(r),\s(r'))^2 \bigr\ra'=0,
\end{align*}
which means that $R(\s(r),\s(r'))$ is essentially concentrated at $0$ under $\la \cdot\ra'.$ Combining the latter observation, the inequality $\eqref{ld:thm1:proof:eq1}$, and the fact that the off-diagonal entries of $V^*$ are all zero yields
\begin{align*}
\lim_{N\rightarrow\infty}\e \bigl\la I\bigl({\bf R}(\bar \s )\in A_\varepsilon(V_*)\bigr)\bigr\ra'=1.
\end{align*}
However, this contradicts \eqref{ld:thm1:proof:eq0}. Thus, we must have $\bar \beta\in \mathcal{R}$. 

{\noindent \bf Case 2:} $\bar \beta = (\beta_1, \dots, \beta_k)\in(0,\beta_{1,c}]\times\cdots\times(0,\beta_{k,c}]$, but $\bar \beta = (\beta_1, \dots, \beta_k) \not \in (0,\beta_{1,c})\times\cdots\times(0,\beta_{k,c}).$ Note that the free energies $F,F_1,\ldots,F_k$ are continuous functions of the temperature parameters. We can approximate $F(\bar \beta)$ by $F(\bar \beta')$ for $\bar \beta'\in(0,\beta_{1,c})\times\cdots\times(0,\beta_{k,c})$. From this and Case $1$, we see that $F(\bar \beta)=0$ and so $\bar \beta\in \bar {\mathcal{R}}$. 

\appendix

\section{Proof of the Guerra-Talagrand Bound}

Recall $X_N$ from Section \ref{vec_spin_model} and recall $\mathcal{V}$, $\mathcal{M}_{v,v_0}$, and $\mathcal{P}_{\beta,V}$ from Section \ref{sec6.1}. The goal of this appendix is to give a sketch of the proof for the Guerra-Talagrand inequality stated in \eqref{GTB}.
We follow the same argument in \cite[Proposition 2]{Chen17}.

{\noindent \bf Step 1:} Fix $v_0$ and assume that $\alpha\in \mathcal{M}_{v,v_0}$ is of the form $\alpha(s)=m1_{[0,v_0)}(s)+1_{[v_0,v]}(s)$ for some $m\in [0,1].$ Let $m_0=0<m_1=m<m_2<m_3=1.$ Let $(c_\tau)_{\tau\in \mathbb{N}^2}$ be the Ruelle probability cascades associated to $0<m_1<m_2<1$, see \cite[Subsection 14.1]{Talbook2}. Set 
\begin{align*}
\rho_0^{11}&=0,\rho_1^{11}=v_0,\rho_2^{11}=v,\\
\rho_0^{12}&=0,\rho_1^{12}=v_0,\rho_2^{12}=v_0,\\
\rho_0^{21}&=0,\rho_1^{21}=v_0,\rho_2^{21}=v_0,\\
\rho_0^{22}&=0,\rho_1^{11}=v_0,\rho_2^{22}=v.
\end{align*} 
Assume that $(z_1^1,z_1^2)$ and $(z_2^1,z_2^2)$ are two independent Gaussian random vectors with mean zero and covariance,
$$
\e z_a^{\ell}z_a^{\ell'}=\xi'(\rho_a^{\ell\ell'})-\xi'(\rho_{a-1}^{\ell\ell'})
$$
for $a=1,2$ and $\ell,\ell'=1,2.$ Let $(z_{i,1,j_1}^1,z_{i,1,j_1}^2)_{j_1\in \mathbb{N}}$ for $1\leq i\leq N$ be i.i.d. copies of $(z_1^1,z_1^2)$ and $(z_{i,2,j_1,j_2}^1,z_{i,2,j_1,j_2}^2)_{(j_1,j_2)\in \mathbb{N}^2}$ for $1\leq i\leq N$ be i.i.d. copies of $(z_2^1,z_2^2).$ These are also independent of each other. For $0\leq t\leq 1,$ consider the interpolating Hamiltonian
$$
X_{N,t}(\sigma^1,\sigma^2,\tau):=\sqrt{t}(X_N(\sigma^1)+X_N(\sigma^2))+\sqrt{1-t}\sum_{1\leq i\leq N}\sum_{\ell=1,2}\sigma_i^\ell\bigl(z_{i,1,j}+z_{i,2,\tau_1,\tau_2}\bigr)
$$
for $\sigma^1,\sigma^2\in \Lambda^N$ and $\tau=(\tau_1,\tau_2)\in \mathbb{N}^2.$ Define the interpolating free energy
$$
\phi_{N,\eta}(t)=\frac{1}{N}\e \log\sum_{\tau\in \mathbb{N}^2}c_{\tau}\int_{{\mathbf{R}(\sigma^1,\sigma^2)}\in A_\eta(V)} e^{\beta X_{N,t}(\sigma^1,\sigma^2,\tau)}\mu^{\otimes N}(d\sigma^1)\mu^{\otimes N}(d\sigma^2).
$$
Note that $$
\phi_{N,\eta}(1)=\frac{1}{N}\e\log \int_{\mathbf{R}(\s^1,\s^2)\in A_\eta(V)}e^{\beta X_{N}(\s^1)+\beta X_{N}(\s^2)}\mu^{\otimes N}(d\s^1)\mu^{\otimes N}(d\s^2).
$$
and $\phi_{N,\eta}(0)$ involves only linear spin interactions.

{\noindent \bf Step 2:} We proceed to compute the derivative of $\phi_{N,\eta}(t).$ Consider the Gibbs measure 
$$
G_{N,\eta,t}(d\sigma^1,d\sigma^2,\tau):=\frac{1}{Z_{N,\eta,t}}c_{\tau}e^{\beta X_{N,t}(\sigma^1,\sigma^2,\tau)}\mu^{\otimes N}(d\sigma^1)\mu^{\otimes N}(d\sigma^2)
$$
for $(\sigma^1,\sigma^2,\tau)\in \Lambda^N\times\Lambda^N\times\mathbb{N}^2$ satisfying that $\mathbf{R}(\sigma^1,\sigma^2)\in A_\eta(V),$ where $Z_{N,\eta,t}$ is the normalizing constant. 
Let $\la \cdot\ra_{N,\eta,t}$ be the Gibbs expectation associated to this free energy. Denote by $(\s^1,\s^2,\tau)$ and $(\hat{\sigma}^{1},\hat{\sigma}^{2},\hat\tau)$ two independent samplings from $G_{N,\eta,t}$. We set the overlaps between these two pairs by $Q^{\ell\ell'}=R(\sigma^\ell,{\hat\sigma}^{\ell'})$ for $1\leq \ell,\ell'\leq 2$ and $$
\tau\wedge\hat\tau=\left\{
\begin{array}{ll}
0,&\mbox{if $\tau_1\neq \tau_2$},\\
1,&\mbox{if $\tau_1=\hat{\tau}_1,\tau_2\neq \hat{\tau}_2$},\\
2,&\mbox{if $\tau_1=\hat{\tau}_1,\tau_2= \hat{\tau}_2$}.
\end{array}\right.
$$
From the same computation in \cite[Chapter 15]{Talbook2} that uses Gaussian integration by parts, 
\begin{align*}
\phi_{N,\eta}'(t)&=-\Pi_{N,\eta}(t)-E_{N,\eta}(t)+O(\eta)
\end{align*}
for
\begin{align*}
\Pi_{N,\eta}(t)&:=\frac{\beta^2}{2}\sum_{\ell,\ell'=1,2}\Bigl(\theta(\rho_{2}^{\ell \ell '})-\e\bigl\la \theta(\rho_{\tau\wedge\hat\tau }^{\ell  \ell  '})\bigr\ra_{N,\eta,t}\Bigr),\\
E_{N,\eta}(t)&:=\frac{\beta^2}{2}\sum_{\ell  ,\ell  '=1,2}\e\bigl\la \Gamma\bigl(\rho_{\tau\wedge\hat\tau  }^{\ell  \ell  '},Q^{\ell \ell '}\bigr)\bigr\ra_{N,\eta,t},
\end{align*}
where $O(\eta)$ means that it uniformly vanishes as $\eta\downarrow 0$, $\theta(x):=x\xi'(x)-\xi(x)$, and $\Gamma(x,y):=\xi(y)-y\xi'(x)+\theta(x)$.

{\noindent \bf Step 3:} We handle $\Pi_{N,\eta}(t)$ and $E_{N,\eta}(t)$ as follows. First, using the fact (see \cite[Section 14.1]{Talbook2}) that 
\begin{align}\label{add:eq-3}
\e\la I(\tau\wedge \hat\tau=a)\ra=m_{a+1}-m_a,\,\,0\leq a\leq 2
\end{align}
leads to
\begin{align*}
\Pi_{N,\eta}(t)&=\beta^2\bigl(2m\bigl(\theta(v_0)-\theta(0)\bigr)+m_2\bigl(\theta(v)-\theta(v_0)\bigr)\bigr).
\end{align*}
The treatment for $E_{N,\eta}(t)$ is the harder part. If $p$ is an even number, we obviously have
\begin{align}\label{add:eq-1}
\Gamma(x,y)\geq 0,\,\,\forall x,y\in \mathbb{R},
\end{align}
which implies $E_{N,\eta}(t)\geq 0.$ If $p$ is odd, \eqref{add:eq-1} is no longer valid and we do not have an obvious sign for $E_{N,\eta}(t).$ 
The idea to overcome this difficulty is to add an asymptotically vanishing perturbation to the interpolating Hamiltonian such that the entries in the overlap matrix $(Q^{\ell\ell'})_{1\leq \ell,\ell'\leq 2}$ are synchronized in the limit, see \cite{PanMultipSP}. A crucial fact here is that this property will force the overlap matrix to be asymptotically positive semi-definite under $\e \la \cdot\ra_{N,\eta,t}$ for all $0\leq t\leq 1$ as $N\to\infty$ and $\eta\downarrow 0$, see \cite{Panchenko2015}. For the precise choice of this perturbation, we refer the reader to the proof of \cite[Proposition 2]{Chen17}. 

For clarity, we adapt the same notation for the free energy and the Gibbs expectation. Note that the procedure of adding an asymptotically vanishing perturbation will also not affect the value of the coupled free energy when $N\to\infty$ and $\eta\downarrow 0$ and the derivative of the corresponding $\phi_{N,\eta}(t)$ still has the same form. Now write
\begin{align}
\begin{split}\label{add:eq-2}
E_{N,\eta}(t)&=\frac{\beta^2}{2}\e\Bigl\la I(\tau\wedge \hat\tau=0)\sum_{\ell  ,\ell  '=1,2}\Gamma\bigl(Q^{\ell \ell '},0\bigr)\Bigr\ra_{N,\eta,t}\\
&+\frac{\beta^2}{2}\e\Bigl\la I(\tau\wedge \hat\tau=1)\sum_{\ell  ,\ell  '=1,2}\Gamma\bigl(Q^{\ell \ell '},v_0\bigr)\Bigr\ra_{N,\eta,t}\\
&+\frac{\beta^2}{2}\e\Bigl\la I(\tau\wedge \hat\tau=2)\sum_{\ell  ,\ell  '=1,2}\Gamma\bigl(Q^{\ell \ell '},\rho_{2 }^{\ell  \ell  '}\bigr)\Bigr\ra_{N,\eta,t}.
\end{split}
\end{align}
To hand this, we recall a lemma from \cite[Lemma 11]{Chen17}\footnote{Although the statement there also requires $s\leq 1$ and $x,y,z\in [-1,1]$, they are not actually needed in the proof. Hence the result is still valid under the present assumption.} that for any $s\geq 0$ and $x,y,z\in \mathbb{R}$ so that $$
\left[
\begin{array}{cc}
x&z\\
z&y
\end{array}\right]
$$
is positive semi-definite, we have 
$$
\Gamma(s,x)+\Gamma(s,y)+2\Gamma(s,z)\geq 0.
$$
By using the positive semi-definiteness of the overlap matrix, this lemma implies that the first and second terms on the right-hand side of \eqref{add:eq-2} are asymptotically nonnegative.
Also, from \eqref{add:eq-3}, the third term of \eqref{add:eq-2} is upper bounded by $C(1-m_2)$. Hence, when $p$ is odd, we have 
\begin{align*}
\liminf_{\eta\downarrow 0}\liminf_{N\to\infty}E_{N,\eta}(t)&\geq -C(1-m_2).
\end{align*}
To summarize, from the above discussion, no matter if $p$ is even or odd,
\begin{align}
\begin{split}\label{add:eq-4}
&\lim_{\eta\downarrow 0}\limsup_{N\to\infty}\phi_{N,\eta}(1)\\
&\leq \lim_{\eta\downarrow 0}\limsup_{N\to\infty}\phi_{N,\eta}(0)+\limsup_{\eta\downarrow 0}\limsup_{N\to\infty}\int_0^1\phi_{N,\eta}'(t)dt\\
&\leq\lim_{\eta\downarrow 0}\limsup_{N\to\infty}\phi_{N,\eta}(0) -\beta^2\bigl(2m\bigl(\theta(v_0)-\theta(0)\bigr)+m_2\bigl(\theta(v)-\theta(v_0)\bigr)\bigr)+C(1-m_2).
\end{split}
\end{align}

{\noindent \bf Step 4:} Our last step is to rearrange the right-hand side of the last inequality. Note that we can release the constraint ${\bf R}(\sigma^1,\sigma^2)\in A_\eta(V)$ in $\phi_{N,\eta}(0)$ by introducing a Lagrange variable $\lambda\in M_2(\mathbb{R})$ so that
\begin{align*}
\phi_{N,\eta}(0)&\leq \frac{1}{N}\e\log \sum_{\tau\in \mathbb{N}^2}c_\tau\int e^{\beta X_{N,0}(\sigma^1,\sigma^2,\tau)+N\la\lambda, {\bf R}(\sigma^1,\sigma^2)\ra} \mu^{\otimes N}(d\sigma^1)\mu^{\otimes N}(d\sigma^2)-\la\lambda,V\ra+\eta.
\end{align*}
If we denote
\begin{align*}
w_{i,\tau}^1&=z_{i,1,\tau_1}^1+z_{i,2,\tau_1,\tau_2}^1,\,\,w_{i,\tau}^2=z_{i,1,\tau_1}^2+z_{i,2,\tau_1,\tau_2}^2,\,\,
w^1=z_1^1+z_2^1,\,\,w^2=z_1^2+z_2^2,
\end{align*}
then the first term on the right-hand side of the last inequality can be written as
\begin{align*}
&\frac{1}{N}\e\log\sum_{\tau\in \mathbb{N}^2}c_\tau \prod_{i=1}^N\int \exp \beta\Bigl(w_{i,\tau}^1\s_i^1+w_{i,\tau}^2\sigma_i^2+\sum_{\ell,\ell'=1,2}\lambda_{\ell\ell'}\s_i^\ell \s_i^{\ell'}\Bigr)\mu(d\s_i^1)\mu(d\s_i^2)\\
&=\frac{1}{m_1}\log \e_1\exp \frac{m_1}{m_2}\log \e_2\Bigl[\exp  \Bigl( m_2\log \int e^{\beta w^1a^1+\beta w^2a^2+\la\lambda a,a\ra}\mu(da^1)\mu(da^2)\Bigr)\Bigr],
\end{align*}
where $\e_1$ is the expectation with respect to $(z_1^1,z_1^2)$ only and $\e_2$ is the expectation with respect to $(z_2^1,z_2^2)$ only. Here, this equality is valid by using \cite[Theorem 14.2.1]{Talbook2}. From this, after sending $m_2\to 1,$ it can be checked directly by the Cole-Hopf transformation that the last equation is indeed equal to $\Psi_{\beta,V,\alpha}(0,0,\lambda).$ On the other hand, the second and third terms in the last line of \eqref{add:eq-4} together equal
$$
\beta^2\Bigl(\int_0^v\xi''(s)s\alpha(s)ds+\int_0^{v_0}\xi''(s)s\alpha(ds)\Bigr).
$$
These and \eqref{add:eq-4} complete our proof.

\bibliographystyle{plain}
{\footnotesize\bibliography{ref}}

\begin{thebibliography}{10}

\bibitem{ALR}
M.~Aizenman, J.~L. Lebowitz, and D.~Ruelle.
\newblock Some rigorous results on the {S}herrington-{K}irkpatrick spin glass
  model.
\newblock {\em Comm. Math. Phys.}, 112(1):3--20, 1987.

\bibitem{AC14}
A.~Auffinger and W.-K. Chen.
\newblock The {P}arisi formula has a unique minimizer.
\newblock {\em Comm. Math. Phys.}, 335(3):1429--1444, 2015.

\bibitem{BBP}
J.~Baik, G.~Ben~Arous, and S.~P\'ech\'e.
\newblock Phase transition of the largest eigenvalue for nonnull complex sample
  covariance matrices.
\newblock {\em Ann. Probab.}, 33(5):1643--1697, 2005.

\bibitem{baik2}
J.~Baik and J.~Silverstein.
\newblock Eigenvalues of large sample covariance matrices of spiked population
  models.
\newblock {\em J. Multivariate Anal.}, 97(6):1382--1408, 2006.

\bibitem{BDMK+16}
J.~Barbier, M.~Dia, N.~Macris, and F.~Krzakala.
\newblock The mutual information in random linear estimation.
\newblock {\em 54th Annual Allerton Conference on Communication, Control, and
  Computing}, pages 625--632, 2016.

\bibitem{BDM+16}
J.~Barbier, M.~Dia, N.~Macris, F.~Krzakala, T.~Lesieur, and L.~Zdeborov\'{a}.
\newblock Mutual information for symmetric rank-one matrix estimation: A proof
  of the replica formula.
\newblock {\em Advances in Neural Information Processing Systems}, 29:424--432,
  2016.

\bibitem{BKMMZ+17}
J.~Barbier, F.~Krzakala, N.~Macris, L.~Miolane, and L.~Zdeborov\'{a}.
\newblock Phase transitions, optimal errors and optimality of message-passing
  in generalized linear models.
\newblock {\em ArXiv e-prints}, 2017.

\bibitem{BM+17}
J.~Barbier and N.~Macris.
\newblock The stochastic interpolation method: A simple scheme to prove replica
  formulas in bayesian inference.
\newblock {\em ArXiv e-prints}, 2017.

\bibitem{BMDK+17}
J.~Barbier, N.~Macris, M.~Dia, and F.~Krzakala.
\newblock Mutual information and optimality of approximate message-passing in
  random linear estimation.
\newblock {\em ArXiv e-prints}, 2017.

\bibitem{BMM+17}
J.~Barbier, N.~Macris, and L.~Miolane.
\newblock The layered structure of tensor estimation and its mutual
  information.
\newblock {\em 55th Annual Allerton Conference on Communication, Control, and
  Computing}, 2017.

\bibitem{Tindel}
X.~Bardina, D.~M\'arquez-Carreras, C.~Rovira, and S.~Tindel.
\newblock The {$p$}-spin interaction model with external field.
\newblock {\em Potential Anal.}, 21(4):311--362, 2004.

\bibitem{BM+11}
M.~Bayati and A.~Montanari.
\newblock The dynamics of message passing on dense graphs, with applications to
  compressed sensing.
\newblock {\em IEEE Trans. Inform. Theory}, 57(2):764--785, 2011.

\bibitem{BGJ+18}
G.~Ben~Arous, R.~Gheissari, and A.~Jagannath.
\newblock Algorithmic thresholds for tensor pca.
\newblock {\em ArXiv e-prints}, 2017.

\bibitem{BMMN+17}
G.~Ben~Arous, S.~Mei, A.~Montanari, and M.~Nica.
\newblock The landscape of the spiked tensor model.
\newblock {\em ArXiv e-prints}, 2017.

\bibitem{benaych2011}
F.~Benaych-Georges and R.~R. Nadakuditi.
\newblock The eigenvalues and eigenvectors of finite, low rank perturbations of
  large random matrices.
\newblock {\em Advances in Mathematics}, 227(1):494--521, 2011.

\bibitem{benaych2012}
F.~Benaych-Georges and R.~R. Nadakuditi.
\newblock The singular values and vectors of low rank perturbations of large
  rectangular random matrices.
\newblock {\em Journal of Multivariate Analysis}, 111:120--135, 2012.

\bibitem{Bol18}
E.~Bolthausen.
\newblock Search results web results a morita type proof of the
  replica-symmetric formula for sk.
\newblock {\em ArXiv e-prints}, 2018.

\bibitem{BKL}
A.~Bovier, I.~Kurkova, and M.~L\"owe.
\newblock Fluctuations of the free energy in the {REM} and the {$p$}-spin {SK}
  models.
\newblock {\em Ann. Probab.}, 30(2):605--651, 2002.

\bibitem{C14}
W.-K. Chen.
\newblock Variational representations for the {P}arisi functional and the
  two-dimensional {G}uerra-{T}alagrand bound.
\newblock {\em Ann. Probab.}, 45(6A):3929--3966, 2017.

\bibitem{Chen17}
W.-K. Chen.
\newblock Phase transition in the spiked random tensor with {R}ademacher prior.
\newblock {\em Ann. Statist.}, 47(5):2734--2756, 2019.

\bibitem{dembo_zeit}
A.~Dembo and O.~Zeitouni.
\newblock {\em Large deviations techniques and applications}, volume~38 of {\em
  Stochastic Modelling and Applied Probability}.
\newblock Springer-Verlag, Berlin, 2010.
\newblock Corrected reprint of the second (1998) edition.

\bibitem{deshpande2016}
Y.~Deshpande, E.~Abbe, and A.~Montanari.
\newblock Asymptotic mutual information for the balanced binary stochastic
  block model.
\newblock {\em Information and Inference: A Journal of the IMA}, 6(2):125--170,
  2016.

\bibitem{DM+14}
Y.~Deshpande and A.~Montanari.
\newblock Information-theoretically optimal sparse {PCA}.
\newblock {\em IEEE Internation Symposium on Information Theory}, pages
  2197--2201, 2014.

\bibitem{DMM+09}
D.~L. Donoho, A.~Maleki, and A.~Montanari.
\newblock Message-passing algorithms for compressed sensing.
\newblock {\em Proceedings of the National Academy of Sciences},
  106(45):18914--18919, 2009.

\bibitem{el2018detection}
A.~El~Alaoui and M.~I. Jordan.
\newblock Detection limits in the high-dimensional spiked rectangular model.
\newblock In {\em Conference On Learning Theory}, pages 410--438, 2018.

\bibitem{AKJ+17}
A.~El~Alaoui, F.~Krzakala, and M.~Jordan.
\newblock Finite-size corrections and likelihood ratio fluctuations in the
  spiked {W}igner model.
\newblock {\em ArXiv e-prints}, 2017.

\bibitem{AKJ20}
A.~El~Alaoui, F.~Krzakala, and M.~Jordan.
\newblock Fundamental limits of detection in the spiked {W}igner model.
\newblock {\em Ann. Statist.}, 48(2):863--885, 2020.

\bibitem{PecheD}
D.~F\'{e}ral and S.~P\'{e}ch\'{e}.
\newblock The largest eigenvalue of rank one deformation of large {W}igner
  matrices.
\newblock {\em Comm. Math. Phys.}, 272(1):185--228, 2007.

\bibitem{GWSV}
D.~Guo, Y.~Wu, S.~S. Shitz, and S.~Verd\'u.
\newblock Estimation in {G}aussian noise: Properties of the minimum mean-square
  error.
\newblock {\em IEEE Transactions and Information Theory}, 54(4):2371--2385,
  2011.

\bibitem{HillarL13}
C.~J. Hillar and L.{-}H. Lim.
\newblock Most tensor problems are {NP}-hard.
\newblock {\em Journal of the ACM}, 60(6):45:1--45:39, 2013.

\bibitem{hopkins2016fast}
S.~B. Hopkins, T.~Schramm, J.~Shi, and D.~Steurer.
\newblock Fast spectral algorithms from sum-of-squares proofs: tensor
  decomposition and planted sparse vectors.
\newblock In {\em Proceedings of the forty-eighth annual ACM symposium on
  Theory of Computing}, pages 178--191. ACM, 2016.

\bibitem{hopkins2015tensor}
S.~B. Hopkins, J.~Shi, and D.~Steurer.
\newblock Tensor principal component analysis via sum-of-square proofs.
\newblock In {\em Conference on Learning Theory}, pages 956--1006, 2015.

\bibitem{JT2}
A.~Jagannath and I.~Tobasco.
\newblock A dynamic programming approach to the {P}arisi functional.
\newblock {\em Proc. Amer. Math. Soc.}, 144(7):3135--3150, 2016.

\bibitem{JM+13}
A.~Javanmard and A.~Montanari.
\newblock State evolution for general approximate message passing algorithms,
  with applications to spatial coupling.
\newblock {\em Inf. Inference}, 2(2):115--144, 2013.

\bibitem{Johnstone}
I.~Johnstone.
\newblock On the distribution of the largest eigenvalue in principal components
  analysis.
\newblock {\em Ann. Statist.}, 29(2):295--327, 2001.

\bibitem{kolda2009tensor}
T.~G. Kolda and B.~W. Bader.
\newblock Tensor decompositions and applications.
\newblock {\em SIAM Review}, 51(3):455--500, 2009.

\bibitem{LM+16}
M.~Lelarge and L.~Miolane.
\newblock Fundamental limits of symmetric low-rank matrix estimation.
\newblock {\em Proceedings of the 2017 Conference on Learning Theory}, PMLR
  65:1297--1301, 2017.

\bibitem{LMLKZ+17}
T.~Lesieur, L.~Miolane, M.~Lelarge, F.~Krzakala, and L.~Zdeborov\'{a}.
\newblock Statistical and computational phase transitions in spiked tensor
  estimation.
\newblock {\em IEEE 2017 International Symposium}, pages 511--515, 2017.

\bibitem{Mar_Pas}
V.~A Marchenko and L.~A. Pastur.
\newblock Distribution of eigenvalues for some sets of random matrices.
\newblock {\em Mathematics of the USSR-Sbornik}, 1(4):457, 1967.

\bibitem{Montanari_spin_book2012}
M.~M\'ezard and A.~Montanari.
\newblock {\em Information, physics, and computation}.
\newblock Oxford graduate texts. Oxford University Press, Oxford, 2009.
\newblock Autre tirage : 2010, 2012.

\bibitem{MPV}
M.~M{\'e}zard, G.~Parisi, and M.~A. Virasoro.
\newblock {\em Spin glass theory and beyond}, volume~9 of {\em World Scientific
  Lecture Notes in Physics}.
\newblock World Scientific Publishing Co., Inc., Teaneck, NJ, 1987.

\bibitem{M+17}
L.~Miolane.
\newblock Fundamental limits of symmetric low-rank matrix estimation: the
  non-symmetric case.
\newblock {\em ArXiv e-prints}, 2017.

\bibitem{miolane2018}
L.~Miolane.
\newblock Phase transitions in spiked matrix estimation: information-theoretic
  analysis.
\newblock {\em arXiv preprint arXiv:1806.04343}, 2018.

\bibitem{MRZ17}
A.~Montanari, D.~Reichman, and O.~Zeitouni.
\newblock On the limitation of spectral methods: from the {G}aussian hidden
  clique problem to rank one perturbations of {G}aussian tensors.
\newblock {\em IEEE Trans. Inform. Theory}, 63(3):1572--1579, 2017.

\bibitem{MR+14}
A.~Montanari and E.~Richard.
\newblock A statistical model for tensor {PCA}.
\newblock {\em Neural Information Processing Systems}, pages 2897--2905, 2014.

\bibitem{MR+16}
A.~Montanari and E.~Richard.
\newblock Non-negative principal component analysis: message passing algorithms
  and sharp asymptotics.
\newblock {\em IEEE Trans. Inform. Theory}, 62(3):1458--1484, 2016.

\bibitem{Montanar_Sen_STOC2016}
A.~Montanari and S.~Sen.
\newblock Semidefinite programs on sparse random graphs and their application
  to community detection.
\newblock In {\em Proceedings of the 48th Annual ACM SIGACT Symposium on Theory
  of Computing}, STOC 2016, pages 814--827, New York, NY, USA, 2016. ACM.

\bibitem{OMH13}
A.~Onatski, M.~J. Moreira, and M.~Hallin.
\newblock Asymptotic power of sphericity tests for high-dimensional data.
\newblock {\em Ann. Statist.}, 41(3):1204--1231, 2013.

\bibitem{Panchenko05}
D.~Panchenko.
\newblock Free energy in the generalized {Sherrington-Kirkpatrick} mean field
  model.
\newblock {\em Rev. Math. Phys.}, 17(7):793--857, 2005.

\bibitem{Pan}
D.~Panchenko.
\newblock {\em The {S}herrington-{K}irkpatrick model}.
\newblock Springer Monographs in Mathematics. Springer, New York, 2013.

\bibitem{PanMultipSP}
D.~Panchenko.
\newblock The free energy in a multi-species {S}herrington-{K}irkpatrick model.
\newblock {\em Ann. Probab.}, 43(6):3494--3513, 2015.

\bibitem{Panchenko2015}
D.~{Panchenko}.
\newblock Free energy in the mixed $p$-spin models with vector spins.
\newblock {\em Ann. Probab.}, 46(2):865--896, 03 2018.

\bibitem{paul2007}
D.~Paul.
\newblock Asymptotics of the leading sample eigenvalues for a spiked covariance
  model.
\newblock {\em Statistica Sinica}, 17(4):1617--1642, 2007.

\bibitem{Peche}
S.~P\'{e}ch\'{e}.
\newblock The largest eigenvalue of small rank perturbations of {H}ermitian
  random matrices.
\newblock {\em Probab. Theory Related Fields}, 134(1):127--173, 2006.

\bibitem{PWB17}
A.~Perry, A.~S. Wein, and A.~Bandeira.
\newblock Statistical limits of spiked tensor models.
\newblock {\em ArXiv e-prints}, 2017.

\bibitem{PWBM16}
A.~Perry, A.~S. Wein, A.~Bandeira, and A.~Moitra.
\newblock Optimality and sub-optimality of \mbox{PCA} for spiked random
  matrices and synchronization.
\newblock {\em ArXiv e-prints}, 2016.

\bibitem{RP+16}
G.~Reeves and H.~D. Pfister.
\newblock The replica-symmetric prediction for compressed sensing with
  {G}aussian matrices is exact.
\newblock {\em ArXiv e-prints}, 2016.

\bibitem{RBBC+18}
V.~Ros, G.~Ben~Arous, G.~Biroli, and C.~Cammarota.
\newblock Complex energy landscapes in spiked-tensor and simple glassy models:
  Ruggedness, arrangements of local minima and phase transitions.
\newblock {\em ArXiv e-prints}, 2018.

\bibitem{Tal03}
M.~Talagrand.
\newblock The {P}arisi formula.
\newblock {\em Ann. of Math. (2)}, 163(1):221--263, 2006.

\bibitem{Talbook1}
M.~Talagrand.
\newblock {\em Mean field models for spin glasses. Basic examples}, volume~54
  of {\em Ergebnisse der Mathematik und ihrer Grenzgebiete. 3. Folge. A Series
  of Modern Surveys in Mathematics [Results in Mathematics and Related Areas.
  3rd Series. A Series of Modern Surveys in Mathematics]}.
\newblock Springer-Verlag, Berlin, 2010.

\bibitem{Talbook2}
M.~Talagrand.
\newblock {\em Mean field models for spin glasses. Advanced replica-symmetry
  and low temperature}, volume~55 of {\em Ergebnisse der Mathematik und ihrer
  Grenzgebiete. 3. Folge. A Series of Modern Surveys in Mathematics [Results in
  Mathematics and Related Areas. 3rd Series. A Series of Modern Surveys in
  Mathematics]}.
\newblock Springer, Heidelberg, 2011.

\bibitem{WEM19}
A.~S. Wein, A.~El~Alaoui, and C.~Moore.
\newblock {The Kikuchi Hierarchy and Tensor PCA}.
\newblock {\em ArXiv e-prints}, 2019.

\bibitem{WV}
Y.~Wu and S.~Verd\'u.
\newblock Functional properties of minimum mean-square error and mutual
  information.
\newblock {\em IEEE Transactions and Information Theory}, 58(3):1289--1301,
  2012.

\end{thebibliography}

\end{document}